\documentclass[aos]{imsart}

\RequirePackage{amsthm,amsmath,amsfonts,amssymb}
\RequirePackage[authoryear]{natbib}
\RequirePackage[colorlinks,citecolor=blue,urlcolor=blue]{hyperref}
\RequirePackage{graphicx}
\RequirePackage{epsfig, color}

\RequirePackage{latexsym}
\RequirePackage{psfrag}
\RequirePackage{color}
\usepackage{natbib}
\usepackage{xcolor}
\usepackage{epsf,epsfig}
\usepackage{epsf,epsfig}
\usepackage{graphics,graphicx}
\usepackage{enumerate}
\usepackage[linesnumbered,ruled,lined]{algorithm2e}
\usepackage{algpseudocode}
\usepackage{bbm}

\usepackage{subfigure}

\usepackage{xr}	

\startlocaldefs
\theoremstyle{plain}
\newtheorem{theorem}{Theorem}

\newtheorem{lemma}{Lemma}

\newtheorem{corollary}[theorem]{Corollary}
\newtheorem{proposition}{Proposition}
\theoremstyle{remark}
\newtheorem{definition}[theorem]{Definition}

\newtheorem{assumption}{Assumption}

\newcommand{\bibdir}{.}

\newcommand{\bM}{\bs{M}}
\newcommand{\N}{\mathbb{N}}
\newcommand{\cN}{\mathcal{N}}
\newcommand{\cD}{\mathcal{D}}
\newcommand{\R}{\mathbb{R}}

\newcommand{\E}{\mathbb{E}}
\newcommand{\bE}{\mathbb{E}}
\newcommand{\argmin}{{\arg\hspace*{-0.5mm}\min}}
\newcommand{\matspan}{\mathrm{span}}

\newcommand{\one}{\mathbbm{1}}

\def\op{\mathrm{op}}
\def\bs{\boldsymbol}
\def\st{\mathrm{s.t.}}

\def\pr{\mathbb{P}}

\def\cN{\mathcal{N}}

\def\argmin{{\arg\hspace*{-0.5mm}\min}}
\def\prob{\mathbb{P}}

\def\pval{\phi}
\def\bv{\bs{\varepsilon}}
\def\bw{\bs{w}}

\def\bP{\bs{P}}
\def\bV{\bs{V}}
\def\bZ{\bs{Z}}
\def\bY{\bs{Y}}
\def\bX{\bs{X}}
\def\be{\bs{e}}
\def\bU{\bs{U}}
\def\cP{\mathcal{P}}
\def\proj{\mathrm{Proj}}
\def\tr{\mathrm{tr}}
\def\anova{\mathrm{anova}}
\def\naive{\mathrm{naive}}

\def\rev{\color{black}}
\def\revone{\color{black}}

\newcommand*\samethanks[1][\value{footnote}]{\footnotemark[#1]}

\newcommand\independent{\protect\mathpalette{\protect\independenT}{\perp}}
\def\independenT#1#2{\mathrel{\rlap{$#1#2$}\mkern2mu{#1#2}}}

\newcommand{\vertiii}[1]{{\left\vert\kern-0.25ex\left\vert\kern-0.25ex\left\vert #1 
		\right\vert\kern-0.25ex\right\vert\kern-0.25ex\right\vert}}

\endlocaldefs

\begin{document}
	
	\begin{frontmatter}
		\title{Residual permutation test for regression coefficient testing}
		\runtitle{Residual permutation test}
		
		\begin{aug}
			\author[A]{\fnms{Kaiyue}~\snm{Wen}\ead[label=e1]{kaiyuew@stanford.edu}}\thanks{\ Equal contribution}\thanks{\ Part of this work was done while KW was at the Institute for Interdisciplinary Information Sciences, Tsinghua University.},
			\author[B]{\fnms{Tengyao}~\snm{Wang}\ead[label=e2]{t.wang59@lse.ac.uk}}\samethanks[1]
			\and
			\author[C]{\fnms{Yuhao}~\snm{Wang}\ead[label=e3]{yuhaow@tsinghua.edu.cn}}\thanks{To whom correspondence should be addressed}
			\address[A]{Department of Computer Science, Stanford University\printead[presep={,\ }]{e1}}
			
			\address[B]{Department of Statistics, London School of Economics and Political Science\printead[presep={,\ }]{e2}}
			
			\address[C]{Institute for Interdisciplinary Information Sciences, Tsinghua University\printead[presep={,\ }]{e3}}
		\end{aug}
		
		\begin{abstract}
		We consider the problem of testing whether a single coefficient is equal to zero in linear models when the dimension of covariates $p$ can be up to a constant fraction of sample size $n$. 
			In this regime, an important topic is to propose tests with \emph{finite-sample valid} size control without requiring the noise to follow strong distributional assumptions. 
			In this paper, we propose a new method, called \emph{residual permutation test} (RPT), which is constructed by projecting the regression residuals onto the space orthogonal to the union of the column spaces of the original and permuted design matrices. 
			RPT can be proved to achieve finite-sample size validity under fixed design with just exchangeable noises, whenever $p < n / 2$. Moreover, RPT is shown to be asymptotically powerful for heavy tailed noises with bounded $(1+t)$-th order moment when the true coefficient is at least of order $n^{-t/(1+t)}$ for $t \in [0,1]$. We further proved that this signal size requirement is essentially {\rev rate-optimal} in the minimax sense. 
			Numerical studies confirm that RPT performs well in a wide range of simulation settings with normal and heavy-tailed noise distributions.
		\end{abstract}
		
		
		
	\end{frontmatter}

\section{Introduction}\label{sec:intro}



Testing and inference of linear regression coefficients is a fundamental problem in statistics research and has inspired methodological innovations in many other research directions in the statistics community \citep[e.g.][]{arias2011global, ZZ14, barber2015controlling, CCD18, BCNZ19}. In this paper, we consider the setting where we have observations $(\bX, \bZ, \bY) \in \R^{n \times p} \times \R^n \times \R^n$ generated according to the following model:
\begin{equation}
	\label{Eq:Model}
	\bs{Y} = \bs{X} \beta + b \bs{Z} + \bs{\varepsilon},
\end{equation}
where $\bs{\varepsilon} := (\varepsilon_1, \ldots, \varepsilon_n)^\top \in \R^n$ is an $n$-dimensional noise vector with $n \ge 2$, and our goal is to test the null hypothesis $H_0: b = 0$ against the alternative $H_1: b \neq 0$. 

Here, we are primarily interested in designing a new coefficient test with finite-sample size validity. In other words, we require our test to have valid size control with arbitrary magnitude of $n$,  instead of requiring some asymptotic regime assumption that may be unrealistic in practice. 
{\revone Classically, one of the most well-known tests for $b=0$ in model~\eqref{Eq:Model} is the ANOVA test \citep{fisher1935design}. It is known to be asymptotically valid when $p$ is fixed and $n\to\infty$, and has finite-sample valid size control when the noise possesses spherical symmetry. However,}
as we will see in Section~\ref{sec:size}, the {\revone finite-sample} size of ANOVA test can be far from the nominal level in the presence of heavy-tailed noises. 
{\revone  This motivates us to propose a new test that has finite-sample valid size control under weaker noise assumptions, especially with relatively large $p$, which is also a topic that has received increasing attention in recent years (see e.g. Section~\ref{sec:literature}). }
In particular, 
{\revone we are interested in developing a test with finite-sample valid size control under a fixed design of $\bX$ and $\bZ$ by assuming that the noise $\boldsymbol{\varepsilon} = (\varepsilon_1,\ldots,\varepsilon_n)^\top$ has \emph{exchangeable components}:}

\begin{assumption}[Exchangeable noise]\label{as:exchangeable}
	For any permutation $\sigma$ of indices $1, \ldots, n$,
	\[
	(\varepsilon_1, \ldots, \varepsilon_n) \overset{\mathrm{d}}{=} (\varepsilon_{\sigma(1)}, \ldots, \varepsilon_{\sigma(n)}).
	\]
\end{assumption}


{\revone We remark that Assumption~\ref{as:exchangeable}  differs from the common assumption that $\E[\bv] = \bs{0}$ for fixed design of $\bX, \bZ$, or that $\E[\bv \mid \bX, \bZ] = 0$ (or even the more relaxed assumption $\E[\bv^\top (\bX, \bZ)] = 0$) when $\bX, \bZ$ follow a random design, which allows for heteroskedastic noise and typically appears in the regression coefficient tests with asymptotically valid size control.}
A common approach to handle exchangeable noise is through the idea of permutation tests~\citep{Pitman37a,Pitman37b,Pitman38}. Recently, \citet{LB21} implemented this idea to the problem of regression coefficient testing.  In their seminal work, the authors proposed a \emph{cyclic permutation test} that achieved finite-sample size validity under Assumption~\ref{as:exchangeable} by exploiting the exchangeability of the noise terms. However, to achieve a size $\alpha$ control, their cyclic permutation test requires that $n / p \ge 1 / \alpha - 1$. For instance, for a sample size of $n = 300$ and a targeting Type-I error rate is $\alpha = 0.01$, at most $p = 2$ covariates are allowed in $\boldsymbol{X}$. This limits the applicability of their test in large dimensions. In this paper, we consider the more challenging question of finite-sample Type-I error control in setting where $p$ is allowed to be of the same order of magnitude as $n$. We propose a \emph{residual permutation test} (RPT), a permutation-based approach that performs hypothesis tests by manipulating the empirical residuals after regression adjustment. The proposed test is guaranteed to have the correct Type-I error control whenever $p < n / 2$. Moreover, our result is fixed design and does not require any regularity conditions on the design matrix $\bX$.


In addition to proving its finite-sample {\revone size} validity,
we further analyze the statistical power of the proposed test in {\revone the regime where $p$ can be up to a constant fraction of $n$, which we will refer to as the proportional regime in this work}, 
especially when the $\varepsilon_i$'s follow a heavy-tailed distribution.
{\revone It should be noted that just with Assumption~\ref{as:exchangeable} we can only guarantee our test to have correct size control, the resulting test may not necessarily have power. Indeed, just with Assumption~\ref{as:exchangeable}, $b$ may not even be identifiable, so that there is no test that is uniformly valid under the null while still maintainng power against some alternative. For example, in the extreme case where $\bZ = \bX \beta^Z$ for some $\beta^Z \in \mathbb{R}^p$, i.e., that $\bZ$ can be expressed as a linear representation of the column vectors of $\bX$, then $b$ is not an identifiable parameter anymore, i.e., no test can have any power. As we can also see in~\eqref{eq:pvalest}, under Assumption~\ref{as:exchangeable} and with such choice of $\bZ$, we can always have $\pr(\pval \le \alpha) \le \alpha$ for all $\alpha \in [0, 1]$, no matter how large $b$ is. This means that in order for our test to have power, additional assumption on $\bZ$ must be imposed.} 
As we will discuss further in Section~\ref{sec:heavy-tail}, statistical methods with robustness to heavy-tailed data have significant demands in practice~\citep{Eklund16,WPL15,Cont01}, and has been actively studied in both modern statistics and theoretical computer science communities.
Despite its importance, there is a lack of available tools that can handle regression coefficient testing under this {\revone proportional} regime with heavy-tailed noise. In this paper, we fill this gap by showing that {\revone under a suitable modelling assumption of $\bZ$,} when the $\varepsilon_i$'s are i.i.d.\ and have a finite $(1 + t)$-th order moment for any $t \in [0,1]$, and that $n / p \ge 3 + m$ for some $m > 0$, our proposed test is asymptotically powerful whenever the coefficient $b$ is of order at least $n^{-t/(1+t)}$.  {\rev In proving this result, a crucial step is to establish a concentration bound for projected length of a random vector with independent heavy-tailed components. This concentration bound may be of independent interest for future research on statistical procedures with heavy-tailed noise, and is stated in Corollary~\ref{cor:convergence}.} We also studied the minimax {\rev rate} optimality of regression coefficient testing with heavy-tailed noises; and proved that in the presence of heavy-tailed noise with only a finite $(1+t)$-th moment, the $n^{-t/(1+t)}$ order requirement for $b$ is essentially {\rev rate-optimal}. 


Since ANOVA has been used extensively in practical applications, as an independent contribution, we provide a more comprehensive analysis of the ANOVA test. 
{\rev Specifically, while ANOVA can be shown to have finite-sample size validity with \emph{spherically symmetric noise}, our simulations show that the it can substantially violate the nominal size control under more general noise distributions.}
At the same time, we propose another permutation-based test: naive residual permutation test (naive RPT), which, like ANOVA, also has valid size control under spherically symmetric noise distribution whenever $p < n$. While naive RPT is still not valid for non-spherically symmetric noises, it does appear to have smaller Type I error violations compared to ANOVA.

In summary, we make the following contributions in this work:
\begin{itemize}
	\item We propose a new test that has finite-sample size validity with fixed-design linear models and exchangeable noises whenever $p < n / 2$.
	\item We prove that when the noise variables are heavy-tailed with bounded $(1 + t)$-th order moment for $t \in [0, 1]$ {\revone and under suitable assumptions of $\bZ$}, our test is asymptotically powerful when $b$ is at least of order $n^{-t/(1+t)}$.
	\item  We perform numerical analysis to show that ANOVA is indeed invalid in general distributions, especially with heavy-tailed data. We also studied other theoretical properties of ANOVA.
	\item  We discuss the {\rev minimax rate optimality} of regression coefficient test with heavy-tailed distributions, and show that our test is essentially optimal in the minimax sense.
\end{itemize}

The rest of this paper is organized as follows. In Section~\ref{sec:literature}, we review existing results in 
{\rev regression coefficient testing, permutation- and randomization-based tests} and heavy-tailed data. In Section~\ref{sec:size}, we provide more studies on the finite-sample properties of ANOVA test with non-Gaussian noises, and propose a new test that is easier to implement and more robust to non-Gaussianity. As ANOVA test has been heavily used in practical applications, we believe this is of independent interest. In Section~\ref{sec:null}, we present our method, and prove its finite-sample size validity. In Sections~\ref{sec:alternative} and~\ref{sec:minimax}, we provide power analysis of RPT and study {\rev its minimax rate optimality} 
under some heavy-tailed assumptions. Finally, in Section~\ref{sec:numerical} we provide numerical analysis. In Section~\ref{sec:discussion}, we end the manuscript with a discussion.

\subsection*{Notation}
We conclude this section by introducing some notation used throughout the paper. For any $n \times p$ dimensional matrix $\bs{A}$, we denote by $\matspan(\bs{A})$ the subspace spanned by the $p$ column vectors of $\bs{A}$; and we write $\matspan(\bs{A})^\perp$ as the space that is orthogonal to $\matspan(\bs{A})$. Given an $n$-dimensional vector $\bs{a}$, we denote by $\proj_{\bs{A}}(\bs{a})$ the projection of $\bs{a}$ onto the subspace $\matspan(\bs{A})$, and denote by $\|\bs{a}\|_2$ as the $\ell_2$-norm of the vector $\bs{a}$. Given two $n \times q_1$ and $n \times q_2$ dimensional matrices $\bs{A}$, $\bs{B}$, we denote by $(\bs{A}, \bs{B})$ as the $n \times (q_1 + q_2)$ matrix via column concatenation of matrices $\bs{A}$ and $\bs{B}$. We write $\cN(0, 1)$ as standard normal distribution. For two sequences $(a_n)_{n\in\mathbb{N}}$ and $(b_n)_{n\in\mathbb{N}}$, we write $a_n = O(b_n)$, or equivalently $b_n = \Omega(a_n)$, if there exists a universal constant $C>0$ such that $|a_n| \leq C|b_n|$ for all $n$; we write $a_n = o(b_n)$, or equivalently $b_n = \omega(a_n)$, if $|a_n| / |b_n| \to 0$.

\section{Literature review}\label{sec:literature}

Our work spans a wide range of research directions, including {\rev hypothesis testing of regression coefficients}, permutation- {\rev and randomization-}based hypothesis tests and {\rev heavy-tailed data analysis}. In this section, we compare our research to works within each direction.

\subsection{{\rev Hypothesis} testing of regression coefficients}\label{sec:coeff}

The most classical approach for testing the null hypothesis $b=0$ is through the analysis of variance (ANOVA) test~\citep{fisher1935design}. ANOVA test was originally proposed by Sir Ronald Fisher in the 1920s, and has been widely used in economics~\citep{Doane16}, finance~\citep{Paolella18} and biology~\citep{Lazic08} etc. In the context of single coefficient testing, when $n >p+1$ and $\bv$ follows a spherically symmetric distribution, if $\tilde\beta := \argmin_{\beta} \|\bs{Y} - \bs{X}\beta\|_2^2$ and $(\hat\beta,\hat b) := \argmin_{(\beta,b)} \|\bs{Y} - \bs{X}\beta - 
b \bs{Z}\|_2^2$, then under $H_0$, the test statistic
\begin{equation}
	\label{Eq:Fstat}
	\pval_{\anova}:= \frac{\|\bs{Y} - \bs{X}\tilde\beta\|_2^2 - \|\bs{Y} - \bs{X}\hat\beta - \hat b \bs{Z}\|_2^2}{\|\bs{Y} - \bs{X}\hat\beta - \hat b \bs{Z}\|_2^2/(n-p-1)} \sim F_{1,n-p-1}
\end{equation}
can be used to construct a test where $H_0$ is rejected when $\phi_{\mathrm{anova}}$ exceeds the $1-\alpha$ quantile of the $F_{1,n-p-1}$ distribution.  As the above distributional result is nonasymptotic and holds whenever $n > p + 1$, the associated test is valid even {\rev when $p$ diverges as a constant fraction of $n$}. However, as we will discuss in Section~\ref{sec:size}, {\revone for a general noise distribution of} $\bv$, ANOVA test is usually \emph{not} guaranteed to have a valid Type-I error control. This encourages us to construct hypothesis tests with valid Type-I error control allowing a broader class of noise distributions. 

As emphasized by \citet{LB21}, this is a challenging problem, with a ``century long effort'' in the statistical community to alleviate the strong assumption of ANOVA. {\revone In the context of finite-sample size validity,} some representative works include~\citet{Hart70,Mein15}. However, the two methods mentioned above still require the noise to follow certain geometric constraint, which is either symmetric about 0 or rotationally invariant.
\citet{LB21} represented, to the best of our knowledge, the first work that established finite-sample size control with only exchangeable noise.
However, as mentioned in the introduction, the cyclic permutation test proposed in \citet{LB21} requires the dimension of $p$ to be much smaller than $n$ for valid size control, and no corresponding statistical power analysis was provided. Alternative test with less restrictive assumptions on dimension $p$ was proposed in~\citet{d2022robust}, which relaxes the dimensionality assumption by requiring the rows of $\bX$ to follow a discrete random distribution with a relatively small number of unique values.

Besides finite-sample size validity, a less demanding criteria for coefficient test is the \emph{asymptotic size validity}. {\rev The idea of permutation or randomization have been heavily used to propose asymptotically valid test; see Section~\ref{sec:preperm} for more details.} In the high-dimensional regime where $p$ is proportional or even much larger than $n$, debiased / desparsified Lasso was proposed to construct confidence intervals and perform coefficient tests \citep{ZZ14,vandegeer14,JM14}. By invoking 1) certain sparsity conditions on the regression coefficients; 2) some regularity conditions on the design matrix $\bX$ 
and 3) sharp tail bounds on the noise variables,
debiased / desparsified Lasso is guaranteed to establish asymptotically valid p-value and confidence intervals for regression coefficients. 
{\rev We remark that the additional sparsity assumption on the regression coefficients allow for the dimension $p$ to diverge at a much faster rate than $n$ compared to asymptotic regime studied in the current paper.}
Other follow up studies include~\citet{ZB18,BCNZ19,SB19}, to name a few. 

More broadly speaking, regression coefficient test can be viewed as a subdomain of the more general conditional independence testing, i.e., testing the null hypothesis $Y \independent Z \mid X$, treating $\bX, \bY, \bZ$ as i.i.d. realizations from some hypothesized superpopulation. Unfortunately, when one has no assumption on the joint distribution of the random variables $X, Y$ and $Z$,~\citet{SP20} proved that it is a ``statistically hard problem'', in the sense that a valid test for the null does not have power against \emph{any} alternative. This means that some restrictions must be added to the class of null distributions to have some power. Following this insight, an important research question then, is to propose tests with valid size control under weak distributional assumptions. In this paper, we show that a linear functional relationship between $\bY$ and $\bX$ is sufficient to have finite-sample size validity with non-trivial power.



\subsection{Permutation- {\rev and randomization-}based hypothesis tests}\label{sec:preperm}

As also mentioned in the introduction section, our new method is based on permutation test~\citep{Pitman37a,Pitman37b,Pitman38}. {\rev Application of permutation and related randomization techniques for statistical inference has a long history in statistics and econometrics \citep{fisher1935design, rubin1980comment, rosenbaum1984conditional, romano1990behavior, kennedy1995randomization, rosenbaum2002covariance,canay2017randomization,young2019channeling}.}  Permutation test was originally developed for independence testing. Specifically, using the exchangeability properties of the sampled data, permutation test is guaranteed to have finite-sample size validity, without any geometric or moment constraints on the underlying distributions.

{\rev For the task of regression coefficient testing, \citet{freedman1983nonstochastic} proposed tests based on permuted regression residuals, and analyzed its asymptotic size control in a fixed dimension. \citet{diciccio2017robust} considered a permutation test using the studentized partial correlation of $\boldsymbol{Y}$ and $\boldsymbol{Z}$ given $\boldsymbol{X}$ and derived asymptotic size and power of the test in a fixed dimension setting. \citet{toulis2019invariant} studied a test based on permuting the residuals of $\boldsymbol{Y}$ regression against $(\boldsymbol{Z},\boldsymbol{X})$. More recently,  \citet{LB21,d2022robust} used permutation test and its extensions to obtain exact size control for testing a single component or a subvector of regression coefficients. 
	{\revone      We note that although all the works mentioned in this paragraph use permutation tests as the basic building block of their tests, their validity guarantees are based on different assumptions.
		Among them, \citet{toulis2019invariant, LB21, d2022robust} considered the size validity of their tests under a similar noise invariance assumption to our Assumption~\ref{as:exchangeable}.
	}
	
	Other related applications of permutation tests include} sharp null hypothesis tests \citep{CDM17,CDLM21}, {\rev instrumental variable tests \citep{imbens2005robust},} and conditional independence tests \citep{BWBS20, KNBW21}.

{\revone
	In addition to the above permutation-based testing, the knockoff-based procedure \citep{candes2018panning} can also be used to perform asymptotically valid coefficient testing in the regression setting. Another related line of works exploit bootstrap or jackknife techniques to provide tests with asymptotic size validity
	~\citep[e.g.][]{miller1974unbalanced,freedman1981bootstrapping,mammen1993bootstrap,chatterjee1999generalised,bickel2008choice}. See the discussions in \citet[Section~A in Supplementary Materials]{LB21} for a comprehensive overview of this area.
}

\subsection{Heavy-tailed data} \label{sec:heavy-tail}

To understand the efficiency of the proposed method in heavy tailed data, in this paper, we further provide power analysis when the noise terms follow a heavy-tailed distribution. In classical high-dimensional literature, due to the simplicity of theoretical analysis, existing methods usually focus on data with sharp tail bounds, such as sub-Gaussian or sub-exponential tail bounds~\citep[see, e.g.][]{Wain19}. However, as also discussed by \citet{SZF20}, such strong tail condition may not be reasonable in real world applications, such as neuroimaging~\citep{Eklund16}, gene expression analysis~\citep{WPL15}, and finance~\citep{Cont01}. 

Since the pioneering work by \citet{Catoni12}, the problem of extracting useful information from heavy-tailed data (or the related adversarially contaminated data) has been an active area of research in mathematical statistics and theoretical computer science literature in the past ten years~\citep{BCL13,LTMP18,LM19,SZF20,FWZ21}. When we allow the dimension $p$ to grow with $n$, heavy-tailed data has been actively studied in mean estimation ~\citep{LM19,LM21}, regression coefficient estimation~\citep{Wang13,FLW17,SZF20,PJL20} and covariance matrix analysis~\citep{LT18,FWZ21}. The definition of ``heavy-tail'' may vary across different articles. Among all literature working with heavy-tailed noise, our assumptions are most similar to those in \citet{SZF20,BCL13}, which assume that the noise variables has at most a finite $(1 + t)$-th order moments for some $t \in (0,1]$ without any geometric or shape constraints. To our knowledge this is also the weakest heavy tail assumption studied in the literature.

In the context of coefficient testing, few methods have been proposed that can work with heavy-tailed data. We fill this gap by providing statistical power guarantees of our constructed test in the presence of heavy-tail noises. Our power analysis stems from our new theoretical insight on the asymptotic convergence of heavy-tailed random variables after subspace projections. It would be of interest if these results could be extended to understand the power of permutation-testing based hypothesis tests in other heavy-tailed scenarios.



\section{Finite-sample size validity of ANOVA beyond Gaussianity}\label{sec:size}

As ANOVA has been frequently used in empirical analysis, it would be of interest to provide a more comprehensive analysis on the sensitivity of ANOVA test with respect to the Gaussianity assumption, both empirically and theoretically. {\rev In fact, although not explicitly stated in \citet{fisher1935design}, Fisher recognized that ANOVA's size validity only requires the noise to be spherically symmetric instead of Gaussian \citep[pp.~163--164]{stigler2016seven}. 
	We provide a slight generalization of this result in Lemma~\ref{lem:anova}, which shows that ANOVA has valid size when \emph{either} the design \emph{or} the noise is spherically symmetric, in the sense defined below. 
} 

\begin{definition}
	We say that a random matrix $\bs{A} \in \R^{n \times q}$ follows a spherically symmetric distribution if for any $\bs{Q}\in\mathbb{O}^{n\times n}$, $\bs{A} \stackrel{\mathrm{d}}{=} \bs{Q} \bs{A}$, where $\mathbb{O}^{n\times n}$ is the set of $n \times n$ orthonormal matrices.
\end{definition}

\begin{lemma}
	\label{lem:anova}
	Suppose $\bs{Y}$ is generated under~\eqref{Eq:Model} with $\beta\in\mathbb{R}^p, b = 0$. Suppose also that $\bv$ is a random vector that is almost surely not a zero vector, $(\bX, \bZ)$ is either deterministic or independent from $\bv$. If either $\bs{\varepsilon}$ or $(\bX, \bZ)$ follows a spherically symmetric distribution, then the test statistic $\pval_{\anova}$ defined in~\eqref{Eq:Fstat} satisfies $\pval_{\anova} \sim F_{1,n-p-1}$. 
\end{lemma}

{\rev For the sake of completeness we provide a proof of Lemma~\ref{lem:anova} in the Supplementary Material.}
The spherical symmetry in the noise or the design is slightly weaker than the usual Gaussianity constraint, however, it is still too strong for many real data applications. For instance, if we assume that observations $(X_i, Z_i, Y_i)$ are independent, then this assumption amounts to either i.i.d.\ normal noise or an i.i.d.\ multivariate normal design.

We now perform a numerical experiment to analyze the size validity of ANOVA test under general distributional classes of $\bv$. We generate data $(\boldsymbol{X}, \boldsymbol{Z}, \boldsymbol{Y})$ according to the model specified in~\eqref{Eq:Model} and that
\begin{equation}\label{Eq:Model2}
	\bZ = \bX \beta^Z + \be.
\end{equation}
In the simulation, we set $b=0$; since the result of ANOVA is invariant to $\beta, \beta^Z$, we simply set them to be zero vectors. We also set $\bX$ as $n \times p$ matrices with i.i.d.\ entries following either $\cN(0,1)$ or $t_1$ distribution, with $(n, p) = (300, 100), (600, 100)$ or $(600, 200)$; and $\be$ and $\bv$ have i.i.d.\ components from one of $\cN(0,1)$, $t_2$ or $t_1$ distributions. 

\begin{table}[t!]
	\centering
	\begin{tabular}{cccc|cc|cc}
		\hline\hline
		& & & & \multicolumn{2}{c}{ANOVA} & \multicolumn{2}{c}{Naive}\\
		\hline
		$n$ & $p$ & X type & noise type & 1\% & 5\% & 1\% & 5\%\\
		\hline
		$300$ & $100$ & Gaussian & Gaussian & $1.01$ & $4.99$ & $1.00$ & $4.96$\\
		$300$ & $100$ & Gaussian & $t_1$ & $1.81$ & $3.1$ & $1.58$ & $3.38$\\
		$300$ & $100$ & Gaussian & $t_2$ & $1.53$ & $4.83$ & $1.39$ & $4.83$\\
		$300$ & $100$ & $t_1$ & Gaussian & $1.01$ & $4.99$ & $1.03$ & $5.03$\\
		$300$ & $100$ & $t_1$ & $t_1$ & $2.43$ & $3.96$ & $1.58$ & $4.25$\\
		$300$ & $100$ & $t_1$ & $t_2$ & $1.80$ & $5.03$ & $1.41$ & $5$\\
		$600$ & $100$ & Gaussian & Gaussian & $0.95$ & $4.9$ & $0.96$ & $4.88$\\
		$600$ & $100$ & Gaussian & $t_1$ & $1.63$ & $2.45$ & $1.28$ & $3.36$\\
		$600$ & $100$ & Gaussian & $t_2$ & $1.69$ & $4.61$ & $1.28$ & $4.79$\\
		$600$ & $100$ & $t_1$ & Gaussian & $1.05$ & $4.86$ & $1.02$ & $4.87$\\
		$600$ & $100$ & $t_1$ & $t_1$ & $1.88$ & $2.84$ & $1.06$ & $3.84$\\
		$600$ & $100$ & $t_1$ & $t_2$ & $1.74$ & $4.79$ & $1.14$ & $5.01$\\
		$600$ & $200$ & Gaussian & Gaussian & $1.01$ & $4.96$ & $1.03$ & $4.93$\\
		$600$ & $200$ & Gaussian & $t_1$ & $1.41$ & $2.48$ & $1.24$ & $2.82$\\
		$600$ & $200$ & Gaussian & $t_2$ & $1.50$ & $4.67$ & $1.36$ & $4.72$\\
		$600$ & $200$ & $t_1$ & Gaussian & $1.01$ & $5.11$ & $0.98$ & $5.09$\\
		$600$ & $200$ & $t_1$ & $t_1$ & $2.02$ & $3.26$ & $1.33$ & $3.74$\\
		$600$ & $200$ & $t_1$ & $t_2$ & $1.70$ & $4.64$ & $1.34$ & $4.66$\\
		\hline\hline
	\end{tabular}
	\caption{\label{Tab:ANOVA} Percentage of rejection of the ANOVA test and naive residual permutation test, estimated over 100000 Monte Carlo repetitions, for various noise distributions at nominal levels of $\alpha=1\%$ and {\revone $\alpha=5\%$}. Data are generated by models~\eqref{Eq:Model} and~\eqref{Eq:Model2}, with $\bX$, $\bs{\varepsilon}$ and $\bs{e}$ having independent components distributed according to the various $X$ types and noise types described in the table. Standard errors for all entries are in the range of $0.02\%$ to $0.05\%$.}
\end{table}

Table~\ref{Tab:ANOVA} summarizes the performance of ANOVA test from $100000$ Monte Carlo simulations. We consider the sizes of the ANOVA test at nominal levels {\revone $\alpha = 0.01, 0.05$}. According to the simulation results, when the noises of $\be$ and $\bv$ follows a standard normal distribution, the ANOVA test has the correct size control, which is consistent with Lemma~\ref{lem:anova}. However, when normality is violated, the ANOVA test will be overly optimistic,  with an empirical size more than twice as large as the nominal level in some {\revone $1\%$-level tests (this issue is more pronounced if we consider a $0.5\%$ test level; see Table~\ref{Tab:Size5pc} in the Supplementary Material)}. In particular, the performance of noise type $t_1$ is in general worse than that of $t_2$, this means that ANOVA test is more vulnerable to heavy-tailed noises. Moreover, the performance of ANOVA is worse with a heavy-tailed design matrix $\bX$.

To better understand the empirical distribution of the simulated p-values, we plot their histogram in Figure~\ref{fig:histgram}(a)-(c).  Apparently, all the histograms are far from uniform on $[0,1]$ under the null hypothesis, with a large spike near zero. In addition, the magnitude of the spike increases as $n$ becomes smaller or that $\bv$ or $\bX$ becomes more heavy-tailed. 
Another interesting property is that the histograms are usually ``U-shaped'', where the peaks appear at regions near either $1$ or $0$.
In sum, when data are generated from non-Gaussian and in particular heavy-tailed distributions, the ANOVA tests are usually far from the correct level. 

\begin{figure}[t!]
	\centering
	\subfigure[]{\includegraphics[width=0.32\textwidth]{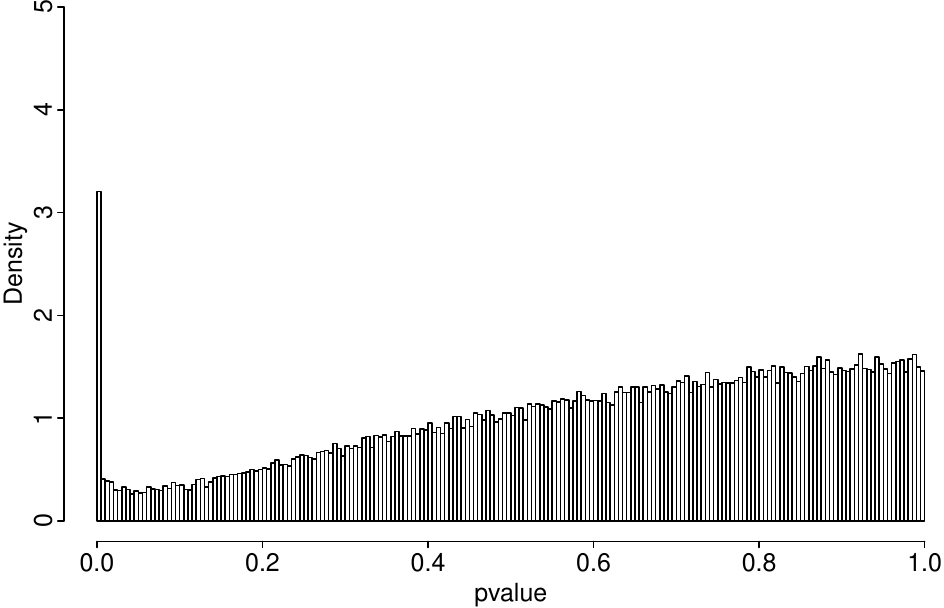}} 
	\subfigure[]{\includegraphics[width=0.32\textwidth]{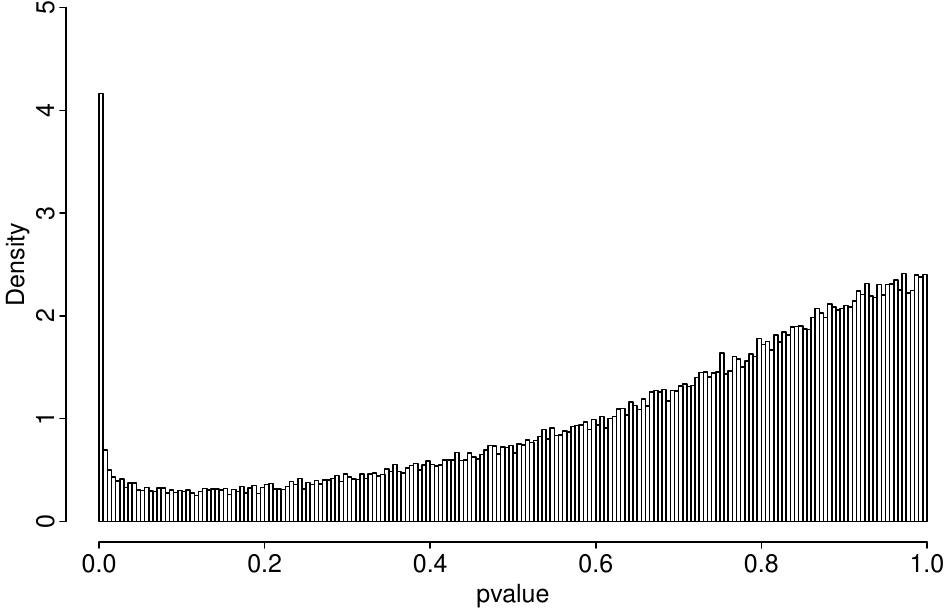}}
	\subfigure[]{\includegraphics[width=0.32\textwidth]{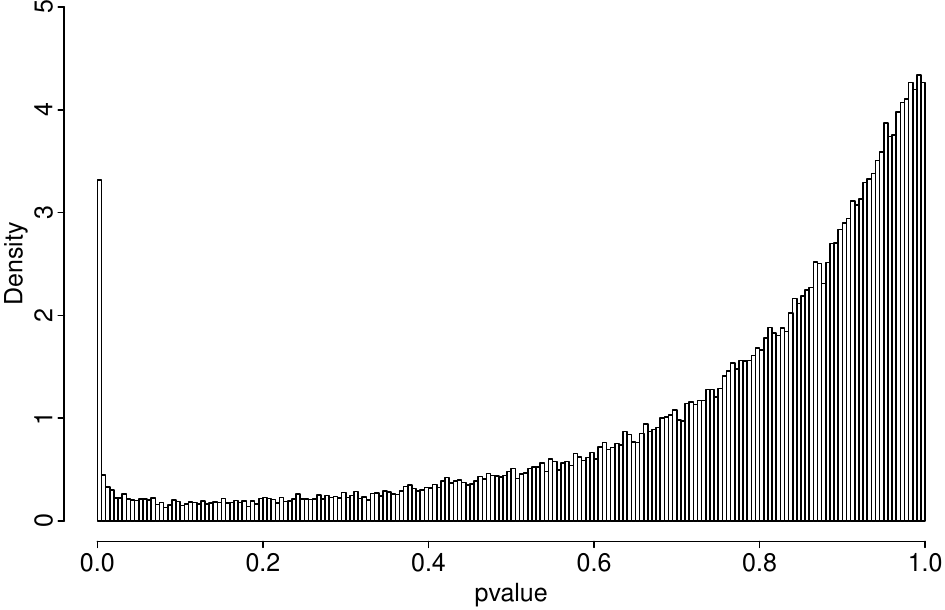}} \\
	\subfigure[]{\includegraphics[width=0.32\textwidth]{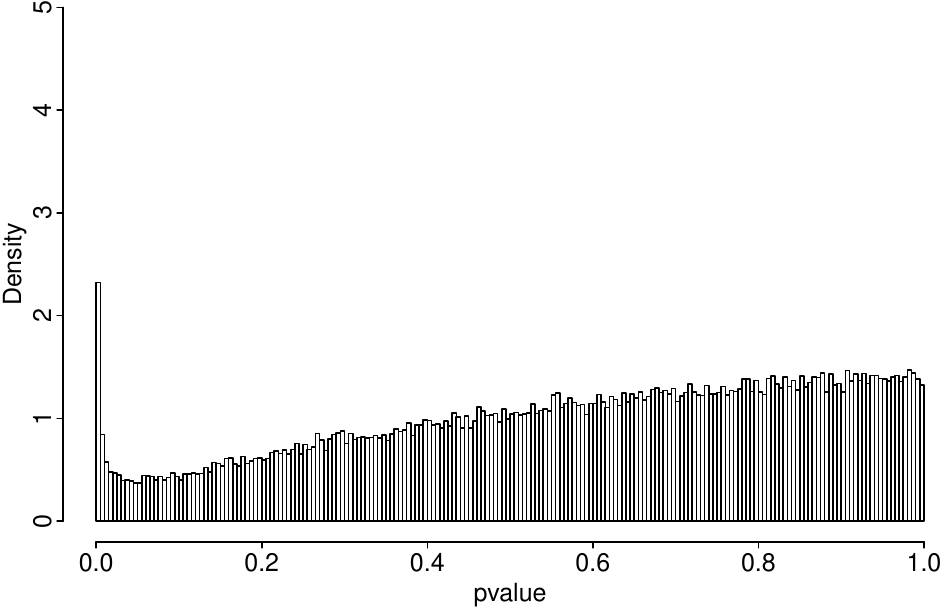}} 
	\subfigure[]{\includegraphics[width=0.32\textwidth]{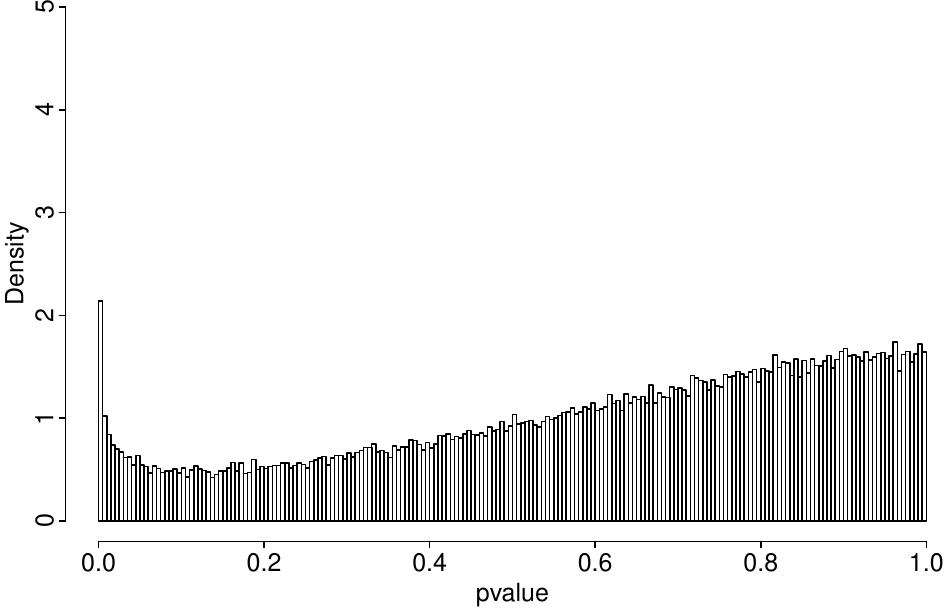}}
	\subfigure[]{\includegraphics[width=0.32\textwidth]{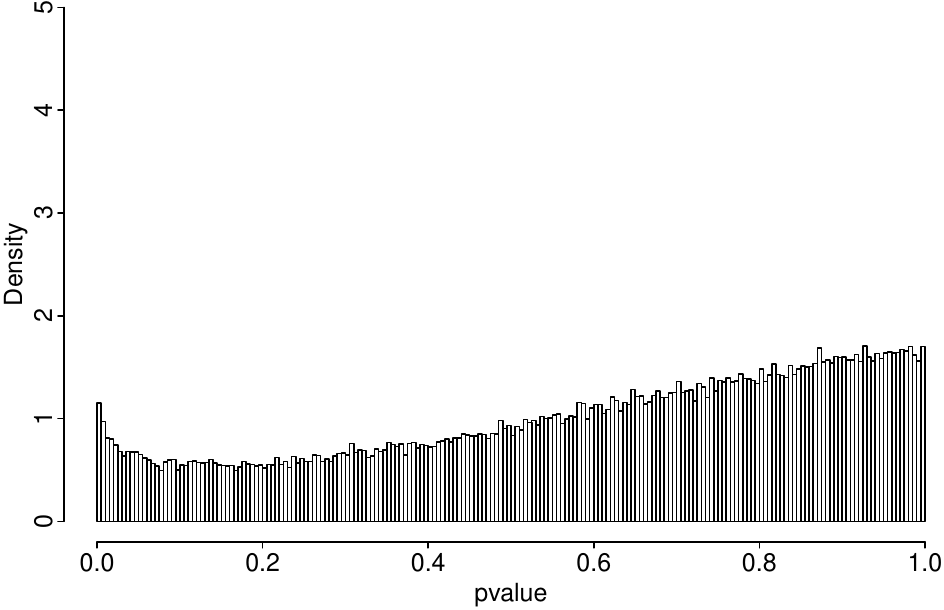}}
	\caption{Histogram of p-values under the null for ANOVA test and naive residual permutation test from $100000$ Monte-Carol replicates. The first line are the histograms of the ANOVA test under different specifications. Specifically, (a) is the result with Gaussian design, $n = 300, p = 100$ and $\bv$ has independent $t_1$ components; (b) is the histogram with the same setting as in (a) except that we switch from Gaussian design to $t_1$ design; (c) is the histogram with Gaussian design, $n = 600, p = 100$ and $\bv$ has independent $t_1$ components. The second line are the histogram for naive test. (e)-(f) use the same simulation settings as (a)-(c). }
	\label{fig:histgram}
\end{figure}

It is worth noting that when $\beta=0$ in~\eqref{Eq:Model}, we can easily construct a permutation test with valid size control by comparing the correlation of $Y$ to $Z$ and to its permutations. From this intuition, a straightforward approach is to first regress both $\bY$ and $\bZ$ onto $\bX$ to eliminate the influence of $\bX$, and then to use regression residuals for permutation test construction.  Specifically, let $\hat{\bs{R}}_\varepsilon := (\bs{I} - \bX (\bX^\top \bX)^{-1} \bX^\top) \bY$ and $\hat{\bs{R}}_e := (\bs{I} - \bX (\bX^\top \bX)^{-1} \bX^\top) \bZ$ be the regression residuals after projecting $\bY$ and $\bZ$ onto $\bX$ respectively. Let $\bs{V}_0 \in \R^{n \times (n - p)}$ be a matrix with orthonormal columns spanning an $(n-p)$-dimensional subspace of $\matspan(\bs{X})^\perp$, then $\bs{I} - \bX (\bX^\top \bX)^{-1} \bX^\top = \bs{V}_0\bs{V}_0^\top$. Hence under $H_0: b =0$,  the regression residuals $\hat{\bs{R}}_\varepsilon$ satisfy $\hat{\bs{R}}_\varepsilon = \bV_0 \bV_0^\top \bY = \bV_0 \bV_0^\top \bv$. From above, we construct a test, which we call as \emph{naive residual permutation test}, based on the \emph{projected residuals} $\hat{\bv} := \bV_0^\top \hat{\bs{R}}_\varepsilon = \bV_0^\top \bY$ and $\hat{\bs{e}} := \bV_0^\top \hat{\bs{R}}_e =  \bV_0^\top \bZ$ as
\begin{equation}\label{eq:naive}
	\pval_{\naive} = \frac{1}{K + 1} \left(1 + \sum_{k = 1}^{K} \one(|\hat{\be}^\top \hat{\bv}| \le |\hat{\be}^\top \bP_k \hat{\bv}|)\right),
\end{equation}
where the $\bP_k \in \R^{(n - p) \times (n - p)}$'s are random permutation matrices that are sampled uniformly at random from the set of all permutation matrices. 
Lemma~\ref{lem:naive} shows that under a slightly weaker condition than Lemma~\ref{lem:anova}, $\pval_{\naive}$ has valid Type-I error control.

\begin{lemma}
	\label{lem:naive}
	Suppose $\bs{Y}$ is generated under~\eqref{Eq:Model} with $\beta\in\mathbb{R}^p$, $b=0$. If either 
	\begin{enumerate}[(a)]
		\item $\bs{\varepsilon}$ or $(\bX, \bZ)$ follows a spherically symmetric distribution;
		\item $\bs{Z}$ is generated under~\eqref{Eq:Model2} and either $\be$ or $(\bX, \bY)$ follows a spherically symmetric distribution,
	\end{enumerate}
	$\pval_{\naive}$ is a valid p-value, i.e., for all $\alpha \in (0,1)$, $\prob(\pval_{\naive} \le \alpha) \le \alpha$.
\end{lemma}

While Lemma~\ref{lem:naive} is slightly less stringent than  Lemma~\ref{lem:anova}, it still requires the spherical symmetry in distributions. To better understand their empirical performances, we also show the performance of $\pval_{\naive}$ with non-Gaussian noises or non-Gaussian designs in Table~\ref{Tab:ANOVA} and Figures~\ref{fig:histgram}(d)-(f) . Without the strong Gaussianity or spherically symmetry assumption, $\pval_{\anova}$ is also not guaranteed to have finite-sample size validity. Nevertheless, when both tests are invalid, the size of naive permutation test is closer to the correct level than its competitor. This indicates that naive test is more robust to non-Gaussian distributions. Moreover, the naive test is an intuitive method and is easy to implement. Thus, the naive test could be used as a preferrable alternative to ANOVA in real data analysis when $n / 2 \le p < n$. 

\section{Residual permutation test: methodology and size validity}\label{sec:null}

In Section~\ref{sec:size}, we have shown from simulation experiments that a naive permutation test on the residuals, although more robust than ANOVA, is still not guaranteed to have finite-sample size validity with just exchangeable noise.  In this section we describe a more refined test using the projected residuals $\hat{\bs{\varepsilon}}$ and $\hat{\bs{e}}$, which we call the \emph{residual permutation test} (RPT),  and present its finite-sample size validity guarantee in Theorem~\ref{thm:null}. For intuition behind such construction, we refer the readers to Section~\ref{sec:intuition}. {\revone We will assume throughout this section that the design matrix $(\bX, \bZ)$ is deterministic.}

To describe RPT, we write $\mathcal{P}$ for the set of all permutation matrices in $\mathbb{R}^{n\times n}$ and we denote by $\bs{P}_0 = \bs{I} \in \mathcal{P}$ the identity matrix. To successfully perform the regression permutation test, we first need to randomly generate of a sequence of $K$ permutation matrices $\{\bs{P}_1, \ldots, \bs{P}_K\} \subseteq \mathcal{P} \setminus \{\bs{P}_0\}$, such that together with $\bs{P}_0$ they form a group:
\begin{assumption}\label{as:permutation}
	The set of permutation matrices $\mathcal{P}_K := \{\bs{P}_0, \bs{P}_1, \ldots, \bs{P}_K\}$ satisfies that for any $\bs{P}_i, \bs{P}_j$, there exists a $k \in \{0, \ldots, K\}$ such that $\bs{P}_k = \bs{P}_i \bs{P}_j$.
\end{assumption}
We write $\bs{V}_0 \in \R^{n \times (n - p)}$ as a matrix with orthonormal columns spanning an $(n-p)$-dimensional subspace of $\matspan(\bs{X})^\perp$ and $\bs{V}_k := \bs{P}_k \bs{V}_0$.\footnote{If $\bs{X}$ is full column rank, then $\bV_0 \bV_0^\top = \bs{I} - \bX (\bX^\top \bX)^{-1} \bX^\top$ and $\matspan(\bs{V}_0)$ and $\matspan(\bs{X})^\perp$ are the same space. Otherwise, $\matspan(\bs{V}_0)$ is a subspace of $\matspan(\bs{X})^\perp$.} In addition, we denote by $\tilde{\bs{V}}_k \in \R^{n \times (n - 2p)}$ a matrix with orthonormal columns spanning a subspace of $\matspan(\bs{V}_0) \cap \matspan(\bs{V}_k)$. Recall that $\hat{\bs{e}} := \bV_0^\top \bZ$ and $\hat{\bv} := \bV_0^\top \bY$. Given a fixed $T:\mathbb{R}^{n-2p}\times \mathbb{R}^{n-2p}\to \mathbb{R}$, we can calculate the p-value of our coefficient test via:
\begin{equation}\label{eq:pvalestori}
	\pval := \frac{1}{K+1} \left(1 + \sum_{k=1}^{K} \one\left\{ \min_{\tilde{\bs{V}} \in \{\tilde{\bs{V}}_1, \ldots, \tilde{\bs{V}}_K\}} T\left(\tilde{\bs{V}}^\top \bs{V}_0 \hat{\bs{e}}, \tilde{\bs{V}}^\top \bs{V}_0 \hat{\bs{\varepsilon}}\right) \leq T\left(\tilde{\bs{V}}_k^\top \bs{V}_0 \hat{\bs{e}}, \tilde{\bs{V}}_k^\top \bs{V}_k \hat{\bs{\varepsilon}}\right)\right\}\right),
\end{equation}
where $T$ can be any bivariate function. For example, one can choose $T(x,y) = |\langle x, y\rangle|$. As demonstrated in the Supplementary Material, the above definition of $\pval$ can be simplified as the following equivalent form 

\begin{equation}\label{eq:pvalest}
	\pval := \frac{1}{K+1} \left(1 + \sum_{k=1}^{K} \one\left\{ \min_{\tilde{\bs{V}} \in \{\tilde{\bs{V}}_1, \ldots, \tilde{\bs{V}}_K\}} T\left(\tilde{\bs{V}}^\top \bZ, \tilde{\bs{V}}^\top \bY\right) \leq T\left(\tilde{\bs{V}}_k^\top \bZ, \tilde{\bs{V}}_k^\top \bP_k \bY\right)\right\}\right).
\end{equation}
The following theorem shows that the proposed p-value is uniformly valid under the null:

\begin{theorem}\label{thm:null}
	Suppose that $(\bs{X},\bs{Z},\bs{Y})$ is generated under model~\eqref{Eq:Model} with $p < n/2$ and that the noise $\bs{\varepsilon}$ satisfies Assumption~\ref{as:exchangeable}. Suppose $\{\bs{P}_k: k = 0,\ldots,K\}$ satisfies Assumption~\ref{as:permutation}. Under $H_0: b = 0$, $\phi$ defined in~\eqref{eq:pvalest} is a valid p-value, i.e.\ $\pr\left(\pval \leq \alpha\right) \leq \alpha$ for all $\alpha\in[0,1]$.
\end{theorem}
We remark that as shown in Theorem~\ref{thm:null}, an important advantage of RPT is that the result is finite sample in the sense that it holds for arbitrary size of $n$. Moreover, our result assumes a fixed-design matrix and does not require any assumption on $\bX$ for finite-sample size validity. For example, the rank of $\bX$ even does not necessarily need to be $p$. Also, Theorem~\ref{thm:null} shows that RPT has valid size for any choice of function $T(\cdot, \cdot)$ and number of permutations $K$. However, in practice, to have good power under the alternative, we typically set $T(x,y) = |\langle x, y\rangle|$ and choose a moderate size of $K = O(1/\alpha)$.

%
%
%


\subsection{Some intuition of RPT}\label{sec:intuition}

In this section, we discuss the intuition behind~\eqref{eq:pvalestori}. As demonstrated in Section~\ref{sec:size}, a naive permutation test on the residuals is in general not valid in the finite-sample setting with just exchangeable noises. This is because under the null, $\phi_\naive$ performs permutations on the vector $\hat{\bv} = \bV_0^\top \bv$ instead of $\bv$ itself. Even if $\bv$ is an exchangeable random vector, $\bV_0^\top \bv$ may no longer be so, which renders the naive test invalid. 

To overcome this challenge, we may want to construct a new test that, under $H_0$, is equivalent to permuting the noise vector $\bv$ directly, instead of the transformed noise $\bV_0^\top\bv$. 
Interestingly, this goal can be achieved based on a further transformation of the vector $\bV_0^\top \bv$. Specifically, given a permutation matrix $\bP_k$, recall that $\bV_k = \bP_k \bV_0$, we may use the transformation that under $H_0$, 
\begin{equation}\label{eq:nullintuition1}
	\hat{\bv} = \bV_0^\top \bv = \bV_0^\top \bP_k^\top \bP_k \bv = \bV_k^\top \bP_k \bv.
\end{equation}

In light of this transformation, we have that under $H_0$, $\bV_k \hat{\bv} = \bV_k \bV_k^\top \bP_k \bv =  \proj_{\bV_k}(\bP_k\bv)$, i.e., a projection of the noise vector $\bP_k \bv$ onto the space $\matspan(\bV_k)$, and equivalently, $\bV_0 \hat{\bv} = \proj_{\bV_0}(\bv)$. However, this is still not enough, as $\proj_{\bV_0}(\bv)$ and $\proj_{\bV_k}(\bP_k\bv)$ corresponds to the projections of the vectors $\bv$ and $\bP_k \bv$ onto different subspaces, which are not directly comparable. This means that we need to further propose a more refined strategy to project $\bv$ and $\bP_k \bv$ onto some \emph{same space} for a fair comparison.

Now recall that we already have $\proj_{\bV_0}(\bv)$ and $\proj_{\bV_k}(\bP_k\bv)$, an ideal choice of such space would then be $\matspan(\tilde{\bV}_k)$, i.e., the intersection of $\matspan(\bV_0)$ and $\matspan(\bV_k)$. 
Specifically, using that $\tilde{\bV}_k$ spans a subspace of $\matspan(\bV_k)$, it is straightforward that $\tilde{\bV}_k^\top = \tilde{\bV}_k^\top \bV_k \bV_k^\top$. 
From this and~\eqref{eq:nullintuition1}, we have that under $H_0$,
\[
\tilde{\bV}_k^\top \bV_k \hat{\bv} = \tilde{\bV}_k^\top \bV_k \bV_k^\top \bP_k \bv = \tilde{\bV}_k^\top \bP_k \bv
\]
and equivalently $\tilde{\bV}_k^\top \bV_0 \hat{\bv} = \tilde{\bV}_k^\top \bv$ since $\tilde{\bV}_k$ spans a subspace of $\matspan(\bV_0)$ as well.

%


{\revone 
	In light of the above analysis, we further have that under $H_0$,
	\begin{align*}
		a_k &:= T\left(\tilde{\bs{V}}_k^\top \bs{V}_0 \hat{\bs{e}}, \tilde{\bs{V}}_k^\top \bs{V}_0 \hat{\bs{\varepsilon}}\right) = T\left(\tilde{\bs{V}}_k^\top \bs{V}_0 \hat{\bs{e}}, \tilde{\bs{V}}_k^\top \bs{\varepsilon}\right) \\
		b_k &:= T\left(\tilde{\bs{V}}_k^\top \bs{V}_0 \hat{\bs{e}}, \tilde{\bs{V}}_k^\top \bs{V}_k \hat{\bs{\varepsilon}}\right) = T\left(\tilde{\bs{V}}_k^\top \bs{V}_0 \hat{\bs{e}}, \tilde{\bs{V}}_k^\top \bP_k \bs{\varepsilon}\right).
	\end{align*}
	%
	%
	Writing further that 
	\[
	a^*:=\min_{\ell\in\{1,\ldots,K\}}a_\ell\quad \text{and}\quad b_k^*:= \min_{ \ell \in \{1, \ldots, K\} } T\left(\tilde{\bs{V}}_\ell^\top \bs{V}_0 \hat{\bs{e}}, \tilde{\bs{V}}_\ell^\top \bP_k \bs{\varepsilon}\right),
	\]
	we can control $\pval$ as
	\begin{equation}
		\pval = \frac{1}{K+1} \left(1 + \sum_{k=1}^{K} \one\left\{ a^* \leq  b_k\right\}\right) \geq \frac{1}{K+1} \left(1 + \sum_{k=1}^{K} \one\left\{ a^* \leq b_k^* \right\}\right),\label{Eq:phi}
	\end{equation}
	where for the last inequality we use the fact that $b_k^* \le b_k$. Observe that we may also write $a^*= g(\bv)$ and $b^*_k=g(\bP_k\bv)$ for $g(\boldsymbol{u}):=\min_{\tilde\bV\in\{\tilde \bV_1,\ldots,\tilde \bV_K\}}T(\tilde\bV^\top \bV_0 \hat{\be}, \tilde\bV^\top \boldsymbol{u})$, which is a function that depends only on $(\bX, \bZ, \mathcal{P}_K)$. This allows us to rewrite the above inequality as
	\[
	\pval \ge \frac{1}{K + 1} \left(1+\sum_{k=1}^K \mathbbm{1}\{g(\bs{\varepsilon}) \leq g(\bs{\bs{P}_k\varepsilon})\}\right).
	\]}
Now our only remaining job is to prove that the p-value displayed at the end of the above inequality
is valid. The following lemma, which is a key ingredient in the proof of Theorem~\ref{thm:null}, shows that once we
construct $\cP_K$ such that Assumption~\ref{as:exchangeable} holds, $\pval$ is a valid p-value:

\begin{lemma}\label{lem:exchangeability}
	Suppose $\bs{\varepsilon}$ satisfies Assumption~\ref{as:exchangeable} and let $\{\bs{P}_0 = \bs{I}, \bs{P}_1, \ldots, \bs{P}_K\}$ be a fixed set of permutation matrices satisfying Assumption~\ref{as:permutation}. Then for any function $g : \R^n \to \R$, we have that 
	\[
	\mathbb{P}\biggl\{\frac{1}{K+1}\biggl(1+\sum_{k=1}^K \mathbbm{1}\{g(\bs{\varepsilon}) \leq g(\bs{\bs{P}_k\varepsilon})\}\biggr) \leq \alpha\biggr\} \leq \frac{\lfloor \alpha (K+1)\rfloor}{K+1} \leq \alpha.
	\]
\end{lemma}

{\revone 
	We remark that from the above discussion, a more ideal approach would be to compute $a^*$ and $b^*_k$ for $k=1,\ldots,K$ and construct the test directly using the right-hand side of~\eqref{Eq:phi}. This test, while almost exact, is unfortunately not directly computable from the data. As a compromise, we replace $b_k^*$ with an upper bound $b_k$ to obtain a feasible test $\pval$. As we will see from the numerical simulations, this has resulted in a relatively conservative test.}

\section{Analysis of statistical power}\label{sec:alternative}

This section provides power analysis of RPT under mild moment assumptions of noises $\varepsilon_i$ and $e_i$'s when the second order moments are not necessarily finite. For simplicity of exposition, throughout this section we assume without loss of generality that $n$ is a multiple of $|\mathcal{P}_K| = K+1$, where $K$ is a fixed constant that is chosen such that $K \ge 1 / \alpha$ for the prespecified Type-I error $\alpha$. The scenario where $n$ is not divisible by $K + 1$ can be handled by randomly discarding a subset of data of size at most $K$ to make $n$ divisible. We will focus on the version of RPT defined in~\eqref{eq:pvalest} with $T(x,y) = |\langle x, y\rangle|$. {\revone While we continue to assume that the design matrix $\bX$ is deterministic, $\bZ$ is assumed to have a random design following specific models in this and the next two sections.} Moreover, we are primarily interested in the dependence of the power of RPT on the tail heaviness of the noise distributions. To this end, we make the following assumption on the model:
\begin{assumption}\label{as:iid}
	$\varepsilon_i$'s are i.i.d.\ from some distribution $\prob_\varepsilon$ with mean 0, $\bZ$ follows the model in~\eqref{Eq:Model2} with $e_i$'s i.i.d. from some distribution $\prob_e$ with mean 0. $\bv$ is independent from $\be$.
\end{assumption}
{\revone
	As mentioned in the introduction, some assumption on $\bZ$ is needed for the regression coefficient $b$ to even be identifiable. The structural assumption on $\bZ$ in~\eqref{Eq:Model2} is stronger than typically assumed in the regression coefficient testing literature. This is partly due to the fact that previous power results mostly assume a fixed $p$ regime \citep[e.g.][]{freedman1983nonstochastic, diciccio2017robust}, or asymptotic regimes where $p = O(n^\gamma)$ for some constant $\gamma < 1$~\citep[e.g.][]{mammen1993bootstrap}, see also the references in the Supplementary Material of~\citet{LB21}. On the other hand, when $p\asymp n$, it is not uncommon to see additional structural assumptions on the design matrix. For instance, debiased lasso \citep{ZZ14} assumes a nodewise linear regression structure similar to~\eqref{Eq:Model2} and \citet{lopes2014residual} imposed an eigenvalue decay condition on the sample covariance matrix. In addition, the exact form of model~\eqref{Eq:Model2} is not essential, and is assumed here to simplify our exposition. As we will see later in Section~\ref{sec:alternativehete}, RPT will be asymptotically powerful as long as the quantity defined in~\eqref{eq:detz} is bounded away from zero (Corollary~\ref{cor:zdet}). While the modelling assumption in~\eqref{Eq:Model2} is a sufficient condition for this to hold with asymptotic probability 1, we may relax it to accommodate nonlinear dependence of $\bZ$ on $\bX$ and heteroskedastic noise (Theorems~\ref{thm:althete} and~\ref{thm:altnonlinear}). Even if all these models do not work and $\bZ$ is completely deterministic, Corollary~\ref{cor:zdet} shows that our test is still powerful, provided~\eqref{eq:detz} is large enough, which is an assumption verifiable by practitioners. It would be of interest to propose new tests that have non-trivial power under a nonlinearity assumption weaker than Theorem~\ref{thm:altnonlinear}, which we leave for future work.
}

We also make the following assumption on the permutation matrices $\bP_1, \ldots, \bP_K$.


\begin{assumption}\label{as:trace} For any $k = 1, \ldots, K$, $|\mathrm{tr}[\bV_0 \bV_0^\top \bs{P}_k]| < \sqrt{2 p} K$ and $\tr[\bP_k] = 0$.
\end{assumption}
Notice that when the covariate matrix $\bX$ is of full column rank $p$, Assumption~\ref{as:trace} is equivalent to that $|\mathrm{tr}[\bX (\bX \bX)^{-1} \bX^\top \bs{P}_k]| < \sqrt{2 p} K$. 

In Theorem~\ref{thm:alternative1}, we showcase the pointwise signal detection rate of $\pval$ given any fixed $\prob_\varepsilon$ and $\prob_e$. Moreover, we just require $\prob_\varepsilon$ to have bounded $(1+t)$-th order moment. 

\begin{theorem}\label{thm:alternative1}
	Fix $K \in \mathbb{N}$. Suppose that $(\bs{X},\bs{Z},\bs{Y})$ is generated under model~\eqref{Eq:Model} where $\bs{\varepsilon}$ and $\bZ$ satisfy Assumption~\ref{as:iid} and
	\[
	0 < \E[|e_1|^2] < \infty \quad \text{and} \quad 0 < \E[|\varepsilon_1|^{1 + t}] < \infty
	\]
	for some constant $t \in [0, 1]$. Assume $\mathcal{P}_K$ satisfies Assumption~\ref{as:trace}. In the asymptotic regime where $b$ and $p$ vary with $n$ in a way such that $n > (3+m)p$ for some constant $m > 0$ and
	\begin{align}\label{eq:bn}
		|b| = \Omega (n^{-\frac{t}{1+t}}) \;\textrm{if}\;\; t < 1 \quad\textrm{or}\quad |b| = \omega(n^{-\frac{1}{2}}) \;\textrm{if}\;\; t = 1,
	\end{align}
	we have $\lim_{n \to \infty} \prob\left(\pval > \frac{1}{K + 1}\right) = 0$.
\end{theorem}

Notice that here we need to assume without loss of generality that $\E[e_i^2] > 0$ and $\E[|\varepsilon_1|^{1 + t}] > 0$ to ensure that both two random variables are \emph{not} almost surely equal to zero. Otherwise, $\pval$ is almost surely equal to $1$, and cannot have any statistical power with any size of $b$. Theorem~\ref{thm:alternative1} shows that under certain assumptions on the $\cP_K$, RPT has power to reject the alternative classes even with heavy-tailed noises. Moreover, our analysis {\revone works in a proportional regime where the number of covariates can be} as large as $n / 3$. Remarkably, the statistical power guarantee in Theorem~\ref{thm:alternative1} does not require the $\varepsilon_i$'s to have a bounded second order moment. This distinguishes us from the class of empirical correlation based approaches, such as debiased / desparsified Lasso or OLS fit based tests, which requires $\varepsilon_i$'s to have at least a bounded second order moment or even stronger conditions such as sub-Gaussianity to have statistical power.

As we will see in Section~\ref{sec:algorithm}, Assumption~\ref{as:trace} is a mild condition that can be checked in practice. However, an inspection of the proof of Theorem~\ref{thm:alternative1} reveals that, even if Assumption~\ref{as:trace} does not hold for $\mathcal{P}_K$, RPT is still asymptotically powerful under the same signal strength condition~\eqref{eq:bn} and a slightly stronger requirement on the number of covariates.
Specifically, we require that $n > (4 + m) p$ for some constant $m > 0$ that does not depend on $n$. {\rev In Theorem~\ref{thm:alternative1}, for simplicity we assume that $K$ is a fixed constant. In the Supplementary Material we further provide an extension of Theorem~\ref{thm:alternative1} where we allow $K$ to diverge with $n$. In particular, we show that for $t < 1$, RPT is still guaranteed to have non-trivial power whenever $K = O(n^{\frac{2t}{1 + t}})$.}

In the following theorem, we show that when $p / n \to 0$, we can further relax $e_i$'s finite second order moment condition to a finite first order moment condition.

\begin{theorem}\label{thm:alternative2}
	Fix $K \in \mathbb{N}$. Suppose that $(\bs{X},\bs{Z},\bs{Y})$ is generated under model~\eqref{Eq:Model} where $\bs{\varepsilon}$ and $\bZ$ satisfy Assumption~\ref{as:iid} and
	\[
	0 < \E[|e_1|] < \infty \quad \text{and} \quad 0 < \E[|\varepsilon_1|^{1 + t}] < \infty
	\]
	for some constant $t \in [0, 1]$. In the asymptotic regime where $b$ and $p$ vary with $n$ in a way such that $p / n \to 0$ and $b$ satisfies~\eqref{eq:bn}, 
	$\lim_{n \to \infty} \prob\left(\pval > \frac{1}{K + 1}\right) = 0$.
\end{theorem}

The statistical power guarantee in Theorem~\ref{thm:alternative1} requires the set of permutations to follow Assumption~\ref{as:trace}, whilst the finite-sample size validity requires instead Assumption~\ref{as:permutation}. Then an important question is, how to effectively construct a $\cP_K$ that satisfies both assumptions. In Section~\ref{sec:algorithm}, we provide an algorithm to answer this question. In order to prove Theorems~\ref{thm:alternative1} and~\ref{thm:alternative2}, we are faced with two questions, the first is that we do not have any assumption on $\bX$, so that $\tilde{\bV}_j$ can follow arbitrary pattern; the second is the heavy tails of $e_i$'s and $\varepsilon_i$'s. We defer the proof of the two theorems to the Supplementary Material. 
To help the readers understand the intuitions of the proof, we provide a proof sketch of the main Theorem~\ref{thm:alternative1} in Section~\ref{sec:sketch1}.



\subsection{An algorithm for construction of permutation set}\label{sec:algorithm}

\begin{algorithm}[!t]
	\DontPrintSemicolon
	\caption{\label{alg:perm}Permutation set construction}
	\KwIn{The number of permutation matrices $K$,
		the orthonormal matrix $\bV_0 \in \R^{n \times (n - 2p)}$ such that $\bV_0^\top \bX = 0$, the maximum number of loops $T$}
	
	
	\Repeat{(i) $|\tr[\bV_0\bV_0^\top \bP_k]| \le \sqrt{2} Kp^{1/2}$ for all $k=1,\ldots,K$ or (ii) the number of iterations has reached its limit $T$}{
		Generate an independent random permutation $\pi$ of indices $\{1,\ldots,n\}$\;
		\For{$k=1,\ldots,K$}{
			{\rev Construct a permutation function $\sigma_k := \pi^{-1} \circ \tilde{\sigma}_k \circ \pi$, where $\circ$ denotes a composition of two functions and $\tilde{\sigma}_k$ is a permutation function such that
				\begin{equation}\label{eq:cycperm}
					\tilde{\sigma}_k(i) := \left\{\begin{aligned}
						& i + k & & \quad\textrm{if}\; i \bmod (K + 1) \le K + 1 - k  \\
						& i - (K + 1 - k) & & \quad \textrm{otherwise},
					\end{aligned}\right.
				\end{equation}
				and set $\bP_k$ as the permutation matrix corresponding to $\sigma_k$.}
		}
	}
	\KwOut{Set of permutation matrices $\mathcal{P}_K := \{\bP_0 := \bs{I}, \bs{P}_1,\ldots,\bs{P}_K\}$ satisfying the criteria (i). When none of the $\cP_K$'s comply, report the $\cP_K$ with the smallest $\sum_{k=1}^K |\tr[\bV_0 \bV_0^\top \bP_k]|$.}
\end{algorithm}


As demonstrated in Theorems~\ref{thm:null} and~\ref{thm:alternative1}, to successfully perform a test that is valid under the null and has sufficient statistical power to get the rate in~\eqref{eq:bn} {\rev when $n / p > 3 + m$ for some constant $m > 0$}, one needs a set of permutations satisfying both Assumptions~\ref{as:permutation} and~\ref{as:trace}. As demonstrated in Proposition~\ref{prop:algorithm} below, such permutation set always exist, so that we can at least apply a brute-forth algorithm to find a desired set. To improve computational efficiency, we further develop a randomized algorithm that can discover the desired permutation set with high probability (Algorithm~\ref{alg:perm}). {\rev To understand this algorithm, notice that if we are just interested in Assumption~\ref{as:permutation}, one simple way is to divide the $n$ indices into $m := n / (K + 1)$ ordered list of indices and perform cyclic permutation on each sub-list. Specifically, we first denote $S_1, \ldots, S_m$ as $m$ ordered list of indices such that
	\[
	(1, \ldots, n) := (\underset{S_1}{\underbrace{1, \ldots, K + 1}}, \; \underset{S_2}{\underbrace{K + 2, \ldots, 2(K + 1)}}, \; \underset{S_m}{\underbrace{(m - 1)(K + 1) + 1, \ldots, m (K + 1)}}),
	\]
	Then we define the $\tilde{\bP}_k$ for $k \ge 1$ (or equivalently its permutation function $\tilde{\sigma}_k$) as that
	\[
	(\tilde{\sigma}_k(1), \ldots, \tilde{\sigma}_k(n)) := (S_1^k, \ldots, S_m^k),
	\]
	where each $S_i^k$ is created via shifting all the elements in $S_i$ by $k$ places. Taking $S_1^k$ for example, it means $S_1^k := (K + 2 - k, \ldots, K + 1, 1, 2, \ldots, K + 1 - k)$. One can easily verify that the resulting set of permutation matrices $\tilde\cP_K := \{\bs{I}, \tilde\bP_1, \ldots, \tilde\bP_K\}$ satisfies Assumption~\ref{as:permutation} since it is constructed by cyclic permutations\footnote{Notice that the ``$\tilde{\sigma}_k$'' described here is exactly the same as the ``$\tilde{\sigma}_k$'' in~\eqref{eq:cycperm}}. However, since $\tilde\cP_K$ is blind of $\bX$, Assumption~\ref{as:trace} may not hold. To overcome this challenge, in Algorithm~\ref{alg:perm} we apply an iterative algorithm where in each round, we set $\sigma_k := \pi^{-1} \circ \tilde{\sigma}_k \circ \pi$ for some random permutation $\pi$ and loop until it reaches the number of rounds limit or the resulting $\cP_K$ satisfies Assumption~\ref{as:trace} (Step~6). This allows Algorithm~\ref{alg:perm} to still preserve Assumption~\ref{as:permutation}, while being more adaptive to $\bX$. 
}
In Proposition~\ref{prop:algorithm}, we show that after doing $T$-th round of such iterations, Algorithm~\ref{alg:perm} is able to deliver a $\cP_K$ satisfying the desired properties with probability at least $1 - \frac{1}{K^T}$.

\begin{proposition}\label{prop:algorithm}
	Given $K, T$, we have that there exists a $\mathcal{P}_K$ satisfying Assumptions~\ref{as:permutation} and~\ref{as:trace}. Moreover, {\revone Algorithm~\ref{alg:perm} always returns a $\mathcal{P}_K$ that satisfies Assumption~\ref{as:permutation}; and with probability at least $1 - \frac{1}{K^T}$, the returned $\mathcal{P}_K$ also satisfies Assumption~\ref{as:trace}.}
\end{proposition}

Notice that throughout this article, we assume that the alternative class is in the form $\bY = \bX \beta + b \bZ + \bv$ for some $b \neq 0$, whence we invoke Assumption~\ref{as:trace} to increase its statistical power. When the alternative class follows other forms, such as $\bY = \bX \beta +  f(\bZ) + \bv$ with some nonlinear function $f : \R^n \mapsto \R^n$, one may not necessarily need Assumption~\ref{as:trace} anymore. 
Instead, one may need other assumptions on $\cP_K$ to adapt to the nonlinear function $f(\cdot)$. In light of Algorithm~\ref{alg:perm} and our theoretical statements, we summarize an implementation of RPT in Algorithm~\ref{alg:test}. {\rev The maximum time complexities of Algorithms~\ref{alg:perm} and~\ref{alg:test} are {\revone $O(T K np)$ and $O(T Knp + K np^2)$} respectively, where $T$ is the maximum number of iterations. The expected time complexities of the two algorithms are instead {\revone $O(K np)$ and $O(K np^2)$}, respectively.} We also remark that due to the construction of the permutation set in Algorithm~\ref{alg:perm}, RPT is inherently a randomized procedure, and unlike permutation tests or bootstrapping, this variability due to randomness cannot be reduced by increasing computational time. It would be of interest to propose a new test that eradicates such reproducibility issue while maintaining all the beneficial features of Algorithm~\ref{alg:perm}, which we leave for future work.

\begin{algorithm}[!t]
	\DontPrintSemicolon
	\caption{\label{alg:test}Residual Permutation Test (RPT)}
	\KwIn{design matrix $\bs{X}\in\mathbb{R}^{n\times p}$, additional covariate of of interest $\bs{Z}\in\mathbb{R}^n$, response vector $\bs{Y}\in\mathbb{R}^n$, number of permutations $K\in\mathbb{N}$, maximal number of iterations $T \in \N$.}
	
	Find an orthonormal matrix $\bs{V}_0\in\mathbb{R}^{n\times (n-p)}$ such that  $\bs{V}_0^\top \bs{X} = 0$.\;
	
	Apply Algorithm~\ref{alg:perm} with inputs $K, T$ and $\bV_0$ to generate $K$ permutation matrices $\{\bs{P}_1,\ldots,\bs{P}_K\}$.
	
	\For{$k = 1,\ldots, K$}{
		Set $\bs{V}_k := \bs{P}_k \bs{V}_0$.\;
		Find an orthonormal matrix $\tilde{\bs{V}}_k \in\mathbb{R}^{n\times (n-2p)}$ such that $\matspan(\tilde{\bs{V}}_k) \subseteq \matspan(\bV_0) \cap \matspan(\bV_k)$.
		\;
		Compute 
		\[
		a_k := \left|\langle \tilde{\bs{V}}_k^\top \bZ, \tilde{\bs{V}}_k^\top \bY\rangle\right| \quad \textrm{and} \quad b_k :=\left|\langle \tilde{\bs{V}}_k^\top \bZ, \tilde{\bs{V}}_k^\top \bP_k \bY\rangle\right|,
		\]
		where $<\cdot, \cdot>$ denotes the inner product.
	}
	\KwOut{p-value $\pval:= \frac{1}{K+1}(1 + \sum_{k=1}^K \mathbbm{1}\{\min_{1\leq j\leq K} a_j \leq b_k\}) $\;}
\end{algorithm}



\subsection{Proof sketch of Theorem~\ref{thm:alternative1}}\label{sec:sketch1}

As $K$ is finite, we mainly need to prove that for any fixed $\bs{P}_j, \bs{P}_k \in \mathcal{P}_K$, with probability converging to $1$, $|\hat{\bs{e}}^\top \bs{V}_0^\top \tilde{\bs{V}}_j \tilde{\bs{V}}_j^\top \bs{V}_0 \hat{\bs{\varepsilon}}| > |\hat{\bs{e}}^\top \bs{V}_0^\top \tilde{\bs{V}}_k \tilde{\bs{V}}_k^\top \bs{V}_k \hat{\bs{\varepsilon}}|$.
To achieve this goal, we need to prove that
\begin{equation}\label{eq:noisecor}
	\frac{|\be^\top \tilde{\bs{V}}_j \tilde{\bs{V}}_j^\top \bv|}{b n} = o_\pr(1)
\end{equation}
(i.e., that the empirical correlation between the projection of $\be$ and $\bv$ onto the space spanned by $\tilde{\bs{V}}_j$ is negligible with high probability) and that with high probability,
\begin{equation}\label{eq:signal}
	\frac{\be^\top  \tilde {\bs{V}}_j \tilde {\bs{V}}_j^\top  \be - \be^\top  \tilde {\bs{V}}_k \tilde {\bs{V}}_k^\top  \bP_k \be}{n} \gtrsim 1 \quad \text{and} \quad \frac{\be^\top  \tilde {\bs{V}}_j \tilde {\bs{V}}_j^\top  \be + \be^\top  \tilde {\bs{V}}_k \tilde {\bs{V}}_k^\top  \bP_k \be}{n} \gtrsim 1.
\end{equation}

To prove \eqref{eq:noisecor}, when $t = 1$, the result is straightforward from Chebyshev's inequality; hence the main challenge is to prove the case $t \in [0, 1)$. {\rev In Corollary~\ref{cor:convergence}, we establish a more general result, which characterizes the stochastic convergence of $|\bw^\top \bv|$ where $\bw$ is an arbitrary deterministic vector and $\bv$ can be heteroscedastic. We refer the readers to Section~\ref{sec:alternativehete} for its statement as well as the intuitions for its proof.}




Thanks to the bounded second order moment of $e_i$'s, the analysis of \eqref{eq:signal} is simpler. Specially, by using a variant of weak law of large number we develop in this paper to control the weighted sum of $e_i^2$'s and a Chebyshev's inequality to control the sum of cross terms $e_i e_j$'s, we can have that with probability converging to $1$,
\[
\frac{\be^\top  \tilde {\bs{V}}_j \tilde {\bs{V}}_j^\top  \be - \be^\top  \tilde {\bs{V}}_k \tilde {\bs{V}}_k^\top  \bP_k \be}{n} \gtrsim \frac{n - 3p - \mathrm{tr}[\bX (\bX^\top \bX)^{-1} \bX^\top \bs{P}_k]}{n}.
\]
Using that $\bP_k$ satisfies Assumption~\ref{as:trace}, we easily obtain the desired result.

{\rev 
	\section{Statistical power under broader classes of alternatives}\label{sec:alternativehete}
	
	In Theorems~\ref{thm:alternative1} and~\ref{thm:alternative2}, for simplicity of illustration, we consider the class of alternative hypotheses where $\bZ$ is a linear model and all the noises are i.i.d. In this section, we consider two relaxations of these assumptions. First, we assume that $\bZ$ follows a linear model with respect to the covariates and all noises are heteroscedastic; second, we allow $\bZ$ to have some nonlinearity, at the cost of slightly more restrictions on the degree of heteroscedasticity of $\varepsilon_i$'s. 
	
	\begin{assumption}\label{as:hete}
		$\bZ$ follows the model in~\eqref{Eq:Model2}; the random vectors $\bv$ and $\be$ are first $n$ components of two independent infinite sequences of independent zero-mean random variables $\varepsilon_1, \varepsilon_2, \ldots$ and $e_1, e_2, \ldots$, respectively. Suppose also that 
		\begin{itemize}
			\item for some universal constants $C_e, c_e > 0$, we have $\mathbb{E}[e_i^2] \leq C_e$ for all $1 \le i < \infty$, and 
			\begin{equation}
				\label{eq:ashete1}
				\lim_{a \to \infty}\sup_{n \ge 1}  \frac{1}{n}\sum_{i=1}^n \E[ e_i^2 \one(e_i^2 \ge a)] = 0 \quad \textrm{and}\quad \liminf_{n \to \infty} \frac{1}{n} \sum_{i=1}^n \E[e_i^2] > c_e.
			\end{equation}
			\item for some fixed $t\in[0,1]$ and some universal constant $C_\varepsilon > 0$, we have $\mathbb{E}[|\varepsilon_i|^{1+t}] \leq C_{\varepsilon}$ for all $i$ and given any fixed $B > 0$,
			\begin{equation}
				\label{eq:ashete2}
				\sum_{i=1}^\infty \pr\left(|\varepsilon_i|^{1 + t} \ge B i\right) < \infty.
			\end{equation}
		\end{itemize}
	\end{assumption}
	
	Informally speaking, instead of requiring all noises to be i.i.d., Assumption~\ref{as:hete} allows noises to be heteroscedastic, under certain restriction on the degree of heteroscedasticity of $\varepsilon_i$'s and $e_i$'s. To intuitively understand~\eqref{eq:ashete2} and the first equation in~\eqref{eq:ashete1}, taking \eqref{eq:ashete2} for example, a sufficient condition for it to hold is that there exists a zero-meaned random variable $\varepsilon_\infty$ satisfying that $\E[|\varepsilon_\infty|^{1 + t}] < \infty$ and that for any $1 \le i < \infty$, $|\varepsilon_i| \preceq_d |\varepsilon_\infty|$, i.e., that $|\varepsilon_i|$ is stochastically dominanted by $|\varepsilon_\infty|$ uniformly for all $\varepsilon_i$'s. When such $\varepsilon_\infty$ exists, for any $n \ge 1$, 
	\begin{align*}
		\sum_{i=1}^n \pr\left(|\varepsilon_i|^{1 + t} \ge B i\right) & \le \sum_{i=1}^n \pr\left(|\varepsilon_\infty|^{1 + t} \ge B i\right) \\
		& \le \int_0^\infty \pr\left(\frac{|\varepsilon_\infty|^{1 + t}}{B} \ge x\right) dx = \E\left[\frac{|\varepsilon_\infty|^{1 + t}}{B}\right] < \infty,
	\end{align*}
	which satisfies~\eqref{eq:ashete2}. Analogously, when there exists a zero-meaned random variable $e_\infty$ with $\E[|e_\infty|^2] < \infty$ and $|e_\infty|$ stochastically dominates all $|e_i|$'s, we can also have
	\[
	\sup_{n \ge 1} \frac{1}{n} \sum_{i = 1}^n \E[e_i^2 \one(e_i^2 \ge a)] \le \E[e_\infty^2 \one(e_\infty^2 \ge a)],
	\]
	which, from dominated convergence theorem, converges to zero as $a \to \infty$. Armed with Assumption~\ref{as:hete}, we have the following theorem on the power of RPT.
	
	\begin{theorem}\label{thm:althete}
		Fix $K\in\mathbb{N}$. Assume that $(\bs{X},\bs{Z},\bs{Y})$ is generated under model~\eqref{Eq:Model} with $\bv$ and $\bZ$ satisfying Assumption~\ref{as:hete}; $\mathcal{P}_K$ satisfies Assumption~\ref{as:trace}. In the asymptotic regime where $b$ and $p$ vary with $n$ in a way such that $n > (3 C_e/ c_e +m)p$ for some constant $m > 0$ and $b$ satisfies~\eqref{eq:bn},
		we have $\lim_{n \to \infty} \prob\left(\pval > \frac{1}{K + 1}\right) = 0$.
	\end{theorem}
	
	In the following, we show that when we are willing to impose slightly more restrictions on the degree of heterogeneity of $\varepsilon_i$'s, we can still maintain the $n^{-\frac{t}{1 + t}}$ rate even when the expectation of $\bZ$ cannot be viewed as a linear function of $\bX$.
	
	\begin{assumption}
		\label{as:nonlinear}
		$\bZ$ is generated according to $\bZ = \bX \beta^Z + \bs{h} + \be$,
		where $\bs{h}$ is an $n$-dimensional deterministic vector; $\bv$ and $\be$ follow the same assumptions as the $\bv$ and $\be$ in Assumption~\ref{as:hete}, with the addition that
		\[
		\lim_{a \to \infty}\sup_{i \ge 1} \E[|\varepsilon_i|^{1 + t} \one(|\varepsilon_i|^{1 + t} > a)] = 0.
		\]
	\end{assumption}
	In Assumption~\ref{as:nonlinear}, to alleviate the linearity requirement of $\bZ$, we introduce an additional
	uniform constraint concerning the tails of $\varepsilon_i$'s. It is worth noting that this new condition is satisfied when there exists a $\varepsilon_\infty$ with $\E[|\varepsilon_\infty|^{1 + t}] < \infty$ that stochastically dominates all $\varepsilon_i$'s. Specifically, when such $\varepsilon_\infty$ exists, then
	\[
	\sup_{i \ge 1} \E[|\varepsilon_i|^{1 + t} \one(|\varepsilon_i|^{1 + t} > a)] \le \E[|\varepsilon_\infty|^{1 + t} \one(|\varepsilon_\infty|^{1 + t} > a)],
	\]
	which, from dominated convergence theorem, converges to zero as $a \to \infty$.
	
	\begin{theorem}\label{thm:altnonlinear}
		Fix $K\in\mathbb{N}$. Assume that $(\bs{X},\bs{Z},\bs{Y})$ is generated under model~\eqref{Eq:Model} with $\bv$ and $\bZ$ satisfying Assumption~\ref{as:nonlinear}; $\mathcal{P}_K$ satisfies Assumption~\ref{as:trace}. In the asymptotic regime where $b, p$ and $\bs{h}$ vary with $n$ in a way such that for some constants $m, r$ with $m > 0, r < c_e$, $\limsup_{n \to \infty} \|\bs{h}\|_2^2 / n \le r$, $n > (3C_e / (c_e - r) +m)p$ and $b$ satisfies~\eqref{eq:bn}, 
		we have $\lim_{n\to\infty}\mathbb{P}\bigl(\phi > \frac{1}{K+1}\bigr) = 0$. 
	\end{theorem}
	
	{\revone When $\mathcal{P}_K$ does not satisfy Assumption~\ref{as:trace}, the same conclusion in Theorem~\ref{thm:altnonlinear} still holds with $n > (4C_e / (c_e - r) +m)p$; see also the analogous comment after Theorem~\ref{thm:alternative1}. Notice also that} when $\bZ$ and $\bP_1, \ldots, \bP_K$ are all deterministic and we keep the data generating model of $\bY$ as~\eqref{Eq:Model}, following an analysis analogous to the proof of Theorem~\ref{thm:altnonlinear}, we can prove that $\pval$ is still asymptotically powerful whenever
	\begin{equation}\label{eq:bnz}
		|b| = \Omega (z_n^{-1} n^{-\frac{t}{1+t}}) \;\textrm{if}\;\; t < 1 \quad\textrm{or}\quad |b| = \omega(z_n^{-1} n^{-\frac{1}{2}}) \;\textrm{if}\;\; t = 1,
	\end{equation}
	where 
	\begin{equation}\label{eq:detz}
		z_n := \left(\frac{\|\bV_0^\top \bZ\|_2}{\sqrt{n}} \right)^{-1} \cdot \min_{1 \le j, k \le K} \min_{z \in \{0, 1\}} \frac{\bZ^\top \tilde{\bV}_j \tilde{\bV}_j^\top \bZ + (-1)^z \bZ^\top \tilde{\bV}_k \tilde{\bV}_k^\top \bP_k \bZ}{n}.
	\end{equation}
	In other words, $\bZ$ does not necessarily need to be random for RPT to have power. 
	To formally describe the above intuition, we have the following corollary.
	\begin{corollary}\label{cor:zdet}
		Fix $K\in\mathbb{N}$. Assume that $(\bs{X},\bs{Z},\bs{Y})$ is generated under model~\eqref{Eq:Model} with $\bv$ as in Assumption~\ref{as:nonlinear} and $p < n / 2$. $\bZ, \cP_K$ are deterministic such that $\|\bV_0^\top \bZ\|_2 > 0$ and $z_n > 0$ uniformly for all $n \ge 3$. In the asymptotic regime where $b$ satisfies~\eqref{eq:bnz}, we have $\lim_{n\to\infty}\mathbb{P}\bigl(\phi > \frac{1}{K+1}\bigr) = 0$.
	\end{corollary}

	An inspection of the proof of Theorem~\ref{thm:altnonlinear} reveals that,  when $\bZ$ satisfies the random model as prescribed in Assumption~\ref{as:nonlinear} and 
	$(n, p)$ is as in Theorem~\ref{thm:altnonlinear}, with probability converging to $1$, $z_n \asymp 1$, and the scale delivered by~\eqref{eq:bnz} and~\eqref{eq:bn} coincide. In practice, one can choose $\bP_1, \ldots, \bP_K$ to maximize~\eqref{eq:detz}.
	
	
	In order to prove Theorem~\ref{thm:altnonlinear}, one needs to understand the rate of convergence of the term $|\bs{h}^\top \tilde{\bV}_j \tilde{\bV}_j^\top \bv|$. Based on our analysis of this term in the proof of Theorem~\ref{thm:altnonlinear}, it is straightforward to get the following corollary, which characterizes the rate of convergence of $\bw^\top \bv$ for arbitrary deterministic $n$-dimensional vector $\bw$, which we believe is of independent interest:
	
	\begin{corollary}\label{cor:convergence}
		Consider the $\bv$ as in Assumption~\ref{as:nonlinear} with $t \in [0, 1)$. Then for any fixed constant $\delta > 0$,
		\[
		\lim_{n \to \infty} \sup_{\bw \in \mathcal{S}^{n - 1}} \pr\left( |\bw^\top \bv| > \delta n^{\frac{1 - t}{2(1 + t)}}\right) = 0,
		\]
		where $\mathcal{S}^{n - 1} := \{\bw \in \R^n: \|\bw\|_2 = 1\}$ is the $(n - 1)$-sphere in the $n$-dimensional Euclidean space.
	\end{corollary}
	Informally then, Corollary~\ref{cor:convergence} means that $|\bw^\top \bv| = o_\pr(\|\bw\|_2 n^{\frac{1 - t}{2(1 + t)}})$ for any choice of $n$-dimensional unit vector $\bw$. For example, one can even allow $\max_{1 \le i \le n} |w_i| / \|\bw\|_2 \asymp 1$. This enables us to prove the rate of convergence of $\bs{h} \tilde{\bV}_j \tilde{\bV}_j^\top \bv$ without any regularity condition on $\bX$ or $\bs{h}$. 
	
	To prove Corollary~\ref{cor:convergence} (or equivalently to find the rate of convergence of $\bs{h}^\top \tilde{\bV}_j \tilde{\bV}_j^\top \bv$),
	the main challenge is to deal with the heavy-tailedness of $\varepsilon_i$'s. We apply a truncation $f_i := \varepsilon_i \one(|\varepsilon_i| \ge B i)$ and seek to control $\bw^\top \bs{f}$ and $\bw^\top (\bv - \bs{f})$ separately, where for simplicity we write $\bs{f} := (f_1, \ldots, f_n)^\top$. We seek to control $\bw^\top \bs{f}$ via proving the following two convergence results (see the Supplementary Material for its proof): 
	\begin{itemize}
		\item For any fixed $B > 0$, $\sup_{\bw \in \mathcal{S}^{n - 1}}\E[|\bw^\top (\bs{f} - \E[\bs{f}])|^2] = o(n^{\frac{1 - t}{1 + t}})$;
		\item As $B \to \infty$, we have that $\sup_{n \ge 1} \|\E[\bs{f}]\|_2^2 / n^{\frac{1 - t}{1 + t}} \to 0$ (notice that here $\|\E[\bs{f}]\|_2^2$ is a function of $B$).
	\end{itemize}
	With the above results, it is straightforward that for any constant $\delta > 0$, by choosing the constant $B_\delta > 0$ sufficiently large, uniformly for all $n \ge 2$,
	\[
	\sup_{\bw \in \mathcal{S}^{n - 1}} |\bw^\top \E[\bs{f}_\delta]| \le \|\E[\bs{f}_\delta]\|_2 \le \frac{\delta}{2} \cdot n^{\frac{1 - t}{2(1 + t)}},
	\]
	where we rewrite $\bs{f}$ as $\bs{f}_\delta$ to emphasize its dependence on $B_\delta$. Moreover, by Chebyshev's inequality, we further have from the above convergence results that as $n \to \infty$,
	\begin{equation}\label{eq:a1sketch}
		\sup_{\bw \in \mathcal{S}^{n - 1}}\pr\left(|\bw^\top (\bs{f} - \E[\bs{f}])| > \frac{\delta}{2} \cdot n^{\frac{1 - t}{2(1 + t)}}\right) \to 0.
	\end{equation}
	Taking together, we control $\bw \bs{f}_\delta$; and our only remaining job is to control the convergence of $\bw^\top (\bv - \bs{f}_\delta)$, which we prove by an argument similar in spirit to Borel-Cantelli Lemma~\citep{Durrett19}.
}


\section{Minimax {\rev rate} optimality of coefficient tests}\label{sec:minimax}

In this section, we investigate the minimax {\rev rate} optimality of RPT by deriving the statistical efficiency limit of coefficient tests with heavy-tailed noises. Without loss of generality, we denote $\cD_t$ as the class of distributions with $t$-th order moment bounded between $[1, 2]$, i.e., for some $t > 0$ and some random variable $\xi$ with distribution $\prob_{\xi}$, 
\[
\prob_\xi \in \cD_t \quad \textrm{iff} \quad \E[\xi] = 0 \;\text{and}\; 1 \le \E[|\xi|^t] \le 2. 
\]
Notice that in the above definition, the thresholds $1$ and $2$ are chosen for notational simplicity, in fact, the general conclusions in this section still hold for $\eta_1 \le \E[|\xi|^t] \le \eta_2$ with arbitrary $\eta_1, \eta_2 > 0$. We further let $\tilde{\cD}$ denote the class of distributions such that
\[
\prob_\xi \in \tilde{\cD} \quad \textrm{iff} \quad \prob\left(|\xi| > \frac{1}{2}\right) > \frac{1}{2}.
\]
With a slight abuse of notation, given $b_0 \in \R$, we write $\prob_{b_0}$ as a distribution of $(\bY, \bZ)$ such that the $b$ in~\eqref{Eq:Model} is equal to $b_0$. Note that we have suppressed the dependence of $\mathbb{P}_{b_0}$ on  $\bX, \beta,\beta^Z,\mathbb{P}_\varepsilon$ and $\mathbb{P}_e$ for notational simplicity. In particular, $\prob_0$ corresponds to the null hypothesis.

From above, we define the minimax testing risk indexed by $t, \bX$ as
\begin{align*}
	& \mathcal{R}_{t, \bX}(\tau) := \\
	& \quad \inf_{\varphi \in \Phi} \left\{\sup_{\prob_\varepsilon \in \cD_t} \sup_{\prob_e \in \cD_1 \cap \tilde{\cD}} \sup_{\beta, \beta^Z \in \R^p} \prob_0(\varphi = 1) + \sup_{|b| \ge \tau}\sup_{\prob_\varepsilon \in \cD_t} \sup_{\prob_e \in \cD_1 \cap \tilde{\cD}} \sup_{\beta, \beta^Z \in \R^p} \prob_b(\varphi = 0)\right\}.
\end{align*}
Here $\Phi$ corresponds to the class of measurable functions of data $(\bX, \bZ, \bY)$ taking value in $\{0,1\}$. 
We first establish the following nonasymptotic minimax lower bound for testing $H_0:b=0$ against $H_1:b\neq 0$ in the presence of heavy-tailed noises. 

\begin{theorem}\label{thm:lower}
	Let $t \in [0,1]$ be given and assume that $\bs{\varepsilon}$ and $\bs{e}$ satisfy  Assumption~\ref{as:iid}. For any $\eta \in (0,1)$, there exists a constant $c_\eta > 0$ depending only on $\eta$ such that for any fixed design $\bs{X}$,
	\[
	\mathcal{R}_{1 + t, \bX}\left(c_\eta n^{-\frac{t}{1 + t}}\right) \ge 1 - \eta.
	\]
\end{theorem}
Theorem~\ref{thm:lower} shows that when entries of $\bs{\varepsilon}$ have finite $(1+t)$-th moment, the minimax separation rate in $b$ for testing $H_0$ against $H_1$ is at least of order $n^{-\frac{t}{1+ t}}$, which matches the upper bound in~Theorem~\ref{thm:alternative1}.
This indicates that the rate $n^{-\frac{t}{1+ t}}$ may be a tight lower bound, and that our constructed test may be an rate optimal test. 
Nevertheless, Theorems~\ref{thm:alternative1} and~\ref{thm:alternative2} are pointwise convergence results, where both $\prob_\xi$ and $\prob_e$ are considered as fixed and does not depend on $n, p$. 
To match the lower bound in Theorem~\ref{thm:lower}, we further provide a power control of RPT uniformly over classes of noise distributions of $\prob_\varepsilon$ and $\prob_e$. Just as in Section~\ref{sec:alternative}, we assume without loss of generality that $n$ is divisible by $K + 1$.

\begin{theorem}\label{thm:supupp}
	Fix $K \in \N$. Assume that $(\bs{X},\bs{Z},\bs{Y})$ is generated under model~\eqref{Eq:Model} with $\bv$ and $\bZ$ satisfying Assumption~\ref{as:iid} and that $\mathcal{P}_K$ satisfies Assumption~\ref{as:trace}. In an asymptotic regime where $b$ and $p$ vary with $n$ in a way such that $n > (3 + m) p$ for some constant $m > 0$ and $|b| = \Omega(n^{-\frac{t}{1 + t} + \delta})$ for some constants $t \in (0, 1]$ and $\delta > 0$, we have for any constant $\nu > 0$ that,
	\begin{equation}\label{eq:upp_sup_high}
		\lim_{n \to \infty} \sup_{\prob_\varepsilon \in \cD_{1+t}} \sup_{\prob_e \in \cD_{2+\nu}} \prob\left(\pval > \frac{1}{K+1}\right) = 0.
	\end{equation}
	If we drop Assumption~\ref{as:trace} and instead assume $p / n \to 0$, 
	then we have for any constant $\nu > 0$ that,
	\begin{align}\label{eq:upp_sup_low}
		\lim_{n \to \infty}  \sup_{\prob_\varepsilon \in \cD_{1+t}} \sup_{\prob_e \in \cD_{1+\nu} \cap \tilde{\cD}} \prob\left(\pval > \frac{1}{K + 1}\right) = 0. 
	\end{align}
\end{theorem}

In Theorem~\ref{thm:supupp}, the separation rate is slightly worse than~\eqref{eq:bn} by a factor of $n^\delta$, where $\delta$ can be any positive constant. Also, it is slightly worse than the lower bound in Theorem~\ref{thm:lower}. This shows that the separation rate $n^{-\frac{t}{1 + t}}$ is a nearly optimal rate of coefficient testing in the minimax sense. At the same time, it also shows that our residual permutation test is a nearly rate-optimal hypothesis test in the minimax sense.

\section{Numerical studies}\label{sec:numerical}


\subsection{Experimental setups}
\label{Sec:SimSetup}
In this section, we evaluate the performance of RPT, together with several competitors, in the following synthetic datasets. The observations $(\bs{X}, \bs{Y}, \bs{Z})\in\mathbb{R}^{n\times p}\times\mathbb{R}^n\times\mathbb{R}^n$ are generated according to the models~\eqref{Eq:Model} and~\eqref{Eq:Model2}
where 
\begin{itemize}
	\item $\bX$ is generated with i.i.d.\ entries from either $\cN(0,1)$ or $t_1$ distribution;
	\item $\beta$ and $\beta^Z$ are $p$-dimensional vectors with the first $5$ components sampled uniformly on the sphere $\mathcal{S}^{4}$ and the rest of the components equal to $0$;
	\item $\be$ and $\bv$ have independent and identically distributed components drawn from $\cN(0,1), t_1$ or $t_2$.
\end{itemize}
We vary $n\in\{300, 600\}$, $p\in \{100, 200\}$ and $b$ in different simulation experiments. 

In practice, we find that the p-value calculated by Algorithm~\ref{alg:test} is slightly on the conservative side. Hence, in addition to the test with p-value constructed by Algorithm~\ref{alg:test}, we also study a variant in our numerical experiments, where the p-value is computed as
$\frac{1}{K+1}(1+\sum_{k=1}^K\mathbbm{1}{\{a_k\leq b_k\}})$ instead (we call this variant as $\textrm{RPT}_{\textrm{EM}}$, where ``EM'' stands for empirical). To benchmark the performance of RPT and $\textrm{RPT}_{\textrm{EM}}$, we also look at the naive residual permutation test in~\eqref{eq:naive}. Other tests used for comparison include 
the ANOVA test described in the introduction, the robust permutation test by \citet{diciccio2017robust} (DR), the residual bootstrap method of \citet{freedman1981bootstrapping} (RB), the residual permutation approach of \citet{freedman1983nonstochastic} (FL), the conditional randomization test (CRT) of \citet{candes2018panning}, the residual randomization (RR) procedure of \citet{toulis2019invariant}, the desparsified Lasso  as implemented in the \texttt{hdi} R package (HDI) \citep{dezeure2015high} and the cyclic permutation test of \citet{LB21} (CPT).


We note that RPT relies on tuning parameters $K$ and $T$. For a test to have a size of $\alpha$, we need to have $K+1$ at least $1/\alpha$. We suggest using $K + 1= \lceil 1/\alpha\rceil$ in practice, though empirical simulation results suggest that our method is robust to the choice of $K$. We also set $T = 1$ to boost the computational efficiency of Algorithm~\ref{alg:perm}.

\subsection{Numeric analysis of validity under the null}\label{sec:simnull}
\begin{table}[t!]
	\centering
	\begin{tabular}{cccc|cc|cc|cc|cc|cc}
		\hline\hline
		& & & & \multicolumn{2}{c}{$\textrm{RPT}_{\textrm{EM}}$} & \multicolumn{2}{c}{RPT} & \multicolumn{2}{c}{DR} & \multicolumn{2}{c}{FL} & \multicolumn{2}{c}{CRT}\\
		\hline
		$n$ & $p$ & $\bX$ & noise & 1\% & 5\% & 1\% & 5\% & 1\% & 5\% & 1\% & 5\% & 1\% & 5\%\\
		\hline
		$300$ & $100$ & $\mathcal{G}$ & $\mathcal{G}$ & $0$ & $0.08$ & $0$ & $0$ & $0.98$ & $5.04$ & $0.99$ & $5.02$ & $0$ & $0.05$\\
		$300$ & $100$ & $\mathcal{G}$ & $t_1$ & $0.51$ & $1.12$ & $0.24$ & $0.5$ & $0.88$ & $4.66$ & $1.28$ & $3.8$ & $1.89$ & $3.03$\\
		$300$ & $100$ & $\mathcal{G}$ & $t_2$ & $0.14$ & $0.45$ & $0.04$ & $0.09$ & $0.67$ & $3.88$ & $1.23$ & $4.91$ & $0.53$ & $1.86$\\
		$300$ & $100$ & $t_1$ & $\mathcal{G}$ & $0$ & $0.09$ & $0$ & $0$ & $3.33$ & $9.02$ & $1.01$ & $4.99$ & $0$ & $0$\\
		$300$ & $100$ & $t_1$ & $t_1$ & $0.01$ & $0.23$ & $0$ & $0$ & $1.28$ & $5.67$ & $1.21$ & $4.35$ & $0.33$ & $0.59$\\
		$300$ & $100$ & $t_1$ & $t_2$ & $0$ & $0.09$ & $0$ & $0$ & $2.54$ & $8.45$ & $1.09$ & $5$ & $0$ & $0.01$\\
		$600$ & $100$ & $\mathcal{G}$ & $\mathcal{G}$ & $0.21$ & $1.77$ & $0.01$ & $0.06$ & $0.95$ & $4.91$ & $0.95$ & $4.91$ & $0$ & $0.04$\\
		$600$ & $100$ & $\mathcal{G}$ & $t_1$ & $0.73$ & $2.31$ & $0.48$ & $1.24$ & $0.92$ & $4.77$ & $1.09$ & $3.82$ & $1.68$ & $2.47$\\
		$600$ & $100$ & $\mathcal{G}$ & $t_2$ & $0.61$ & $2.29$ & $0.2$ & $0.55$ & $0.68$ & $4.04$ & $1.09$ & $4.87$ & $0.61$ & $1.94$\\
		$600$ & $100$ & $t_1$ & $\mathcal{G}$ & $0.23$ & $1.72$ & $0.01$ & $0.08$ & $3.95$ & $9.52$ & $0.93$ & $4.92$ & $0$ & $0$\\
		$600$ & $100$ & $t_1$ & $t_1$ & $0.13$ & $1.49$ & $0$ & $0.01$ & $1.37$ & $5.76$ & $1.04$ & $4.15$ & $0.25$ & $0.42$\\
		$600$ & $100$ & $t_1$ & $t_2$ & $0.1$ & $1.54$ & $0$ & $0.02$ & $3.33$ & $9.25$ & $1.05$ & $5.05$ & $0.01$ & $0.01$\\
		$600$ & $200$ & $\mathcal{G}$ & $\mathcal{G}$ & $0$ & $0.12$ & $0$ & $0$ & $1.04$ & $4.94$ & $1.02$ & $4.94$ & $0$ & $0.04$\\
		$600$ & $200$ & $\mathcal{G}$ & $t_1$ & $0.46$ & $1.02$ & $0.26$ & $0.51$ & $0.89$ & $4.77$ & $1.18$ & $3.41$ & $1.5$ & $2.37$\\
		$600$ & $200$ & $\mathcal{G}$ & $t_2$ & $0.12$ & $0.5$ & $0.04$ & $0.1$ & $0.68$ & $4.08$ & $1.2$ & $4.82$ & $0.49$ & $1.94$\\
		$600$ & $200$ & $t_1$ & $\mathcal{G}$ & $0$ & $0.12$ & $0$ & $0$ & $3.45$ & $9.07$ & $0.98$ & $5.11$ & $0$ & $0$\\
		$600$ & $200$ & $t_1$ & $t_1$ & $0.01$ & $0.26$ & $0$ & $0$ & $1.25$ & $5.61$ & $1.13$ & $4.12$ & $0.27$ & $0.49$\\
		$600$ & $200$ & $t_1$ & $t_2$ & $0$ & $0.1$ & $0$ & $0$ & $2.71$ & $8.74$ & $1.01$ & $4.75$ & $0$ & $0.01$\\
		\hline\hline
	\end{tabular}
	\caption{\label{Tab:SizeControl}Percentage of rejections of various tests under the null, estimated over 100000 Monte Carlo repetitions, for various noise distributions at nominal levels of $\alpha=1\%$ and {\revone $\alpha=5\%$}. Data are generated from the model in~\eqref{Eq:Model} and~\eqref{Eq:Model2} with $b=0$. $\bX$, $\bv$ and $\be$ are generated according to the various distribution types prescribed in the table. Here ``$\mathcal{G}$'' stands for standard normal distribution. Percentage signs are omitted.}
\end{table}
We start by analysing the validity of various tests under the null described in Section~\ref{Sec:SimSetup}. {\rev We estimated the size of RPT, $\textrm{RPT}_{\textrm{EM}}$, DR, FL, CRT, RB, RR and HDI at nominal levels of $1\%$ and {\revone $5\%$} for $(n, p) \in\{(300, 100), (600, 100), (600, 200)\}$ (see Table~\ref{Tab:Size5pc} in the Supplementary Material for estimated size at {\revone 0.5\%} nominal level). RB, RR and HDI displayed more serious violation of the empirical sizes in these simulation settings (see Table~\ref{Tab:AdditionalSizeSimulation} in the Supplementary Material).} 
The results for the remaining procedures are summarised in Table~\ref{Tab:SizeControl}. Notice that since the p-values of both ANOVA and the naive RPT are invariant with respect to the choices of $\beta, \beta^Z$ and $\Sigma$, the results in Table~\ref{Tab:ANOVA} are directly comparable to the ones in Table~\ref{Tab:SizeControl}. Therefore, we do not repeat the simulations of the two tests here.

{\rev From Table~\ref{Tab:SizeControl}, we see that DR has good size control when the design matrix $\boldsymbol{X}$ has Gaussian components and exceeds the nominal size levels when $\boldsymbol{X}$ is generated with $t_1$ components. FL performed the best when $n/p$ is relatively large, consistent with the asymptotic size validity of the test established in \citet{freedman1983nonstochastic}, though with low $n/p$ ratios and heavy-tailed noise, the empirical sizes can exceed the nominal level. CRT is conservative when components of $\boldsymbol{X}$ and the noise have the same distribution, but can violate the size control when the noise distributions have much heavier tails than that of components of $\boldsymbol{X}$.}
On the other hand, RPT exhibits valid, {\revone though sometimes conservative}, size controls in all settings, which is consistent with our theoretical findings. More interestingly, the size of $\textrm{RPT}_{\textrm{EM}}$ is also valid across all the simulation settings, even with heavy-tailed noises and heavy-tailed design. 
In Section~\ref{sec:simpower}, we further study the empirical power of RPT and $\textrm{RPT}_{\textrm{EM}}$.

\begin{figure}[t!]
	\centering
	\subfigure[Gaussian design, Gaussian noise]{\includegraphics[width=0.45\textwidth]{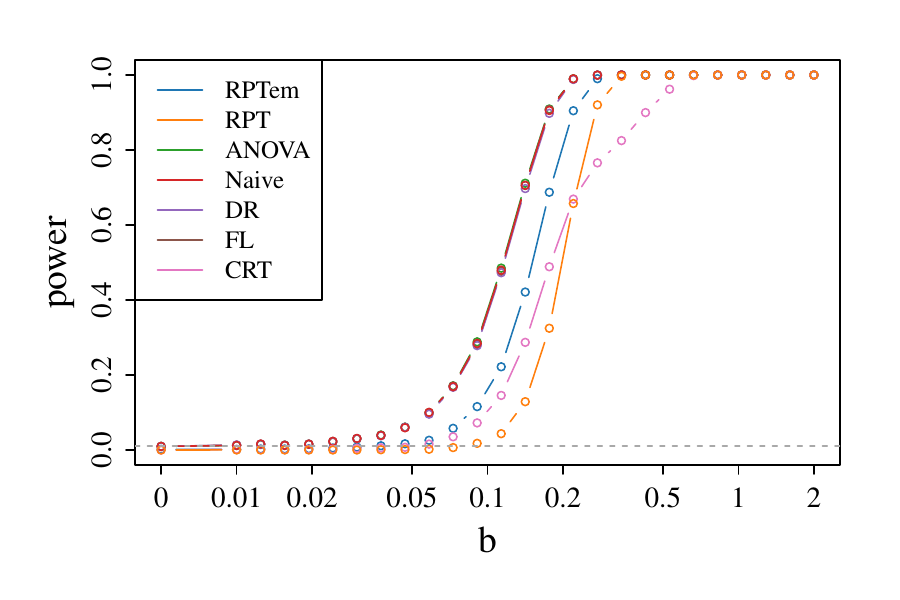}} 
	\subfigure[Gaussian design, $t_1$ noise]{\includegraphics[width=0.45\textwidth]{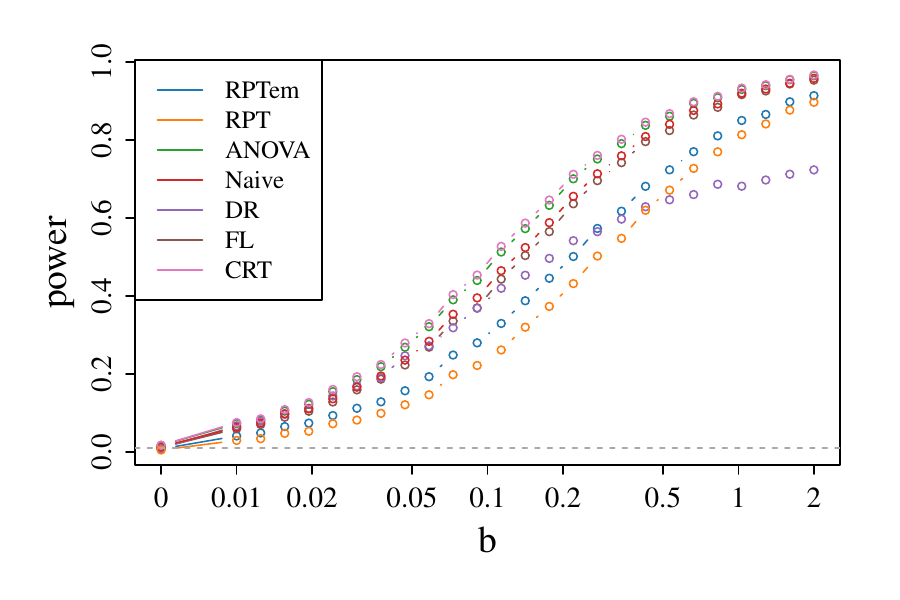}} 
	\subfigure[Gaussian design, $t_2$ noise]{\includegraphics[width=0.45\textwidth]{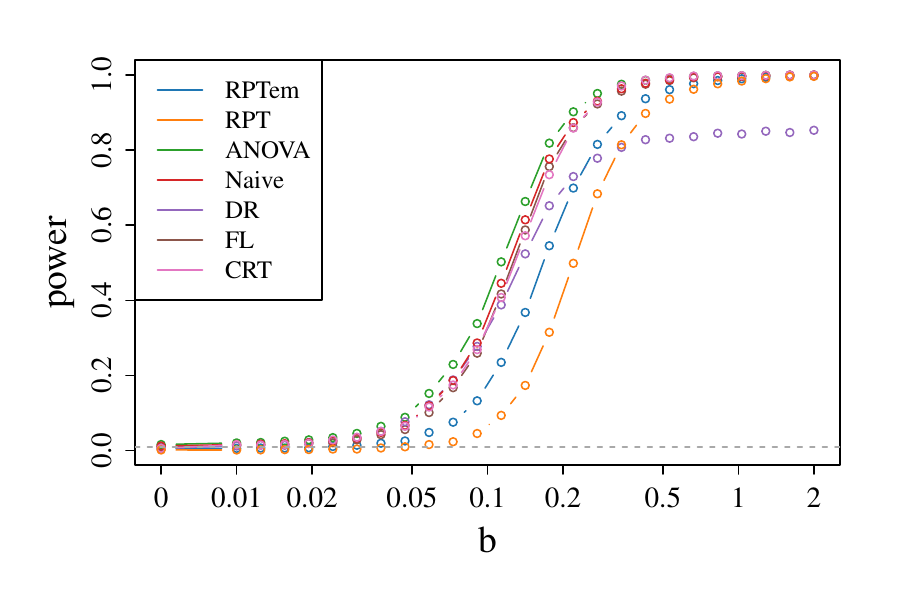}} 
	\subfigure[$t_1$ design, Gaussian noise]{\includegraphics[width=0.45\textwidth]{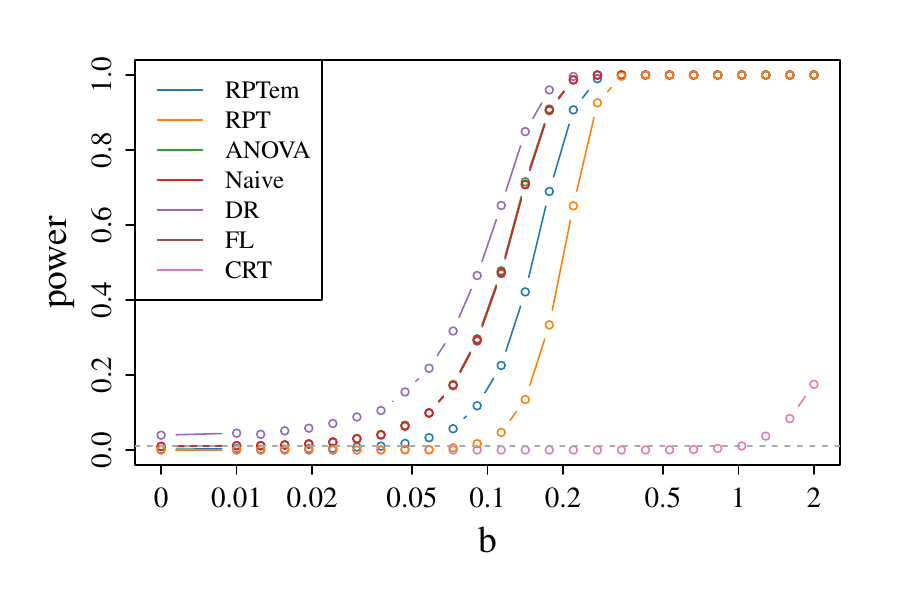}} 
	\subfigure[$t_1$ design, $t_1$ noise]{\includegraphics[width=0.45\textwidth]{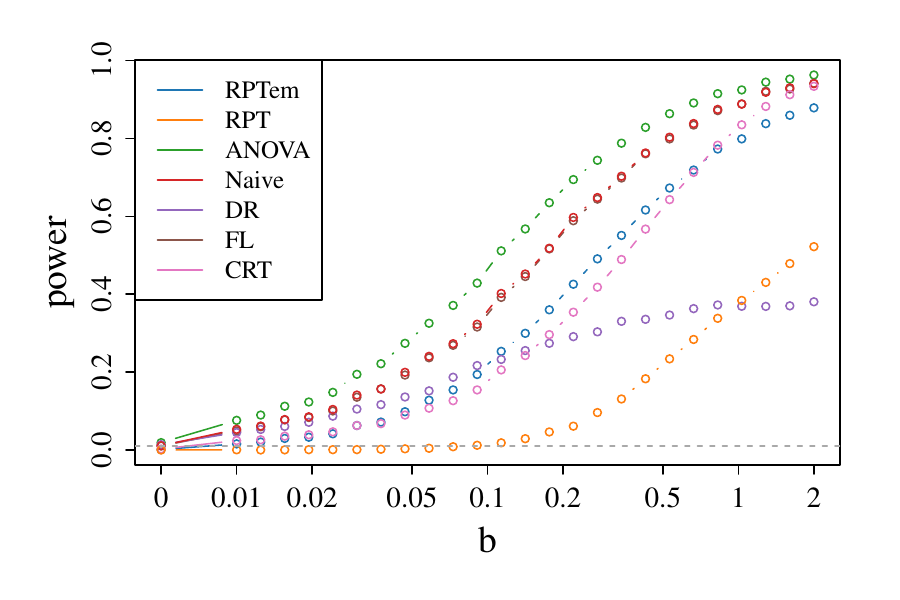}}
	\subfigure[$t_1$ design, $t_2$ noise]{\includegraphics[width=0.45\textwidth]{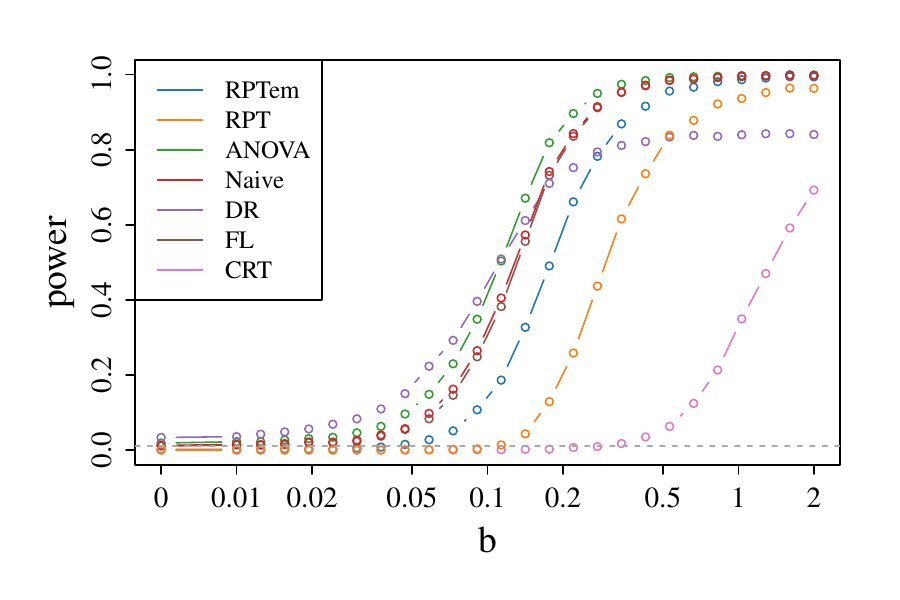}}
	\caption{Power (proportion of rejections) with nominal level $\alpha = 0.01$ (represented by the horizontal dashed line) over 10000 replicates for $b = 0$ or on a logarithmic grid between 0.01 and 2.  Here $\bX, \bv$ and $\be$ are generated according to various distribution types prescribed in the caption of each figure.} 
	\label{fig:power}
\end{figure}

\subsection{Numeric analysis of alternative power}\label{sec:simpower}

In Section~\ref{sec:alternative}, we established asymptotic power guarantees of RPT under fixed design and heavy tailed noises. 
In this section, we validate these theoretical insights via numerical analysis.
To benchmark the results, we investigate the power of all tests considered in Section~\ref{Sec:SimSetup}. We set $n=600$, $p=100$ and vary the $b$ in~\eqref{Eq:Model} for $b$ equals to $0$ or one of the 25 different values on an equally spaced logarithmic grid in the range of 0.01 to 2.
We analyze the power of all methods with design following Gaussian and $t_1$ distributions and noises following Gaussian, $t_1$, and $t_2$ distributions.
The estimated power curves for $\text{RPT}_{\text{EM}}$, RPT, ANOVA, naive RPT, DR, FL and CRT over 10000 repetitions are displayed in Figure~\ref{fig:power} (see also Figure~\ref{Fig:AdditionalPowerSimulation} in the Supplementary Material for power curves of RB, RR and HDI).


From Figures~\ref{fig:power}(a)-(c), (d) and (f), we can conclude that in most of the settings, the power of RPT is slightly worse than ANOVA, {\rev the naive RPT and FL}. The difference is more pronounced when both the design and the noise follow a heavy-tailed distribution (Figure~\ref{fig:power}(e)). However, bearing in mind the lack of valid size control of {\rev ANOVA, naive RPT, DR, FL and CRT, especially when design and noise are heavy-tailed,} we would argue that the gap in power between RPT and these competitors is the price to pay for finite-sample size validity with only exchangeable noise in {\revone the proportional regime}. Moreover, RPT is nevertheless still guaranteed to reject the alternative with high probability given {\rev a signal size $b$ not too much larger than the competitors. In addition, we observe that DR does not seem to have power converging to 1 as $b$ increases for heavy-tailed noise, while the power of CRT is substantially reduced for heavy-tailed design distributions.} 

Another interesting phenomenon is that the power of $\textrm{RPT}_{\textrm{EM}}$ is generally stronger than RPT, especially in the setting displayed in Figure~\ref{fig:power}(e), where both design and noise follow $t_1$ distribution. This, together with the size validity display in Section~\ref{sec:simnull}, suggests $\textrm{RPT}_{\textrm{EM}}$, although being lack of theoretical support, can serve as a viable alternative of RPT in empirical analysis. We leave the theoretical investigations of $\textrm{RPT}_{\textrm{EM}}$ as future work.

Finally, we compare RPT with the cyclic permutation test proposed in \citet{LB21}. As the cyclic permutation test is not well-defined for $n / p < 1/\alpha + 1$, we consider a relatively low dimensional setting where $n=1000$, $p=40$ and $\alpha=0.05$. {\revone We consider the test with both the default variable ordering (CPT) and a pre-ordering computed using a genetic algorithm (CPT-GA). Due to computational limitations, the genetic algorithm is computed with only $1000$ random initializations.} The data generation mechanism is the same as that described in Section~\ref{Sec:SimSetup}{\revone, except that to adapt the high computational cost of CPT-GA, for each specification we first generate $10$ repetitions of ($\bX, \bZ$), then for each ($\bX, \bZ$), we generate $1000$ repetitions of $\bv$, summing up to $10000$ repetitions}. Figure~\ref{Fig:PowerLB} shows the power curves of $\textrm{RPT}_{\textrm{EM}}$, RPT, {\revone CPT-GA and CPT} under various design matrix and noise distributions. {\revone We see that all four methods are well-calibrated at 5\% level when $b=0$, with RPT slightly more conservative than the other three approaches. For all the settings, the power of RPT and $\textrm{RPT}_{\textrm{EM}}$ converges to $1$ faster than CPT, though CPT has higher rejection rate than RPT as $b$ begins to diverge from zero. For CPT-GA, the performance of RPT and CPT-GA are comparable; and CPT-GA can significantly outperform RPT when both the design and noise are heavy-tailed. Moreover, we have found that genetic algorithm can significantly increase the power of CPT, which is consistent with the observation from~\citet{LB21}. From \citet{LB21}, it is expected that once we increase the number of random initializations for genetic algorithm from $1000$ to the recommended setting of $10000$, CPT-GA can become more powerful.}

\begin{figure}[t!]
	\centering
	\subfigure[Gaussian design, Gaussian noise]{\includegraphics[width=0.32\textwidth]{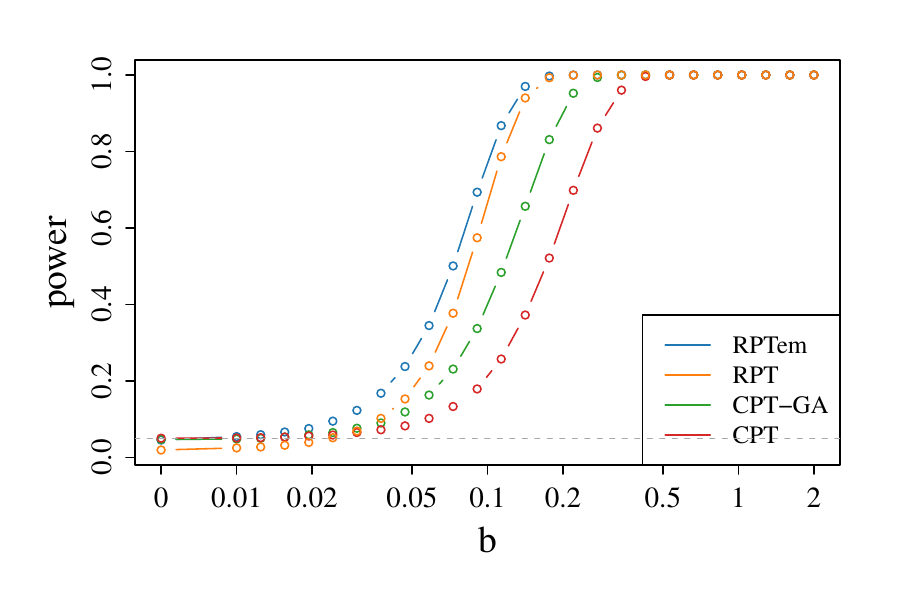}} 
	\subfigure[Gaussian design, $t_1$ noise]{\includegraphics[width=0.32\textwidth]{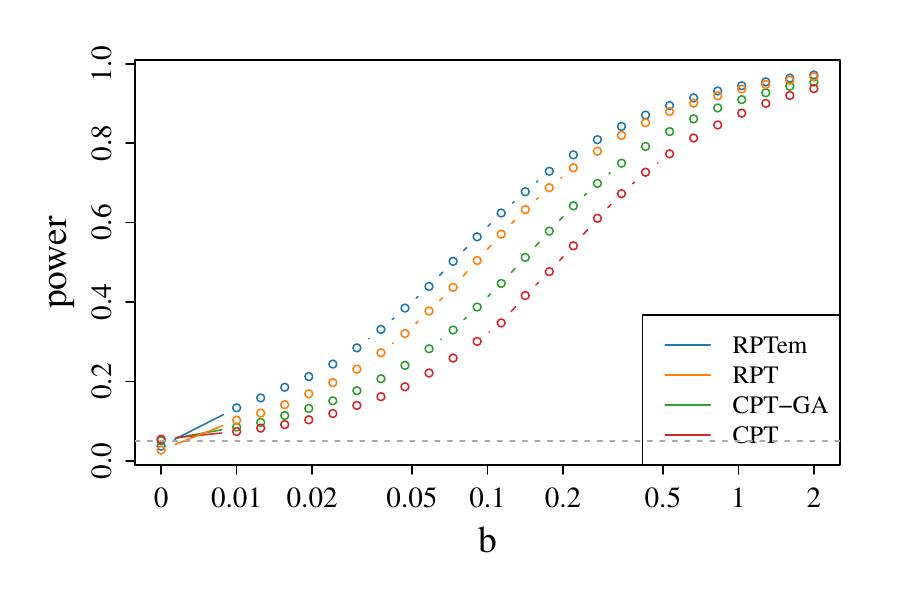}}
	\subfigure[Gaussian design, $t_2$ noise]{\includegraphics[width=0.32\textwidth]{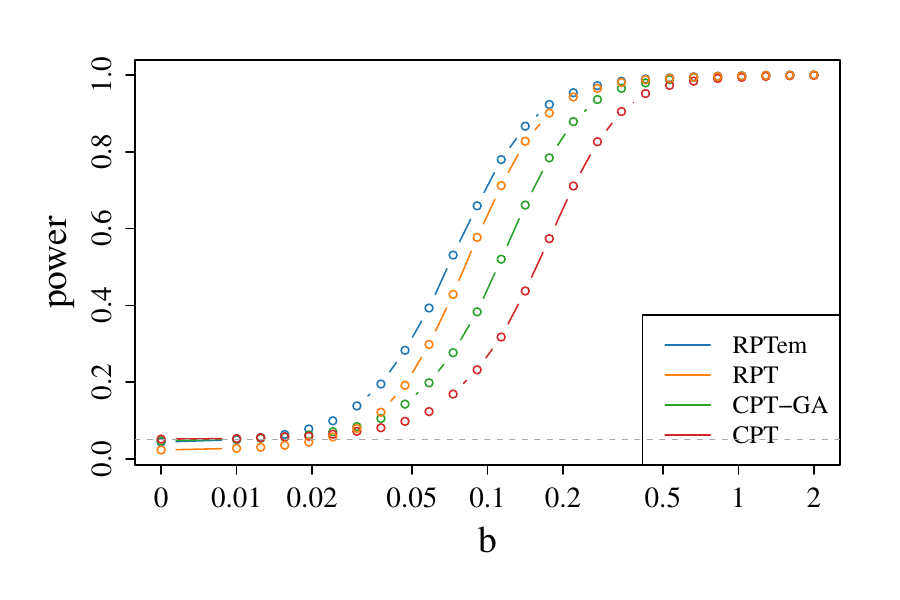}} \\
	\subfigure[$t_1$ design, Gaussian noise]{\includegraphics[width=0.32\textwidth]{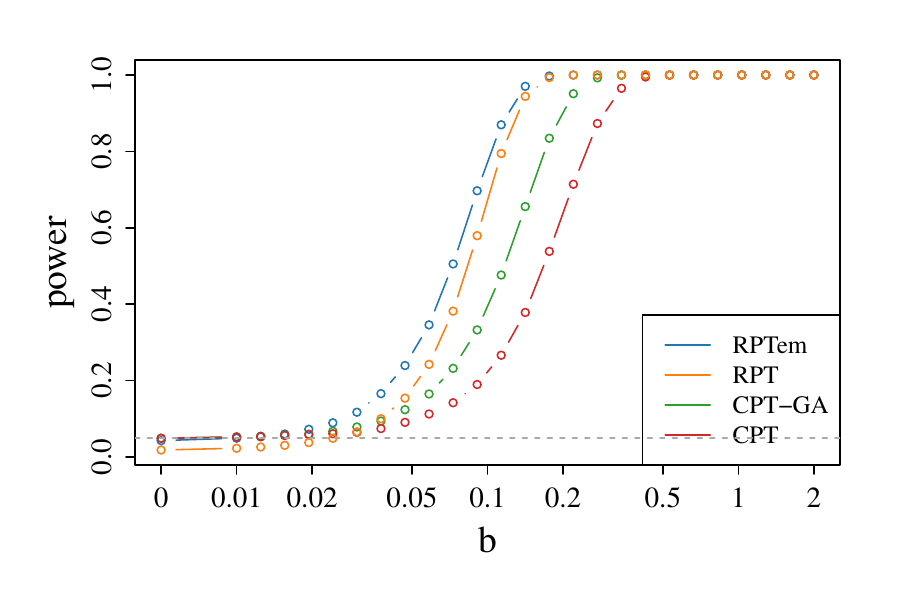}} 
	\subfigure[$t_1$ design, $t_1$ noise]{\includegraphics[width=0.32\textwidth]{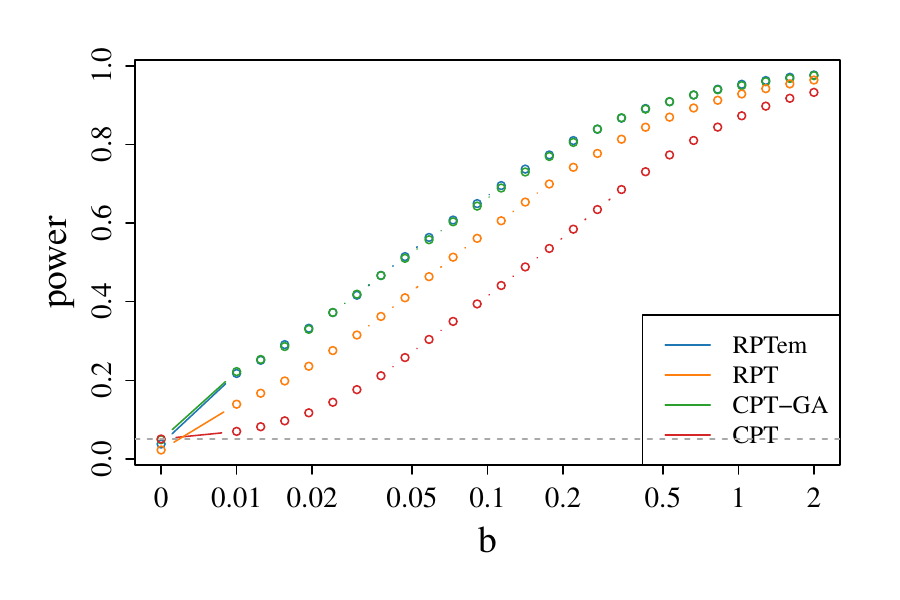}}
	\subfigure[$t_1$ design, $t_2$ noise]{\includegraphics[width=0.32\textwidth]{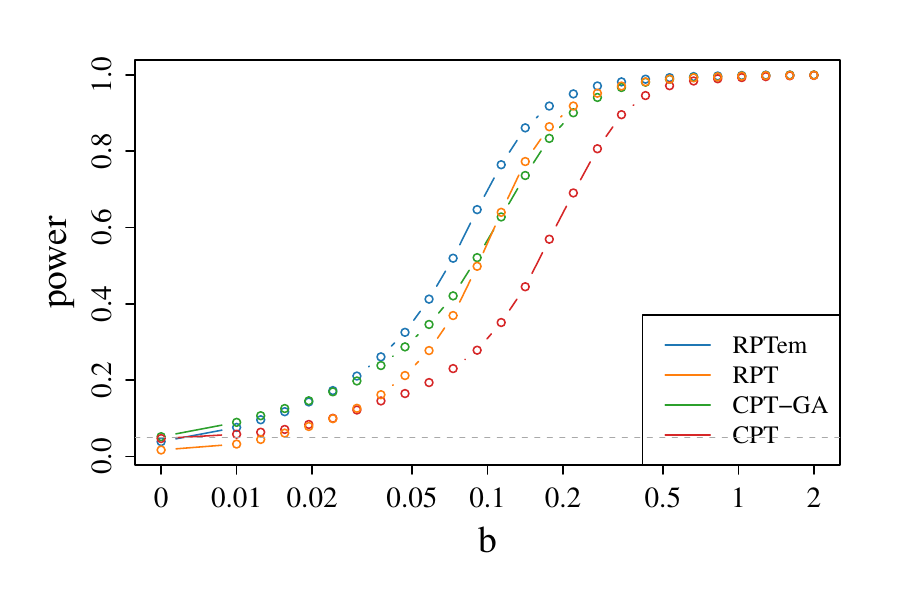}}
	\caption{\rev Power (proportion of rejections) with nominal level $\alpha = 0.05$ (represented by the horizontal dashed line) over 10000 replicates for $b = 0$ or on a logarithmic grid between 0.01 and 2. Here $\bX, \bv$ and $\be$ are generated according to various distribution types prescribed in the caption of each figure.}
	\label{Fig:PowerLB}
\end{figure}



\section{Discussion}\label{sec:discussion}


In this paper, we propose a new method for fixed design regression coefficient test {\revone when the number of covariates $p$ can be as large as a fraction of the sample size $n$}. RPT is a permutation-based approach that exploits the exchangeability of the noise terms to achieve finite-sample size control. Our approach uses the fact that the empirical residuals of the classical OLS fit is equivalent to the projection of the noise vector onto an subspace orthogonal to the design to construct a test with valid size for $p < n / 2$ based on multiple subspace projection. At the same time, we provide power analysis of RPT, and derived the signal detection rate of the coefficient $b$ in the presence of heavy-tailed noise vector $\bs{\varepsilon}$. As a by product, we propose $\textrm{RPT}_{\textrm{EM}}$ and demonstrate its size validity and power via numerical experiments. It would be of interest to understand the theoretical properties of $\textrm{RPT}_{\textrm{EM}}$ in future study.

In the regime where $n / 2 \le p < n$, we propose the naive RPT, and prove its finite-sample size validity under spherically invariant distributions, and compare it with ANOVA as well as other competing approaches via numerical experiments. In the meanwhile, we provide a more profound analysis of ANOVA test, which is of independent interest for practitioners interested in ANOVA.

In this paper, permutation test facilitates an important basis for construction of our test. This sheds light on extending permutation tests to solve other 
problems in modern statistics, which we leave as future work. In addition, permutation tests and its related the rank based tests have also been applied in model-free uncertainty quantification of machine learning predictions~\citep{LRW13,BHV14,RPC19}. It would be of interest if the power analysis techniques invented in this paper could be used to understand the efficiency of these approaches in modern machine learning applications.

%
%

\begin{acks}[Acknowledgments]
 The authors would like to thank the Editor, the Associate Editor and two anonymous referees for helpful comments that improved the paper. YW is also affiliated with Shanghai Artificial Intelligence Laboratory and Shanghai Qi Zhi Institute.
\end{acks}
\begin{funding}
 The research of TW was supported by EPSRC grant EP/T02772X/2. The research of YW was supported by National Key R \& D Program (2022YFA1008100), and the grant of National Natural Science Foundation of China (NSFC) 12201341.
\end{funding}

\begin{supplement}
	\stitle{Supplementary Material to ``Residual permutation test for regression coefficient testing''}
	\sdescription{The supplementary material contains {\rev additional power analysis of RPT when $K$ diverges with $n$,} the proofs of all the theoretical statements appeared in the paper, {\rev and additional simualtion studies}.}
\end{supplement}


\bibliographystyle{imsart-nameyear} 
\bibliography{reference}       


\newpage

\renewcommand {\theproposition} {A\arabic{proposition}}
\renewcommand {\thefigure} {A\arabic{figure}}
\renewcommand {\thetable} {A\arabic{table}}
\renewcommand {\theequation} {A\arabic{section}.\arabic{equation}}
\renewcommand {\thelemma} {A\arabic{lemma}}
\renewcommand {\thesection} {A\arabic{section}}
\renewcommand {\thetheorem} {A\arabic{theorem}}
\renewcommand {\thecorollary} {A\arabic{corollary}}
\renewcommand {\thecondition} {A\arabic{condition}}

\renewcommand {\thepage} {A\arabic{page}}
\setcounter{page}{1}

\setcounter{equation}{0}
\setcounter{section}{0}
\setcounter{figure}{0}
\setcounter{proposition}{0}
\setcounter{theorem}{0}
\setcounter{lemma}{0}
\setcounter{table}{0}
\setcounter{condition}{0}

\begin{center}
	\bf \large
	\uppercase{Supplement to ``Residual permutation test for regression coefficient testing''}
\end{center}

Appendix~\ref{sec:divergingK} provides additional power analysis of RPT when $K$ diverges with $n$. 

Appendix~\ref{sec:pfvalidity} provides proof of all size validity statements of the ANOVA test, naive RPT and RPT. It includes the proof of all the theoretical statements in Sections~\ref{sec:size} and \ref{sec:null} and also the discussion of the equivalence between~\eqref{eq:pvalestori} and~\eqref{eq:pvalest}.

Appendix~\ref{sec:powerpre} provides a preliminary lemma for RPT's power analysis.

Appendix~\ref{sec:pfalternative} provides proof of the statistical power results of RPT when $K$ is fixed and noises are i.i.d. It includes the proof of the theoretical statements in Section~\ref{sec:alternative}.

Appendix~\ref{sec:pfaddpower} provides proof of the rest of the power results of RPT. In includes proof of the theoretical statements in Section~\ref{sec:alternativehete} and Appendix~\ref{sec:divergingK}.

Appendix~\ref{sec:pfopt} studies the minimax rate optimality of coefficient test with heavy-tailed noises. It includes proof of the theoretical statements in Section~\ref{sec:minimax}.

Appendix~\ref{sec:addnum} provides additional numerical analysis.

\noindent{\bf Notations} we define $\|\|_\op$ as operator norm, $\|\|_2$ as $\ell_2$-norm, $\|\|_F$ as Frobenius norm. We define $a_1 / a_2 = \infty$ if $a_1 = a_2 = 0$ or $a_2 = 0$. 
Without loss of generality, we assume $b > 0$. Let ${\bf 1}$ be an $n$ dimensional vector with all entries equal to $1$. 

{\rev 
	\section{Statistical power of RPT with a diverging $K$}\label{sec:divergingK}
	
	
	In this section, we discuss the power of RPT when we allow $K$ to diverge with $n$. We have the following theorem:
	
	\begin{theorem}\label{thm:kalternative1}
		Suppose that $(\bs{X},\bs{Z},\bs{Y})$ is generated under model~\eqref{Eq:Model} where $\bs{\varepsilon}$ and $\bZ$ satisfy Assumption~\ref{as:iid} and
		\[
		0 < \E[|e_1|^{2 + \kappa}] < \infty \quad \text{and} \quad 0 < \E[|\varepsilon_1|^{1 + t}] < \infty
		\]
		for some constants $t \in [0, 1]$ and $\kappa \ge 0$. Assume also that $\mathcal{P}_K$ satisfies Assumption~\ref{as:trace}. In the asymptotic regime where $K$ varies with $n$ with rate 
		\begin{equation*}
			K = O(n^{\frac{2 \kappa}{2 + \kappa}}) \;\textrm{if}\;\; \kappa < 2 \quad\textrm{or}\quad K = o(n) \;\textrm{if}\;\; \kappa \ge 2,
		\end{equation*}
		and $p, b$ vary with $n$ and $K$ such that
		$n > (3+m)p + \min\{p, \sqrt{2 p} K\}$ for some constant $m > 0$ and
		\begin{align}\label{eq:bnk}
			|b| = \Omega (\sqrt{K} n^{-\frac{t}{1+t}}) \;\textrm{if}\;\; t < 1 \quad\textrm{or}\quad |b| = \omega(\sqrt{K} n^{-\frac{1}{2}}) \;\textrm{if}\;\; t = 1,
		\end{align}
		we have $\lim_{n \to \infty} \prob\left(\pval > \frac{1}{K + 1}\right) = 0$.
	\end{theorem}
	
	When $K = o(n / \sqrt{p})$, we require $n / p$ to be asymptotically larger than $3$ to get the desired power. As $K$ gets larger, the threshold becomes $4$ instead. When $K$ is a constant, then the power rate in~\eqref{eq:bnk} matches the main result~\eqref{eq:bn}. As the size of $K$ increases, we need more moments for $e_i$ to maintain the rate~\eqref{eq:bnk}. In particular, when the fourth order moment exists for $e_i$, $K$ can be as large as $o(n)$.
	
	In the rest of this section, we discuss the power of RPT when data generating mechanism satisfies Assumption~\ref{as:nonlinear} and $K$ is not a constant. 
	Extending the proofs of Theorems~\ref{thm:altnonlinear} and~\ref{thm:kalternative1}, we can conclude that under Assumption~\ref{as:nonlinear}, with additional constraints that
	\begin{itemize}
		\item uniformly for all $i$, $\E[|e_i|^{\kappa + 2}] \le C_e$ for some constant $C_e > 0$;
		\item for any fixed constant $B > 0$,
		\[
		\sum_{i = 1}^\infty \prob(|e_i^2 - \E[e_i^2]|^{1 + \kappa / 2} \ge B i) < \infty;
		\]
		\item $\sup_{i \ge 1} \E[|e_i^2 - \E[e_i^2]|^{1 + \kappa / 2} \one(|e_i^2 - \E[e_i^2]|^{1 + \kappa / 2} \ge a)] \to 0$ as $a \to \infty$,
	\end{itemize}
	and that $n > (3C_e / (c_e - r) +m)p + C_e / (c_e - r) \cdot \min\{p, \sqrt{2p} K\}$ for some fixed constant $m > 0$, RPT is still asymptotically powerful when $b$ and $K$ scales as in Theorem~\ref{thm:kalternative1}. In other words, even with heteroscedastic noise or nonlinear $\bZ$, RPT is still guaranteed to be asymptotically powerful with a diverging $K$.
}

\section{Proof of finite-sample size validity statements}\label{sec:pfvalidity}
\subsection{ANOVA size validity}

\begin{proof}[Proof of Lemma~\ref{lem:anova}]
	Recall that 
	\[
	\proj_{\bs{X}}:=\bs{X}(\bs{X}^\top \bs{X})^{-1}\bs{X}^\top \quad\textrm{and}\quad \proj_{\bs{X},\bs{Z}}:=(\bs{X}, \bs{Z})\bigl\{(\bs{X}, \bs{Z})^\top (\bs{X}, \bs{Z})\bigr\}^{-1}(\bs{X}, \bs{Z})^\top.
	\]
	First assume that $\bv$ is spherically symmetric. Since $\bs{\varepsilon}$ has a spherically symmetric distribution, we can write $\bs\varepsilon = \rho\bs\xi$, such that $\bs\xi \sim \mathrm{Unif}(\mathcal{S}^{n})$, i.e., a random vector that is sampled uniformly from the unit sphere with respect to the Haar measure; and that $\rho$ is some random variable taking value in $[0, \infty)$ and is independent from $\bs{\xi}$.  Then, we have almost surely,
	\begin{equation}\label{eq:anova}
		\pval_{\anova} = \frac{\|(\bs{I} - \mathrm{Proj}_{\bs{X}})(\bs{\varepsilon})\|_2^2 - \|(\bs{I} - \mathrm{Proj}_{\bs{X},\bs{Z}})(\bs{\varepsilon})\|_2^2}{\|(\bs{I} - \mathrm{Proj}_{\bs{X},\bs{Z}})(\bs{\varepsilon})\|_2^2/(n-p-1)} =\frac{\|(\mathrm{Proj}_{\bs{X},\bs{Z}}- \mathrm{Proj}_{\bs{X}})(\bs{\xi})\|_2^2}{\|(\bs{I} - \mathrm{Proj}_{\bs{X},\bs{Z}})(\bs{\xi})\|_2^2/(n-p-1)}.
	\end{equation}
	Hence, the distribution of $\pval_{\anova}$ does not depend on $\rho$. 
	
	By Cochran's theorem, we know that $\pval_{\anova} \sim F_{1,n-p-1}$ when $\bv \sim \cN(0, \bs{I})$, i.e., a multivariate standard normal distribution. Moreover, when $\bv \sim \cN(0, \bs{I})$, we have $\bv$ satisfies the above decomposition $\bv = \rho \bs{\xi}$ for some random variable $\rho$. Now recall that $\pval_{\anova}$ does not depend on $\rho$ (as shown in~\eqref{eq:anova}), we must have $\pval_{\anova} \sim F_{1,n-p-1}$ for all spherically symmetric $\bs\varepsilon$ as desired. 
	
	If instead $(\bX, \bZ)$ is spherically symmetric, let $\bs{Q}$ be an independent random matrix that is sampled uniformly from $\mathbb{O}^{n\times n}$ with respect to the Haar measure, then
	\[
	\pval_{\anova} \stackrel{\mathrm{d}}= \frac{\|(\mathrm{Proj}_{\bs{QX},\bs{QZ}} - \mathrm{Proj}_{\bs{QX}})(\bs{\varepsilon})\|_2^2 }{\|(I_n - \mathrm{Proj}_{\bs{QX},\bs{QZ}})(\bs{\varepsilon})\|_2^2/(n-p-1)}=\frac{\|(\mathrm{Proj}_{\bs{X},\bs{Z}} - \mathrm{Proj}_{\bs{X}})(\bs{Q}^{-1}\bs{\varepsilon})\|_2^2 }{\|(I_n - \mathrm{Proj}_{\bs{X},\bs{Z}})(\bs{Q}^{-1}\bs{\varepsilon})\|_2^2/(n-p-1)}.
	\]
	Since $\bs{Q}^{-1}\bs{\varepsilon}$ has a spherically symmetric distribution, the desired conclusion follows from the first case.
\end{proof}

\subsection{Validity of naive residual permutation test}

\begin{proof}[Proof of Lemma~\ref{lem:naive}]
	Without loss of generality we just prove the lemma with Condition (a). We first consider the case where $\bv$ follows a spherically symmetric distribution. Then using an analogous analysis as in Lemma~\ref{lem:anova}, we have
	\begin{align*}
		\pval_{\naive} & = \frac{1}{K + 1} \left(1 + \sum_{k = 1}^{K} \one(|\hat{\be}^\top \hat{\bv}| \le |\hat{\be}^\top \bP_k \hat{\bv}|)\right) \\
		& = \frac{1}{K + 1} \left(1 + \sum_{k = 1}^{K} \one(|\bZ^\top \bV_0 \bV_0^\top \bs{\xi}| \le |\bZ^\top \bV_0 \bP_k \bV_0^\top \bs{\xi}|)\right).
	\end{align*}
	This means that just like $\pval_{\anova}$, the distribution of $\pval_{\naive}$ does not depend on $\rho$. Moreover, when $\bv$ follows a multivariate standard normal distribution, $\bV_0 \bv$ is a $n - p$ dimensional multivariate standard normal random vector and thus $\pval_{\naive}$ is a valid p-value. Then using an analogous argument as in the proof of Lemma~\ref{lem:anova}, we have that $\pval_{\naive}$ is a valid p-value for all spherically symmetric noises.
	
	If instead $(\bX, \bZ)$ is spherically symmetric, again let $\bs{Q}$ be an independent matrix sampled uniformly from $\mathbb{O}^{n\times n}$, then
	\begin{align*}
		\pval_{\naive} & \stackrel{\mathrm{d}}= \frac{1}{K + 1} \left(1 + \sum_{k = 1}^{K} \one(|(\bs{Q}\bZ)^\top \bs{Q} \bV_0 \bV_0^\top \bs{Q}^\top \bv| \le |(\bs{Q} \bZ)^\top \bs{Q} \bV_0 \bP_k \bV_0^\top \bs{Q}^\top \bv|)\right) \\
		& = \frac{1}{K + 1} \left(1 + \sum_{k = 1}^{K} \one(|\bZ^\top \bV_0 \bV_0^\top \bs{Q}^\top \bv| \le |\bZ \bV_0 \bP_k \bV_0^\top \bs{Q}^\top \bv|)\right).
	\end{align*}
	Then using an analogous argument, we prove the size validity of $\pval_{\naive}$.
\end{proof}

\subsection{Validity of residual permutation test}\label{sec:rptvalidity}

We first show that the two definitions of RPT defined in~\eqref{eq:pvalestori} and \eqref{eq:pvalest} are equivalent. Since by definition, $\hat{\bv} = \bV_0^\top \bY$, we easily have $\tilde{\bV}_k^\top \bV_0 \hat{\bv} = \tilde{\bV}_k^\top \bV_0 \bV_0^\top \bY = \tilde{\bV}_k^\top \bY$, where for the last equality we apply Lemma~\ref{lem:tildev}. Using an analogous argument, we can prove that $\tilde{\bV}_k^\top \bV_0 \hat{\be} = \tilde{\bV}_k^\top \bZ$. Now for $\tilde{\bV}_k^\top \bV_k \hat{\bv}$, we apply that 
\[
\tilde{\bV}_k^\top \bV_k \hat{\bv} = \tilde{\bV}_k^\top \bV_k \bV_0^\top \bY = \tilde{\bV}_k^\top \bV_k \bV_0^\top \bP_k^\top \bP_k\bY = \tilde{\bV}_k^\top \bV_k \bV_k^\top \bP_k\bY = \tilde{\bV}_k^\top \bP_k\bY,
\]
where for the last equality we apply again Lemma~\ref{lem:tildev}. Putting together, we see that the two definitions of $\pval$ in~\eqref{eq:pvalestori} and \eqref{eq:pvalest} are numerically equivalent.


In the rest of this section, our goal is to prove Theorem~\ref{thm:null}. We start with the following preliminary lemmas. Recall that for any matrix $\bU \in \R^{n \times q}$ with orthonormal columns and any vector $\bs{a} \in \R^n$, $\proj_{\bU} (\bs{a}) := \bU \bU^\top \bs{a}$. 

\begin{lemma}\label{lem:tildev}
	Let $\bU \in \R^{n \times p_1}$ and $\bV \in \R^{n \times p_2}$ be two matrices with orthonormal columns spanning subspaces of $\R^n$. Let $\bs{W} \in  \R^{n \times q}$ be a matrix with orthonormal columns spanning a subspace of $\matspan(\bU) \cap \matspan(\bV)$. Then for any vector $\bs{a} \in \R^n$, $\bs{W}^\top \bs{a} = \bs{W}^\top \proj_{\bU}(\bs{a})  = \bs{W}^\top \proj_{\bV}(\bs{a})$.
\end{lemma}

\begin{proof}
	This is straightforward using that
	\[
	\bs{W}^\top = \bs{W}^\top \bU \bU^\top = \bs{W}^\top \bV \bV^\top
	\]
	since $\bV$ spans a subspace of $\matspan(\bU)$ and $\matspan(\bV)$.
\end{proof}

\begin{lemma}\label{lem:hatep}
	Under $H_0$, $\bV_0 \hat{\bv} = \proj_{\bV_0}(\bv)$. Moreover, for any permutation matrix $\bP_k$, we have that $\bV_k \hat{\bv} = \proj_{\bV_k}(\bP_k\bv)$.
\end{lemma}

\begin{proof}
	Since we are under the $H_0$, we have that
	\[
	\hat{\bv} = \bV_0^\top \bY = \bV_0^\top \left(\bX \beta + \bv\right).
	\]
	Then as a direct consequence of that $\matspan(\bV_0)$ is orthogonal to $\matspan(\bX)$, we have that $\bV_0^\top \bX = 0$ and thus $\hat{\bv} = \bV_0^\top \bv$. From above, we have
	\[
	\bs{V}_0 \hat{\bs{\varepsilon}} = \bs{V}_0 \bV_0^\top \bv = \proj_{\bV_0}(\bv)
	\]
	and that
	\[
	\bs{V}_k \hat{\bs{\varepsilon}} = \bs{V}_k \bV_0^\top \bv = \bs{V}_k \bV_0^\top \bP_k^\top \bP_k \bv = \bs{V}_k \bV_k^\top \bP_k \bv = \proj_{\bV_k}(\bP_k \bv).
	\]
\end{proof}

\begin{proof}[Proof of Theorem~\ref{thm:null}]
	Throughout the proof we work on a fixed $(\bX, \bZ)$ and a fixed set of permutation matrices $\{\bP_0, \ldots, \bP_K\}$ satisfying Assumption~\ref{as:permutation}. 
	
	From Lemmas~\ref{lem:tildev} and~\ref{lem:hatep}, we have that for any $\alpha \in [0,1]$,
	\begin{align*}
		\mathrm{I}_\alpha & := \pr\left(\frac{1}{K+1} \left(1 + \sum_{k=1}^{K} \one\left\{ \min_{\tilde{\bs{V}} \in \{\tilde{\bs{V}}_1, \ldots, \tilde{\bs{V}}_K\}} T\left(\tilde{\bs{V}}^\top \bs{V}_0 \hat{\bs{e}}, \tilde{\bs{V}}^\top \bs{V}_0 \hat{\bs{\varepsilon}}\right) \leq T\left(\tilde{\bs{V}}_k^\top \bs{V}_0 \hat{\bs{e}}, \tilde{\bs{V}}_k^\top \bs{V}_k \hat{\bs{\varepsilon}}\right) \right\}\right) \leq \alpha\right) \\
		& = \pr\left(\frac{1}{K+1} \left(1 + \sum_{k=1}^{K} \one\left\{ \min_{\tilde{\bs{V}} \in \{\tilde{\bs{V}}_1, \ldots, \tilde{\bs{V}}_K\}} T\left(\tilde{\bs{V}}^\top \bs{V}_0 \hat{\bs{e}}, \tilde{\bs{V}}^\top \bs{\varepsilon}\right)\leq T\left(\tilde{\bs{V}}_k^\top \bs{V}_0 \hat{\bs{e}}, \tilde{\bs{V}}_k^\top \bP_k \bs{\varepsilon}\right) \right\}\right) \leq \alpha\right).
	\end{align*}
	Then using that for any $k \in \{1, \ldots, K\}$, 
	\begin{align*}
		& \one\left\{ \min_{\tilde{\bs{V}} \in \{\tilde{\bs{V}}_1, \ldots, \tilde{\bs{V}}_K\}} T\left(\tilde{\bs{V}}^\top \bs{V}_0 \hat{\bs{e}}, \tilde{\bs{V}}^\top \bs{\varepsilon}\right)\leq T\left(\tilde{\bs{V}}_k^\top \bs{V}_0 \hat{\bs{e}}, \tilde{\bs{V}}_k^\top \bP_k \bs{\varepsilon}\right) \right\} \\
		& \ge 
		\one\left\{ \min_{\tilde{\bs{V}} \in \{\tilde{\bs{V}}_1, \ldots, \tilde{\bs{V}}_K\}} T\left(\tilde{\bs{V}}^\top \bs{V}_0 \hat{\bs{e}}, \tilde{\bs{V}}^\top \bs{\varepsilon}\right)\leq \min_{\tilde{\bs{V}} \in \{\tilde{\bs{V}}_1, \ldots, \tilde{\bs{V}}_K\}} T\left(\tilde{\bs{V}}^\top \bs{V}_0 \hat{\bs{e}}, \tilde{\bs{V}}^\top \bP_k \bs{\varepsilon}\right) \right\},
	\end{align*}
	we have
	\begin{align*}
		\mathrm{I}_\alpha & \leq \pr\left(\frac{1}{K+1} \left(1 + \sum_{k=1}^{K} \one\left\{ \min_{\tilde{\bs{V}} \in \{\tilde{\bs{V}}_1, \ldots, \tilde{\bs{V}}_K\}}  T\left(\tilde{\bs{V}}^\top \bs{V}_0 \hat{\bs{e}}, \tilde{\bs{V}}^\top \bs{\varepsilon}\right) \right.\right.\right.\\
		& \hspace*{5cm} \left.\left.\left. \leq \min_{\tilde{\bs{V}} \in \{\tilde{\bs{V}}_1, \ldots, \tilde{\bs{V}}_K\}}  T\left(\tilde{\bs{V}}^\top \bs{V}_0 \hat{\bs{e}}, \tilde{\bs{V}}^\top \bP_k \bs{\varepsilon}\right) \right\}\right) \leq \alpha\right).
	\end{align*}
	By defining $g : \R^{n} \mapsto \R$ as a fixed projection depending only on $(\bs{X}, \bs{Z})$ and $\cP_K$ such that for any $\bs{a} \in \R^n$,
	\[
	g(\bs{a}) = \min_{\tilde{\bs{V}} \in \{\tilde{\bs{V}}_1, \ldots, \tilde{\bs{V}}_K\}} T\left(\tilde{\bs{V}}^\top \bV_0 \hat{\bs{e}}, \tilde{\bs{V}}^\top \bs{a} \right),
	\]
	we can further rewrite the above inequality as
	\[
	\mathrm{I}_\alpha \le \pr\left(\frac{1}{K+1} \left(1 + \sum_{k=1}^{K} \one\left\{g(\bs{\varepsilon}) \leq g(\bs{P}_k\bs{\varepsilon})\right\}\right) \leq \alpha\right).
	\]
	
	Using Lemma~\ref{lem:exchangeability}, we can finally have that
	\[
	\mathrm{I}_\alpha \leq 
	\pr\left(\frac{1}{K+1} \left(1 + \sum_{k=1}^{K} \one\left\{g(\bs{\varepsilon}) \leq g(\bs{P}_k\bs{\varepsilon})\right\}\right) \leq \alpha\right) \le \alpha,
	\]
	which proves the desired results.
\end{proof}

\subsubsection{Proof of Lemma~\ref{lem:exchangeability}}
\begin{proof}
	Let $\xi_0,\ldots,\xi_K \stackrel{\mathrm{i.i.d.}}{\sim} \cN(0,1)$ and independent of all other randomness in the problem. Let 
	\[
	R_k := \sum_{k'=0}^K  \mathbbm{1}\{g(\bs{P}_k\bs{\varepsilon}) \leq  g(\bs{P}_{k'}\bs{\varepsilon})\},
	\]
	and 
	\[
	\tilde R_k := \sum_{k'=0}^K\Bigl(\mathbbm{1}\{g(\bs{P}_k\bs{\varepsilon}) <  g(\bs{P}_{k'}\bs{\varepsilon})\ +  \mathbbm{1}\{g(\bs{P}_k\bs{\varepsilon}) = g(\bs{P}_{k'}\bs{\varepsilon}) \text{ and } \xi_k \leq \xi_{k'}\}\Bigr).
	\]
	In other words, $\tilde{R}_k$ is the rank of $g(\bs{P}_k\bs{\varepsilon})$ among $(g(\bs{P}_{k'}\bs{\varepsilon}):k'=0,\ldots,K)$ in a decreasing order, with random tie-breaking. Also, observe that $R_k\geq \tilde R_k$. By Assumptions~\ref{as:exchangeable} we have $
	\bs{\varepsilon} \stackrel{\mathrm{d}}{=} \bs{P}_k\bs{\varepsilon}$ for all $k$, hence 
	\begin{align*}
		R_0 &\stackrel{\mathrm{d}}{=} \sum_{k'=0}^K \mathbbm{1}\{g(\bs{P}_k\bs{\varepsilon}) \le g(\bs{P}_{k'}\bs{P}_{k}\bs{\varepsilon})\}=\sum_{k'=0}^K \mathbbm{1}\{g(\bs{P}_k\bs{\varepsilon}) \le g(\bs{P}_{k'}\bs{\varepsilon})\} = R_k,
	\end{align*}
	where we used Assumption~\ref{as:permutation} in the penultimate equality. Thus, for all $k\in\{0,\ldots,K\}$ and $x\in\{1,\ldots,K+1\}$, 
	\begin{equation}\label{eq:rkbnd}
		\mathbb{P}(R_k \leq x) = \frac{1}{K+1}\sum_{k'=0}^K \mathbb{P}(R_{k'} \leq x) \leq \frac{1}{K+1}\sum_{k'=0}^K \mathbb{P}(\tilde R_{k'} \leq x).
	\end{equation}
	
	On the other hand, almost surely $(\tilde R_0,\tilde R_1,\ldots,\tilde R_K)$ is a re-arrangement of $(1,\ldots,K+1)$. This means that for any fixed $j \in \{1, \ldots, K + 1\}$, almost surely there is a $k'$ such that $\tilde R_{k'} = j$. In other words, for $j \in \{1, \ldots, K + 1\}$,
	\[
	\sum_{k'=0}^K \mathbb{P}(\tilde R_{k'} = j) = 1.
	\]
	By taking this back to~\eqref{eq:rkbnd}, we may further bound~\eqref{eq:rkbnd} as
	\[
	\prob(R_k \leq x) \le \frac{x}{K + 1}.
	\]
	
	Then
	\[
	\mathbb{P}\biggl\{\frac{1}{K+1}\biggl(1+\sum_{k=1}^K \mathbbm{1}\{g(\bs{\varepsilon}) \leq g(\bs{P}_k\bs{\varepsilon})\}\biggr) \leq \alpha\biggr\} = \mathbb{P}\biggl(\frac{R_0}{K+1}\leq \alpha\biggr) \leq \frac{\lfloor \alpha (K+1)\rfloor}{K+1} \leq \alpha,
	\]
	as desired. 
\end{proof}

{\rev 
	\section{A preliminary lemma for power analysis}\label{sec:powerpre}
	
	In this section, our main goal is to prove Lemma~\ref{lem:tcomb}, which can be used to characterize the stochastic convergence of $\be \tilde{\bV}_k \tilde{\bV}_k^\top \bv$ or $\bs{h} \tilde{\bV}_k \tilde{\bV}_k^\top \bv$ in the proofs of Theorems~\ref{thm:alternative1},~\ref{thm:alternative2},~\ref{thm:altnonlinear} and~\ref{thm:kalternative1}.
	
	\begin{lemma}\label{lem:t1}
		Let $\bM_1, \ldots, \bM_K \in\R^{n\times n}$ be $K$ deterministic matrices that varies with $n$ and satisfies $\|\bM_k\|_\op \leq 1$ for all $k = 1, \ldots, K$; let $\bw := (w_1, \ldots, w_n)^\top$ be an $n$-dimensional deterministic vector that varies with $n$. Assume that $\bv$ satisfies Assumption~\ref{as:nonlinear} with $t = 1$. Then if $b = \omega \bigl(\sqrt{K} n^{-1/2}\bigr)$, we have that for any fixed $\delta, \gamma > 0$,
		\[
		\lim_{n \to \infty} \mathbb{P}\biggl(\exists 1 \le k \le K \;\st\; \frac{|\bw^\top \bM_k \bv|}{b\max\{\|\bw\|_2^2, \gamma n\}} > \delta\biggr) = 0.
		\]
	\end{lemma}
	
	\begin{proof}
		We have for any $1 \le k \le K$,
		\begin{align*}
			\E [|\bw^\top \bM_k\bv|^2] = \bE [\bw^\top \bM_k\bv\bv^\top \bM_k^\top \bw] \le C_\varepsilon \|\bw\|_2^2,
		\end{align*}
		and thus by Chebyshev's inequality and a union bound, for any $\delta > 0$,
		\begin{align*}
			& \mathbb{P}\biggl(\exists 1 \le k \le K \;\st\; \frac{|\bw^\top \bM_k \bv|}{b \max\{\|\bw\|_2^2, \gamma n\}} > \delta\biggr) \le \sum_{k = 1}^K\prob \left(\frac{|\bw^\top \bM_k\bv|}{b \max \{\|\bw\|_2^2, \gamma n\}} > \delta \right) \\
			& \le \sum_{k = 1}^K \frac{\E [|\bw^\top \bM_k \bv|^2]}{\delta^2 b^2 \max \{\|\bw\|_2^4, \gamma^2 n^2\}} \le \frac{C_\varepsilon K}{\delta^2 \gamma b^2 n}.
		\end{align*}
		From above, the desired result follows from $K / b^2 n= o(1)$.
	\end{proof}
	
	\begin{lemma}\label{lem:t10moment}
		Consider the $\bv$ in Assumption~\ref{as:nonlinear} with $t \in [0, 1)$; let $\bw \in \R^n$ be an $n$-dimensional vector that varies with $n$, we have that for any fixed $B > 0$, there exists a sequence of positive real values $c_n \to 0$ that does not depend on $\bw$ such that
		\[
		\sum_{i = 1}^n \E[w_i^2 \varepsilon_i^2 \one(|\varepsilon_i| \le B i^{\frac{1}{1 + t}})] \le c_n \cdot \|\bw\|_2^2 n^{\frac{1 - t}{1 + t}}.
		\]
		Moreover,
		\[
		\lim_{B \to \infty} \sup_{n \ge 1}\left\{\sum_{i=1}^n \left(\E[\varepsilon_i \one(|\varepsilon_i| \le B i^{\frac{1}{1 + t}})] \right)^2 / n^{\frac{1 - t}{1 + t}}\right\} = 0.
		\]
	\end{lemma}
	
	\begin{proof}
		Without loss of generality we assume throughout this proof that $C_\varepsilon = 1$. To control the first inequality, let $f_i := \varepsilon_i \one(|\varepsilon_i| \le B i^{\frac{1}{1 + t}})$, let $a_n$ be a sequence of integers such that as $n \to \infty$, $a_n \to \infty$ and $a_n / n \to 0$, then
		\begin{equation}\label{eq:fsqmean}
			\begin{aligned}
				\sum_{i = 1}^n w_i^2\bE [f_i^2]
				&=  \sum_{i = 1}^{a_n} w_i^2\bE [f_i^2] + \sum_{i = a_n + 1}^{n} w_i^2\bE [f_i^2] \\
				&\le  \sum_{i = 1}^{a_n} w_i^2\bE [\varepsilon_i^2 \one(|\varepsilon_i| \le Bi^{\frac{1}{1+t}})] + \sum_{i = a_n + 1}^{n} w_i^2\bE [\varepsilon_i^2 \one(|\varepsilon_i| \le B a_n^{\frac{1}{1+t}})] \\
				& \quad + \sum_{i = a_n + 1}^{n} w_i^2\bE [\varepsilon_i^2 \one( B a_n^{\frac{1}{1+t}} < |\varepsilon_i| \le Bi^{\frac{1}{1+t}})] \\
				&\le \sum_{i = 1}^{n} w_i^2\bE [\varepsilon_i^2 \one(|\varepsilon_i| \le B a_n^{\frac{1}{1+t}})] + \sum_{i = a_n + 1}^{n} w_i^2\bE [\varepsilon_i^2 \one( B a_n^{\frac{1}{1+t}} < |\varepsilon_i| \le Bi^{\frac{1}{1+t}})].
			\end{aligned}
		\end{equation}
		
		For the first term in the above inequality,
		\begin{align*}
			& \sum_{i = 1}^{n} w_i^2 \E [\varepsilon_i^2 \one(|\varepsilon_i| \le B a_n^{\frac{1}{1+t}})] = \sum_{i = 1}^{n} w_i^2 \E[|\varepsilon_i|^{1 + t} |\varepsilon_i|^{1 - t} \one(|\varepsilon_i| \le B a_n^{\frac{1}{1+t}})] \le \sum_{i = 1}^{n} w_i^2  \E[|\varepsilon_i|^{1 + t} B^{1 - t} a_n^{\frac{1 - t}{1+t}}] \\
			& \le B^{1 - t} \|\bw\|_2^2 \cdot a_n^{\frac{1 - t}{1+t}} = c_{1n}\|\bw\|_2^2 \cdot n^{\frac{1 - t}{1+t}},
		\end{align*}
		where $c_{1n} := B^{1 - t} \cdot a_n^{\frac{1 - t}{1+t}} / n^{\frac{1 - t}{1+t}} \to 0$ is a sequence of real values that does not depend on $\bw$.
		
		For the second term on the right hand side of~\eqref{eq:fsqmean}, we have
		\begin{align*}
			& \sum_{i = a_n + 1}^{n} w_i^2 \E [\varepsilon_i^2 \one( B a_n^{\frac{1}{1+t}} < |\varepsilon_i| \le Bi^{\frac{1}{1+t}})] = \sum_{i = a_n + 1}^{n} w_i^2 \E [|\varepsilon_i|^{1 + t} |\varepsilon_i|^{1 - t} \one( B a_n^{\frac{1}{1+t}} < |\varepsilon_i| \le Bi^{\frac{1}{1+t}})] \\
			& \le \sum_{i = a_n + 1}^{n} w_i^2 \E [|\varepsilon_i|^{1 + t} \one( B a_n^{\frac{1}{1+t}} < |\varepsilon_i| \le Bi^{\frac{1}{1+t}})] \cdot B^{1 - t} n^{\frac{1 - t}{1 + t}} \\
			& \le B^{1 - t} \|\bw\|_2^2 n^{\frac{1 - t}{1 + t}} \sup_{i \ge 1} \E[|\varepsilon_i|^{1 + t} \one(|\varepsilon_i| > Ba_n^{\frac{1}{1+t}})].
		\end{align*}
		Recall the restrictions we have about $\varepsilon_i$'s in Assumption~\ref{as:nonlinear}, we have as $n \to \infty$,
		\[
		c_{2n} := B^{1 - t} \cdot \sup_{i \ge 1} \E[|\varepsilon_i|^{1 + t} \one(|\varepsilon_i| > Ba_n^{\frac{1}{1+t}})] \to 0.
		\]
		Notice further that with the above definition, $c_{2n}$ also does not depend on $\bw$. In light of the above two results, we prove the first inequality where we select $c_n := c_{1n} + c_{2n}$.
		
		For the second inequality, we first prove it when $t \in (0, 1)$. Using that all $\varepsilon_i$'s are mean-zeroed, we have for any fixed $B > 0$,
		\begin{align}\label{eq:lemfmom}
			& \sum_{i = 1}^n (\E [f_i])^2 = \sum_{i = 1}^n (\E [\varepsilon_i \one(|\varepsilon_i| > Bi^{\frac{1}{1+t}})])^2 \overset{(i)}{\le} \sum_{i = 1}^n  (\E [|\varepsilon_i| ^{1+t}])^{\frac{2}{1+t}} \left(\E \left[\left\{\one(|\varepsilon_i| > Bi^{\frac{1}{1+t}})\right\}^{\frac{1 + t}{t}}\right]\right)^{\frac{2t}{1+t}}  \nonumber\\
			& = \sum_{i = 1}^n  (\E [|\varepsilon_i| ^{1+t}])^{\frac{2}{1+t}} (\E [\one(|\varepsilon_i| > Bi^{\frac{1}{1+t}})])^{\frac{2t}{1+t}} \le \sum_{i = 1}^n \prob(|\varepsilon_i|^{1+t} > B^{1+t}i)^{\frac{2t}{1+t}} \nonumber \\
			& \overset{(ii)}{\le} n \left(\frac{\sum_{i = 1}^n \prob(|\varepsilon_i|^{1+t} > B^{1+t}i)}{n}\right)^{\frac{2t}{1+t}} \le n^{\frac{1 - t}{1 + t}} \left(\sum_{i = 1}^\infty \prob(|\varepsilon_i|^{1+t} > B^{1+t}i)\right)^{\frac{2t}{1+t}}. 
		\end{align}
		where $(i)$ uses H\"{o}lder's inequality; $(ii)$ uses Jensen's inequality. Now given a fixed $\eta > 0$, from the conditions of $\varepsilon_i$'s we must have that there exists a $N_\eta$ such that $\sum_{i = N_\eta}^\infty \prob(|\varepsilon_i|^{1+t} > i) \le \eta / 2$.
		Moreover, form Markov's inequality we further have that there exists a $B_\eta$ such that for any $B \ge B_\eta$, $\sum_{i = 1}^{N_\eta} \prob(|\varepsilon_i|^{1+t} > B i) \le \eta / 2.$
		Putting together, we have that for any $B \ge \max\{B_\eta, 1\}$,
		\[
		\sum_{i = 1}^\infty \prob(|\varepsilon_i|^{1+t} > B^{1 + t} i) \le \eta.
		\]
		Since the above result holds for arbitrary $\eta > 0$, we have as $B \to \infty$, 
		\[\sum_{i = 1}^\infty \prob(|\varepsilon_i|^{1+t} > B^{1 + t} i) \to 0.
		\]
		In light of above and~\eqref{eq:lemfmom}, we prove that the second result in the lemma statement holds with $t \in (0 ,1)$. 
		
		We now turn to the case with $t = 0$. In this case, we have for any fixed $B > 0$,
		\[
		\sum_{i = 1}^n (\E [f_i])^2 \le n \cdot \left(\sup_{i \ge 1} \E[|\varepsilon_i| \one(|\varepsilon_i| > B i)]\right)^2 \le n \cdot \left(\sup_{i \ge 1} \E[|\varepsilon_i| \one(|\varepsilon_i| > B)]\right)^2.
		\]
		Using the requirement in Assumption~\ref{as:nonlinear} we have $\sup_{i \ge 1} \E[|\varepsilon_i| \one(|\varepsilon_i| > B)]$ converges to zero as $B \to \infty$, which proves the second result in the lemma statement in the $t = 0$ case.
		
		Taking together, we prove the desired result.
	\end{proof}
	
	\begin{lemma}\label{lem:tcomb}
		Consider the $\bM_1, \ldots, \bM_K$ and $\bw$ in Lemma \ref{lem:t1}; assume that $\bv$ satisfies Assumption~\ref{as:nonlinear} with $t \in [0, 1]$ and that $b$ satisfies~\eqref{eq:bnk}.
		Then for any fixed $\delta, \gamma > 0$,
		\[
		\lim_{n \to \infty} \mathbb{P}\biggl(\exists 1 \le k \le K \;\st\;  \frac{|\bw^\top \bM_k \bv|}{b\max\{\|\bw\|_2^2, \gamma n\}} > \delta\biggr) = 0.
		\]
	\end{lemma}
	
	\begin{proof}
		When $t = 1$ the result follows from Lemma~\ref{lem:t1}. In the rest of the proof we assume throughout $t \in [0, 1)$. From the scaling of $b$ we have that there exist $C_b, N_b > 0$ such that for all $n \ge N_b$, $b \ge C_b n^{-\frac{t}{1 + t}}$, which yields that for any $B > 0, n \ge N_b$,
		\[
		\frac{\|\bw\|_2 \sqrt{\sum_{i=1}^n \left(\E[\varepsilon_i \one(|\varepsilon_i| \le B i^{\frac{1}{1 + t}})] \right)^2}}{b\max\{\|\bw\|_2^2, \gamma n\}} \le \frac{\sqrt{\sum_{i=1}^n \left(\E[\varepsilon_i \one(|\varepsilon_i| \le B i^{\frac{1}{1 + t}})] \right)^2}}{\sqrt{\gamma} C_b n^{\frac{1 - t}{2 (1 + t)}}}.
		\]
		In light of the above and with Lemma~\ref{lem:t10moment}, we have that for any fixed $\delta, \gamma > 0$, there exists a constant $B_{\gamma, \delta} > 0$ depending on $(\gamma, \delta)$ such that for $n$ sufficiently large,
		\begin{equation}\label{eq:t10hete}
			\frac{\|\bw\|_2 \sqrt{\sum_{i=1}^n \left(\E[\varepsilon_i \one(|\varepsilon_i| \le B_{\gamma, \delta} i^{\frac{1}{1 + t}})] \right)^2}}{b\max\{\|\bw\|_2^2, \gamma n\}} \le \frac{\delta}{3}.
		\end{equation}
		
		Writing $f_i := \varepsilon_i \one(|\varepsilon_i| \le B_{\gamma, \delta} i^{\frac{1}{1 + t}})$ and $\bs{f} := (f_1, \ldots, f_n)^\top$, then we have that
		\begin{align*}
			|\bw^\top \bM_k \bv| & \le |\bw^\top \bM_k (\bs{f} - \E[\bs{f}])| + |\bw^\top \bM_k \E[\bs{f}]| + |\bw^\top \bM_k(\bv - \bs{f})| \\
			& \le |\bw^\top \bM_k (\bs{f} - \E[\bs{f}])| + \|\bw\|_2 \|\E[\bs{f}]\|_2 + \|\bw\|_2  \|\bv - \bs{f}\|_2.
		\end{align*}
		In light of this decomposition and also~\eqref{eq:t10hete}, we only need to prove that as $n \to \infty$,
		\begin{equation}\label{eq:t10hete1}
			\pr\left(\exists 1 \le k \le K \;\st\;  \frac{|\bw^\top \bM_k (\bs{f} - \E[\bs{f}])|}{b\max\{\|\bw\|_2^2, \gamma n\}} > \frac{\delta}{3}\right) \to 0
		\end{equation}
		and that
		\begin{equation}\label{eq:t10hete2}
			\pr\left(\frac{\|\bw\|_2 \|\bv - \bs{f}\|_2}{b\max\{\|\bw\|_2^2, \gamma n\}} > \frac{\delta}{3}\right) \le \pr\left(\frac{\|\bv - \bs{f}\|_2}{b\sqrt{\gamma n}} > \frac{\delta}{3}\right) \to 0.
		\end{equation}
		To prove~\eqref{eq:t10hete1}, applying Chebyshev's inequality and a union bound, we have
		\begin{align*}
			& \pr\left(\exists 1 \le k \le K \;\st\;  \frac{|\bw^\top \bM_k (\bs{f} - \E[\bs{f}])|}{b\max\{\|\bw\|_2^2, \gamma n\}} > \frac{\delta}{3}\right) \le \sum_{k = 1}^K \pr\left(\frac{|\bw^\top \bM_k (\bs{f} - \E[\bs{f}])|}{b\max\{\|\bw\|_2^2, \gamma n\}} > \frac{\delta}{3}\right) \\
			& \le \sum_{k = 1}^K \frac{9 \E[|\bw^\top \bM_k (\bs{f} - \E[\bs{f}])|^2]}{\delta^2 b^2\max\{\|\bw\|_2^4, \gamma^2 n^2\}}.
		\end{align*}
		Then by applying Lemma~\ref{lem:t10moment} where we select $\bw$ as $\bM_k^\top \bw$ and using that all the $\varepsilon_i$'s are independent and the basic inequality that $\E[(f_i - \E[f_i])^2] \le \E[f_i^2]$, we have
		\[
		\pr\left(\exists 1 \le k \le K \;\st\;  \frac{|\bw^\top \bM_k (\bs{f} - \E[\bs{f}])|}{b\max\{\|\bw\|_2^2, \gamma n\}} > \frac{\delta}{3}\right) = o\left(\frac{K \|\bw\|_2^2 n^{\frac{1 - t}{1 + t}}}{b^2 \max\{\|\bw\|_2^4, \gamma^2 n^2\}}\right) = o(1),
		\]
		which proves~\eqref{eq:t10hete1}.
		
		To prove~\eqref{eq:t10hete2}, apparently we already have
		\[
		\sum_{i = 1}^\infty \prob(f_i \neq \varepsilon_i) = \sum_{i = 1}^\infty \prob ( |\varepsilon_i| > B_{\gamma, \delta} i^{\frac{1}{1 + t}}) < \infty.
		\]
		Then given any constant $\eta > 0$, there exists a constant $N_\eta$ depending only on $\eta$ such that
		\[
		\prob(\exists i > N_\eta \;\st\; f_i \neq \varepsilon_i) \le \sum_{i = N_\eta + 1}^\infty \prob(f_i \neq \varepsilon_i) \leq \frac{\eta}{2}.
		\]
		Using that $N_\eta$ is finite, we further have there exists a constant $M_\eta$ such that
		\[
		\prob( \exists \ell \le N_{\eta} , \;\st\; |\varepsilon_\ell| > M_\eta) \le \frac{\eta}{2}.
		\]
		
		Writing the event $\mathcal{E}_\eta := \{ \forall i > N_{\eta}, f_i = \varepsilon_i \;\&\; \forall \ell \le N_{\eta} , |\varepsilon_i| \le M_{\eta}\}$ then we easily have that under this event, with $n$ sufficiently large, $\frac{\|\bv - \bs{f}\|_2}{b\sqrt{\gamma n}} \le \frac{\delta}{3}$. In other words, given any fixed $\eta > 0$, with $n$ sufficiently large,
		\[
		\pr\left(\frac{\|\bv - \bs{f}\|_2}{b\sqrt{\gamma n}} > \frac{\delta}{3}\right) \le \pr(\mathcal{E}_\eta^c) \le \eta,
		\]
		which proves~\eqref{eq:t10hete2}. In light of our control of~\eqref{eq:t10hete1} and~\eqref{eq:t10hete2}, we prove the desired result.
	\end{proof}
}

\section{Proof of basic power results}\label{sec:pfalternative}

In this section, we prove the power result of RPT in the basic model where $\bZ$ follows a linear function with respect to $\bX$ and all noises are i.i.d.

\subsection{Proof of Theorem~\ref{thm:alternative1}}
\label{sec:upp_asymp_high}

\begin{lemma}\label{lem:offdiagbnd}
	Let $\bM \in \R^{n \times n}$ be a matrix with all diagonal entries equal to zero. Then for any $\prob_e \in \cD_2$, we have for any fixed $\delta > 0$,
	\[
	\prob\left(\frac{\be^\top \bM \be}{n} > \delta\right) \leq \frac{4\|\bM\|_F^2}{n^2 \delta^2}.
	\]
\end{lemma}

\begin{proof}
	Observe that
	\begin{align*}
		\E \left[\left(\frac{\be^\top \bM \be}{n}\right)^2\right] &= \bE \left[\left(\sum_{i \neq j} \bM_{i,j}\frac{e_i e_j}{n}\right)^2 \right].
	\end{align*}
	Using that for any $i \neq j$, $e_i \independent e_j$, we have
	\[
	\E \left[\left(\frac{\be^\top \bM \be}{n}\right)^2\right] = \E \left[\sum_{i,j} \bM_{i,j}^2 \frac{e_i^2 e_j^2}{n^2}\right] \le \frac{4 \|\bM\|_F^2}{n^2}.
	\]
	Then by applying Chebyshev's inequality, we obtain the desired result.
\end{proof}

\begin{lemma}\label{lem:van06}
	For each $n$, let $V_{n,i} (1 \le i \le n)$ be independent random variables. Suppose that there exists a constant $C > 0$ such that for any $i, n$,  $\E[|V_{n,i}|] \le C$ and that 
	\[
	\lim_{a \to \infty} \sup_{n \ge 1} \frac{1}{n}\sum_{i=1}^n \E[ |V_{n,i}| \one(|V_{n, i}| > a)] = 0,
	\]
	then $\frac{1}{n} \sum_{i=1}^n (V_{n,i} - \E[V_{n,i}])$ converges in probability to zero.
\end{lemma}

\begin{proof}
	We just need to prove that for any fixed $\epsilon_1, \epsilon_2 > 0$, there exists a $N_{\epsilon_1, \epsilon_2}$ such that for any $n \ge N_{\epsilon_1, \epsilon_2}$,
	\begin{equation}\label{eq:van06}
		\pr\left(\left|\frac{1}{n} \sum_{i=1}^n (V_{n,i} - \E[V_{n,i}])\right| > \epsilon_1\right) \le \epsilon_2.
	\end{equation}
	Let $\epsilon = \epsilon_1 \epsilon_2 / 3$; then there exists a constant $a_\epsilon > 0$ such that for any $n \ge 1$, $ \frac{1}{n}\sum_{i=1}^n \E[ |V_{n,i}| \one(|V_{n, i}| > a_\epsilon)] \le \epsilon$.
	
	Write $\bar{V}_{n,i} =  V_{n,i} \one(|V_{n, i}| \le a_\epsilon)$ and $\tilde{V}_{n,i} = V_{n,i} \one(|V_{n, i}| > a_\epsilon)$. Then it holds that 
	\begin{align*}
		\E \left[ \left| \frac{1}{n} \sum_{i=1}^n (V_{n,i} - \E[V_{n,i}]) \right| \right] &= \E \left[ \left| \frac{1}{n} \sum_{i=1}^n (\bar{V}_{n,i} - \E[\bar{V}_{n,i}]) + \frac{1}{n} \sum_{i=1}^n (\tilde{V}_{n,i} - \E[\tilde{V}_{n,i}]) \right| \right]
		\\&\le \E \left[   \left| \frac{1}{n}  \sum_{i=1}^n (\bar{V}_{n,i} - \E[\bar{V}_{n,i}]) \right| \right] + 
		\E \left[  \frac{1}{n} \sum_{i=1}^n | \tilde{V}_{n,i} - \E[\tilde{V}_{n,i}]| \right] 
	\end{align*}
	For the second term on the right hand side of the above inequality, we have 
	\[
	\E \left[  \frac{1}{n} \sum_{i=1}^n | \tilde{V}_{n,i} - \E[\tilde{V}_{n,i}]| \right] \le  \frac{2}{n} \sum_{i=1}^n \E|\tilde{V}_{n,i}| \le 2 \epsilon. 
	\]
	For the first term, using the basic inequality that for any random variable $V$, $\E[|V|] \le \sqrt{\E[V^2]}$ and the independence of $\bar{V}_{n,i}$'s, we have
	\begin{align*}
		\E \left[  \left|  \frac{1}{n} \sum_{i=1}^n (\bar{V}_{n,i} - \E[\bar{V}_{n,i}]) \right| \right] &\le \left( \frac{1}{n^2} \E \left[ \left| \sum_{i=1}^n  (\bar{V}_{n,i} - \E[\bar{V}_{n,i}]) \right|^2 \right] \right)^{1/2} \\
		&\le \Big( \frac{1}{n^2}  \sum_{i=1}^n  \E \big| \bar{V}_{n,i} - \E[\bar{V}_{n,i}] \big|^2 \Big)^{1/2} \le \frac{a_\epsilon}{\sqrt{n}}.
	\end{align*}
	Hence for any $n \ge a_\epsilon^2 / \epsilon^2$,
	\[
	\E \left[ \left| \frac{1}{n} \sum_{i=1}^n (V_{n,i} - \E[V_{n,i}]) \right| \right] \le 3 \epsilon.
	\]
	In light of above, and by a Markov's inequality, we can have~\eqref{eq:van06} with $N_{\epsilon_1, \epsilon_2} := a_\epsilon^2 / \epsilon^2$, thereby proving the desired result.
\end{proof}

{\rev 
	\begin{lemma}\label{lem:suff}
		The assumption of $\bv$ in Theorem~\ref{thm:alternative1} is a sufficient condition of that of $\bv$ in Assumption~\ref{as:nonlinear}.
	\end{lemma}
	
	\begin{proof}
		We have
		\[
		\sum_{i = 1}^\infty \pr(|\varepsilon_i|^{1 + t} \ge B i) = \sum_{i = 1}^\infty \pr(|\varepsilon_1|^{1 + t} \ge B i) \le \int_0^\infty \pr\left(\frac{|\varepsilon_1|^{1 + t}}{B} \ge x\right) dx = \E\left[\frac{|\varepsilon_1|^{1 + t}}{B}\right] < \infty.
		\]
		Moreover, for any constant $a > 0$,
		\[
		\sup_{i \ge 1} \E[|\varepsilon_i|^{1 + t} \one(|\varepsilon_i|^{1 + t} > a)] = \E[|\varepsilon_1|^{1 + t} \one(|\varepsilon_1|^{1 + t} > a)].
		\]
		Since $\E[|\varepsilon_1|^{1 + t}]$ is bounded, by dominated convergence theorem, the above quantity converges to zero as $a \to \infty$.
	\end{proof}
}

\begin{proof}[Proof of Theorem~\ref{thm:alternative1}]
	Without loss of generality, we assume throughout this proof that $\prob_e \in \cD_2$ and $\prob_\varepsilon \in \cD_{1 + t}$. 
	
	To prove the desired result, it suffices to prove that for all $\delta > 0$,
	\begin{equation}\label{eq:upp_epsilon}
		\begin{aligned}
			\prob\left(\exists 1 \le k \le K \;\st\;  \frac{|\be^\top \tilde{\bs{V}}_k \tilde{\bs{V}}_k^\top \bv|}{b n} \ge \delta\right) & \to 0 ;\\
			\prob\left(\exists 1 \le k \le K \;\st\;  \frac{|\be^\top  \tilde {\bs{V}}_k \tilde {\bs{V}}_k^\top \bP_k \bv|}{b n} \ge \delta\right) & \to 0 ;
		\end{aligned}
	\end{equation}
	and that with probability converging to $1$, for all $1 \le j, k, \le K$,
	\begin{equation}\label{eq:upp_e}
		\begin{aligned}
			& \frac{\be^\top  \tilde {\bs{V}}_j \tilde {\bs{V}}_j^\top  \be - \be^\top  \tilde {\bs{V}}_k \tilde {\bs{V}}_k^\top  \bP_k \be}{n}  \ge \frac{m}{2 (4 + m)};\\
			& \frac{\be^\top  \tilde {\bs{V}}_j \tilde {\bs{V}}_j^\top  \be + \be^\top  \tilde {\bs{V}}_k \tilde {\bs{V}}_k^\top  \bP_k \be}{n}  \ge \frac{m}{2 (4 + m)}.
		\end{aligned}
	\end{equation}

	To prove the first claim of~\eqref{eq:upp_epsilon}, since $\prob_e \in \cD_2$, we have from the law of large number,
	\[
	\prob\left(\frac{1}{2} n\le \|\be\|_2^2 \le \frac{5}{2} n\right) \to 1.
	\]
	Let $\mathcal{E}$ denote the above event, applying basic inequalities of random events, we have
	\begin{align*}
		& \prob\left(\exists 1 \le k \le K \;\st\; \frac{|\be^\top \tilde{\bs{V}}_k \tilde{\bs{V}}_k^\top \bv|}{b n} \ge \delta\right) \le \prob\left(\exists 1 \le k \le K \;\st\; \frac{|\be^\top \tilde{\bs{V}}_k \tilde{\bs{V}}_k^\top \bv|}{b n} \ge \delta \mid \mathcal{E}\right) \prob(\mathcal{E}) + \prob(\mathcal{E}^c) \\
		& \overset{(i)}{=} \prob\left(\exists 1 \le k \le K \;\st\; \frac{|\be^\top \tilde{\bs{V}}_k \tilde{\bs{V}}_k^\top \bv|}{b \max\{\|\be\|_2^2, \frac{5}{2} n\}} \ge \frac{2}{5}\delta \mid \mathcal{E}\right) \prob(\mathcal{E}) + \prob(\mathcal{E}^c)
	\end{align*}
	where $(i)$ straightly follows from that we are under $\mathcal{E}$. Then as a direct consequence of Lemma~\ref{lem:tcomb} with $\be$ as $\bw$ and $\tilde{\bs{V}}_k \tilde{\bs{V}}_k^\top$ as $\bM_k$, and also Lemma~\ref{lem:suff} and that $K$ is a constant, we prove the first claim of~\eqref{eq:upp_epsilon}. For the second claim of~\eqref{eq:upp_epsilon}, by instead taking $\tilde{\bs{V}}_k \tilde{\bs{V}}_k^\top \bP_k$ as $\bM_k$, the result follows from the same argument as the first claim of~\eqref{eq:upp_epsilon}. 
	
	In the rest of the proof we focus on proving the first statement of~\eqref{eq:upp_e}, and the second statement can be proven via a similar argument. Since we assume $K$ as fixed, it boils down to proving that for any fixed $j, k$, with probability converging to $1$, the first statement of~\eqref{eq:upp_e} holds. To achieve this goal, we apply the decomposition
	\begin{align*}
		\frac{\be^\top  \tilde {\bs{V}}_j \tilde {\bs{V}}_j^\top  \be - \be^\top  \tilde {\bs{V}}_k \tilde {\bs{V}}_k^\top  \bP_k \be}{n} = & \frac{\be^\top  (\tilde {\bs{V}}_j \tilde {\bs{V}}_j^\top  - \mathrm{diag}(\tilde {\bs{V}}_j \tilde {\bs{V}}_j^\top ))\be - \be^\top  (\tilde {\bs{V}}_k \tilde {\bs{V}}_k^\top  \bP_k - \mathrm{diag}(\tilde {\bs{V}}_k \tilde {\bs{V}}_k^\top  \bP_k)) \be}{n}\\
		& + \frac{\be^\top  \mathrm{diag}(\tilde{\bs{V}}_j \tilde{\bs{V}}_j^\top)  \be - \be^\top  \mathrm{diag}(\tilde {\bs{V}}_k \tilde {\bs{V}}_k^\top  \bP_k) \be}{n} \\
		=: &\; \mathrm{I} + \mathrm{II},
	\end{align*}
	where for any matrix $\bs{A} \in \R^{n \times n}$, $\mathrm{diag}(\bs{A})$ corresponds to the diagonal matrix such that all the diagonal elements are equal to the diagonal elements of $\bs{A}$.
	
	For $\mathrm{I}$, observe that 
	\begin{align*}
		\|\tilde {\bs{V}}_j \tilde {\bs{V}}_j^\top  - \mathrm{diag}(\tilde {\bs{V}}_j \tilde {\bs{V}}_j^\top ) \|_F^2 & \le \|\tilde {\bs{V}}_j \tilde {\bs{V}}_j^\top  \|_F^2 = \tr[\tilde {\bs{V}}_j \tilde {\bs{V}}_j^\top  \tilde {\bs{V}}_j \tilde {\bs{V}}_j^\top] \\
		&= \tr[\tilde {\bs{V}}_j \tilde {\bs{V}}_j^\top] = n - 2p,
	\end{align*}
	and that
	\begin{align*}
		\|\tilde {\bs{V}}_k \tilde {\bs{V}}_k^\top  \bP_k - \mathrm{diag}(\tilde {\bs{V}}_k \tilde {\bs{V}}_k^\top  \bP_k) \|_F^2 &\le \|\tilde {\bs{V}}_k \tilde {\bs{V}}_k^\top  \bP_k \|_F^2 = \mathrm{tr}(\tilde {\bs{V}}_k \tilde {\bs{V}}_k^\top  \bP_k \bP_k^\top \tilde {\bs{V}}_k \tilde {\bs{V}}_k^\top  ) \\
		&= \mathrm{tr}(\tilde {\bs{V}}_k \tilde {\bs{V}}_k^\top ) = n - 2p,
	\end{align*}
	we can apply Lemma~\ref{lem:offdiagbnd} to show that for any constant $\delta > 0$,
	\begin{equation}\label{eq:upp_e_I}
		\lim_{n \to \infty} \prob(|\mathrm{I}| \le \delta) \to 1.
	\end{equation}
	
	For $\mathrm{II}$, given any fixed $\bP_j$ and $\bP_k$, define $a_{n,i}  := \frac{(\tilde {\bs{V}}_j \tilde {\bs{V}}_j^\top  - \tilde {\bs{V}}_k \tilde {\bs{V}}_k^\top \bP_k)_{i,i}}{n}$ and write $V_{n,i} := n a_{n,i} e_i^2$. Then we can rewrite $\mathrm{II}$ as $\mathrm{II} = \frac{1}{n} \sum_{i=1}^n V_{n,i}$. Notice that for each $n$, it holds that $\big | a_{n,i} | \le 2/n$. From this, we can have that $\E[|V_{n,i}|] \le 2$ uniformly for all $i, n$ and that for any $a > 0$,
	\begin{align*}
		\sup_{n \geq 1} \frac{1}{n}\sum_{i=1}^n  \E[|V_{n,i}| \one(|V_{n,i}| > a)] \le \E[2 e_1^2 \one( 2e_1^2 > a)].
	\end{align*}
	Using dominated convergence theorem and that $\E[e_1^2] < \infty$, we have that $\E[2 e_1^2 \one( 2e_1^2 > a)] \to 0$ as $a \to \infty$; this allow us to apply Lemma~\ref{lem:van06} to get that for any constant $\delta > 0$,
	\begin{equation}\label{eq:upp_e_II1}
		\prob\left(|\mathrm{II} - \E[\mathrm{II} \mid \bP_j, \bP_k]| > \delta \mid \bP_j, \bP_k\right) \to 0.
	\end{equation}

	
	Thus, it remains to control $\E[\mathrm{II} \mid \bP_j, \bP_k] = \sum_{i=1}^n a_{n,i}$. 
	We write 
	\begin{equation}\label{eq:ak}
		\bs{A}_k \bs{A}_k^\top  = \bs{I} - \tilde{\bs{V}}_k \tilde{\bs{V}}_k^\top,
	\end{equation}
	where $\bs{A}_k$ is a $n \times (n - 2p)$ matrix with orthonormal columns. Since the column space of $\tilde {\bs{V}}_k$ is at the intersection of $\mathrm{span}(\bX)^\perp$ and $\mathrm{span}(\bP_k \bX)^\perp$, we have that $\mathrm{span}(\bX)$ must be a subspace of $\mathrm{span}(\bs{A}_k)$. Hence without loss of generality we can write $\bs{A}_k := [\bs{A}_0, \bs{B}_k]$, where $\bs{A}_0 \in \R^{n \times p}$ is a matrix with orthonormal columns spanning $\mathrm{span}(\bV_0)^\perp$. With the above notations, we calculate
	\begin{align*}
		\E[\mathrm{II} \mid \bP_j, \bP_k] &= \sum_{i=1}^n a_{n,i} = \frac{1}{n} \tr[\tilde {\bs{V}}_j \tilde {\bs{V}}_j^\top  - \tilde {\bs{V}}_k \tilde {\bs{V}}_k^\top \bP_k] 
		= \frac{1}{n} ((n-2p) - \tr[\bs{A}_k\bs{A}_k^\top \bP_k]) \\
		& = \frac{1}{n}((n-2p) - \tr[\bs{A}_0\bs{A}_0^\top \bP_k + \bs{B}_k \bs{B}_k^\top \bP_k]).
	\end{align*}
	From Assumption~\ref{as:trace}, we have $\tr[\bs{A}_0\bs{A}_0^\top \bP_k] \le \sqrt{2 p} K$, and using Lemma~\ref{lem:tr}, we have $\tr[\bs{B}_k \bs{B}_k^\top \bP_k] \le \tr[\bs{B}_k \bs{B}_k^\top] \le p$, putting together we further have
	\begin{equation}
		\E[\mathrm{II} \mid \bP_j, \bP_k] \ge \frac{1}{n}((n - 2p) - p - \sqrt{2 p} K) \ge \frac{m}{4 + m}, \label{eq:explower}
	\end{equation}
	where the last inequality holds for sufficiently large $n$. 
	From above and~\eqref{eq:upp_e_II1}, and also our control of the term $\mathrm{I}$ in~\eqref{eq:upp_e_I}, we have that the first statement of~\eqref{eq:upp_e} holds with probability converging to $1$. Using an analogous argument we prove the second statement of \eqref{eq:upp_e}. In light of this and our analysis of~\eqref{eq:upp_epsilon}, we obtain the desired result.
\end{proof}


\subsection{Proof of Theorem~\ref{thm:alternative2}}\label{sec:upp_asymp_low}

\begin{lemma}\label{lem:partconv}
	Consider a deterministic permutation matrix $\bP \in \R^{n \times n}$ that varies with $n$ and $\tr[\bP] = 0$. We have that for any fixed $\delta > 0$,
	\[
	\forall \prob_e \in \cD_1, \quad \lim_{n \to \infty} \prob(|\be^\top \bP \be| / n > \delta) = 0.
	\].
\end{lemma}
\begin{proof}
	Let $\sigma$ be the permutation corresponding to $\bP$. From Lemma~\ref{lem:partition}, we have there exists a partition $U_1, U_2, U_3$ with $|U_j \cap \sigma(U_j)| = 0$ and that $|U_j| \ge \frac{n}{4} - 1$ for $j = 1, 2, 3$ such that
	\[
	\frac{\be^\top \bP \be}{n} = \frac{1}{n}\sum_{j = 1}^3 \sum_{i \in U_j}e_i e_{\sigma(i)}.
	\]
	Then 
	\[
	\lim_{n \to \infty} \prob(|\be^\top \bP \be| / n > \delta) \le \sum_{j=1}^3 \prob\left(\frac{1}{|U_j|} \left|\sum_{i \in U_j}e_i e_{\sigma(i)}\right| > \frac{\delta}{3}\right).
	\]
	From above, it remains to prove that for any $j$ and any fixed $\delta > 0$,
	\[
	\prob\left(\frac{1}{|U_j|} \left|\sum_{i \in U_j}e_i e_{\sigma(i)}\right| > \delta\right) \to 0.
	\]
	
	Let $\tilde{e}_i$ be a sequence of i.i.d. random variables that is independent from $\bw$ and that $\tilde{e}_i \stackrel{\mathrm{d}}{=} e_1 e_2$. Then we easily have that $\tilde{e}_i$ are i.i.d. random variables with zero mean and bounded first order moment. Then using the weak law of large number, we have that with $a_n(\delta) := \sup_{m \ge n} \prob(|\sum_{i = 1}^m\tilde{e}_i / m| > \delta)$,
	\[
	\lim_{n \to \infty} a_n(\delta) = \limsup_{n \to \infty} \prob\left(\left|\sum_{i = 1}^n\tilde{e}_i / n\right| > \delta\right) = \lim_{n \to \infty} \prob\left(\left|\sum_{i = 1}^n\tilde{e}_i / n\right| > \delta\right) = 0.
	\]
	Using that the $U_j$ and $\sigma(U_j)$ has no overlap, we have
	\[
	\sum_{i = 1}^{|U_j|}\tilde{e}_i \stackrel{\mathrm{d}}{=} \sum_{i \in U_j} e_i e_{\sigma(i)}
	\]
	and thus
	\[
	\prob\left(\frac{1}{|U_j|} \left|\sum_{i \in U_j}e_i e_{\sigma(i)}\right| > \delta\right) \le a_{|U_j|}(\delta) \le a_{\lceil n / 4 - 1 \rceil}(\delta) \to 0,
	\]
	where for the last inequality we use that $a_n(\delta)$ is non-increasing and $|U_j| \ge n / 4 - 1$.
\end{proof}

\begin{lemma}
	\label{lem:symbound}
	Assume that $\prob_e$ follows a distribution that is symmetric around zero; and let $\bs{U} \in \R^{n \times n}$ be a positive semi-definite matrix. 
	Then we have that for any $\delta > 0$,
	\[
	\prob\left(\be^\top \bs{U} \be > \delta \|\be\|_2^2\right) \leq \frac{\tr[\bs{U}]}{\delta n}.
	\]
\end{lemma}

\begin{proof}
	Let  $\bs{J}$ be a random diagonal matrix where all diagonal entries $\bs{J}_{i,i}$ are i.i.d. binary random variables with $\prob(\bs{J}_{i,i} = 1) = \prob(\bs{J}_{i,i} = -1) = \frac{1}{2}$. We write $\bP$ for a uniformly random permutation matrix that is independent from $\bs{J}$. Since $\prob_e$ is symmetric and all the $e_i$'s are independent, we have that $\be \overset{d}{=} \bP \bs{J} \be$, i.e., they are equal in distribution. 
	
	This allows us to prove the statement by controlling $\prob (\be^\top \bs{J}^\top \bs{P}^\top \bs{U}\bP \bs{J}\be > \delta \|\be\|_2^2)$ due to that
	\begin{align}\label{eq:eqdbnd}
		\prob\left(\be^\top \bs{U} \be \geq \delta \|\be\|_2^2\right) & = \prob\left(\be^\top \bs{J}^\top \bs{P}^\top \bs{U}\bP \bs{J}\be > \delta \be^\top \bs{J}^\top \bP^\top \bP \bs{J} \be\right) \nonumber\\
		& = \prob\left(\be^\top \bs{J}^\top \bs{P}^\top \bs{U}\bP \bs{J}\be > \delta \|\be\|_2^2\right).
	\end{align}
	
	First, for any fixed $\be_0 \in \R^n$, we have
	\[
	\E [\be^\top \bs{J}^\top \bP^\top \bs{U} \bP \bs{J} \be | \be = \be_0] = \be_0^\top  \E [\bs{J}^\top \bP^\top \bs{U} \bP \bs{J}] \be_0.
	\]
	
	Second, for any fixed matrix $\bM \in \R^{n \times n}$, we have $\E [(\bs{J}^\top \bM \bs{J})_{i,j}] = \E [\bs{J}_{i,i}\bM_{i,j}\bs{J}_{j,j}] = 0$ whenever $i \neq j$ and $\E [(\bs{J}^\top \bM \bs{J})_{i,i}] = \bE [\bs{J}_{i,i}\bM_{i,i}\bs{J}_{i,i}] = \bM_{i,i}$. Putting together and applying Lemma~\ref{lem:permdiag}, we have
	\begin{align*}
		\E [\be^\top \bs{J}^\top \bP^\top \bs{U} \bP \bs{J} \be | \be = \be_0] = \be_0^\top  \E [\bs{J}^\top \bP^\top \bs{U} \bP \bs{J}] \be_0  = \frac{\tr(\bs{U})}{n} \|\be_0\|_2^2.
	\end{align*}
	
	From above and Markov's inequality, we have 
	\begin{align*}
		\prob (\be^\top \bs{J}^\top \bs{P}^\top \bs{U}\bP \bs{J}\be > \delta \|\be\|_2^2) & = \E \left[\prob \left(\be^\top \bs{J}^\top \bs{P}^\top \bs{U}\bP \bs{J}\be > \delta \|\be\|_2^2 \mid \be \right)\right]\\
		& \leq \E \left[\frac{\E[\be^\top \bs{J}^\top \bs{P}^\top \bs{U}\bP \bs{J}\be \mid \be]}{\delta \|\be\|_2^2}\right] \\
		& = \E \left[\frac{\tr[\bs{U}]}{ \delta n}\right] = \frac{\tr[\bs{U}]}{\delta n}.
	\end{align*}
	In light of the above equality and~\eqref{eq:eqdbnd}, we obtain the desired result.
\end{proof}

\begin{lemma}\label{lem:partition}
	Consider a permutation $\sigma$ of $\{1, \ldots, n\}$ such that for any $i \in \{1, \ldots, n\}$, $\sigma(i) \neq i$. Then there exists a partition $U_1, U_2, U_3$ of the set $\{1, \ldots, n\}$ such that 
	\[
	\forall j \in \{1,2,3\}, \quad |U_j| \in \left[\frac{n}{4} - 1, \frac{n}{2} + 1\right] \;\&\; |U_j \cap \sigma(U_j)| = 0.
	\]
\end{lemma}

\begin{proof}
	Let $G$ be a directed graph on vertices $\{1, \ldots, n\}$ where there exists a directed edge $i \to j$ in $G$ if and only if $j = \sigma(i)$.  Then the cycles in $G$ are of length at least $2$.
	
	Let $U$ denote a set with the maximum number of notes such that $|U \cap \sigma(U)| = 0$, then apparently $|U| < \frac{n}{2} + 1$. Let $G'$ denote the subgraph of $G$ removing all the edges of the type $(u, \sigma(u))$ for $u \in U$. Then we must have that a node is in $U^c$ if and only if the node has an out edge in $G'$. Moreover, we claim that (i) $G'$ does not contain a circle with length $2$; (ii) all the connected component of $G'$ has no more than $2$ edges. To prove claim (i), suppose in contradiction there exists a circle $a \to b \to a$ in $G'$, then we must have that $a,b \not\in U$. This means that the set $U' = U \cup \{b\}$ can still satisfy that $|U' \cap \sigma(U')| = 0$, which contradicts that $U$ is maximal. To prove claim (ii), suppose in contradiction there exists a connected component with at least $3$ edges, then in this component there must exists a path $a \to b \to c \to d$ or $a \to b \to c \to a$. Then we easily have that $b, c \not\in U$. This means that the set $U' = U \cup \{b\}$ can still satisfy that $|U' \cap \sigma(U')| = 0$, which contradicts that $U$ is maximal.
	
	From the two claims, we must have that all the connected components in $G'$ must be of the form $a \to b$ or $a \to b \to c$. We now introduce three sets of nodes $A, B, C$, where $A$ consists of all the nodes $a$ such that $a \to b$ formalizes a connected component in $G'$; $B$ consists of all the nodes $a$ such that $a \to b \to c$ is a connected component in $G'$; and $C$ consists of all the nodes $b$ such that $a \to b \to c$ is a connected component in $G'$. Now recall the claim that a node is in $U^c$ if and only if the node has an out edge in $G'$, we have that the four disjoint sets $A, B, C, U$ formalizes a partition of all the nodes; moreover, $\sigma(A) \subseteq U$, $\sigma(B) = C$, $\sigma(C) \subseteq U$, $\sigma(U) = A \cup B$.
	
	From above, we split $A$ into two sets $A_1, A_2$ with size $|A_1|$ and $|A_2|$ differ by at most $1$; and set $U_1 = U, U_2 = A_1 \cup B, U_3 = A_2 \cup C$. Then it is straightforward that for all $i = 1, 2, 3$, 
	\[
	\frac{n}{4} - 1 \le \frac{n - |U_1| - 1}{2} \le |U_i| \le |U_1| \le \frac{n}{2} + 1
	\]
	and that
	\[
	|U_i \cap \sigma(U_i)| = 0,
	\]
	which proves the desired result.
\end{proof}

\begin{proof}[Proof of Theorem~\ref{thm:alternative2}]
	Without loss of generality, we assume throughout that $\prob_e \in \cD_1$ and $\prob_\varepsilon \in \cD_{1 + t}$. 
	Following analogous argument as in the proof of Theorem~\ref{thm:alternative1}, we tackle this problem via proving that for any fixed $\delta > 0$,
	\begin{equation}\label{eq:upp_epsilon_2}
		\begin{aligned}
			\prob\left(\exists 1 \le k \le K \;\st\; \frac{|\be^\top \tilde{\bs{V}}_k \tilde{\bs{V}}_k^\top \bv|}{b \| \be \|^2_2} \ge \delta\right) & \to 0 ;\\
			\prob\left(\exists 1 \le k \le K \;\st\; \frac{|\be^\top  \tilde {\bs{V}}_k \tilde {\bs{V}}_k^\top \bP_k \bv|}{b \| \be \|^2_2} \ge \delta\right) & \to 0 ;
		\end{aligned}
	\end{equation}
	and that with probability converging to $1$, for all $j, k$,
	\begin{equation}\label{eq:upp_e_2}
		\begin{aligned}
			& \frac{\be^\top  \tilde {\bs{V}}_j \tilde {\bs{V}}_j^\top  \be - \be^\top  \tilde {\bs{V}}_k \tilde {\bs{V}}_k^\top  \bP_k \be}{\|\be \|^2}  \ge \frac{1}{5};\\
			& \frac{\be^\top  \tilde {\bs{V}}_j \tilde {\bs{V}}_j^\top  \be + \be^\top  \tilde {\bs{V}}_k \tilde {\bs{V}}_k^\top  \bP_k \be}{\|\be \|^2}  \ge \frac{1}{5}.
		\end{aligned}
	\end{equation}
	
	To prove \eqref{eq:upp_epsilon_2}, since $\prob_e \in \cD_1$, we have that there exists some threshold $\tau > 0$ such that $\frac{2}{3}\le \E[|e_i|^2 \one(|e_i| \le \tau)] < \infty$, from this and by standard results of weak law of large number, we have
	\[
	\lim_{n \to \infty} \pr\left(\frac{1}{n} \sum_{i = 1}^n |e_i|^2 \one(|e_i| \le \tau) \ge \frac{1}{2}\right) = 1,
	\]
	which also means that
	\begin{equation}\label{eq:e1mom}
		\lim_{n \to \infty} \pr\left(\frac{1}{n} \sum_{i = 1}^n e_i^2 \ge \frac{1}{2}\right) = 1.
	\end{equation}
	Denote the above event by $\mathcal{E}$; then we have
	\[
	\prob\left(\exists 1 \le k \le K \;\st\; \frac{|\be^\top \tilde{\bs{V}}_k \tilde{\bs{V}}_k^\top \bv|}{b \| \be \|^2_2} \ge \delta \mid \mathcal{E} \right) = 
	\prob\left(\exists 1 \le k \le K \;\st\; \frac{|\be^\top \tilde{\bs{V}}_k \tilde{\bs{V}}_k^\top \bv|}{b \max\{\| \be \|^2_2, \frac{1}{2} n\}} \ge \delta \mid \mathcal{E} \right).
	\]
	\eqref{eq:upp_epsilon_2} then follows from exactly the same proof as that of~\eqref{eq:upp_epsilon}.
	
	
	
	In the rest of the proof we assume throughout that both $\bP_j$ and $\bP_k$ are fixed permutation matrices or equivalently being conditioned on. Since we assume $K$ is fixed, we only need to prove that for any fixed $j, k$, with probability converging to $1$, the two inequalities in~\eqref{eq:upp_e_2} hold. In the rest of this proof we prove the first inequality of \eqref{eq:upp_e_2}, and the second inequality follows from an analogous argument. To prove this, let $\be'$ denote an independent replication of $\be$. Recall the definition of $\bs{A}_k$ in~\eqref{eq:ak}, we have
	\begin{equation}\label{eq:upp_e_21}
		\begin{aligned}
			& (\be-\be')^\top (\tilde {\bs{V}}_j \tilde {\bs{V}}_j^\top  - \tilde {\bs{V}}_k \tilde {\bs{V}}_k^\top \bP_k)(\be-\be') \\ 
			& = \|\be-\be'\|_2^2 - (\be - \be')^\top \bP_k(\be - \be') - (\be-\be')^\top (\bs{A_j}\bs{A_j}^\top  - \bs{A_k}\bs{A_k}^\top \bP_k)(\be - \be') \\
			& \ge \|\be-\be'\|_2^2 -  (\be - \be')^\top \bP_k(\be - \be') - (\be-\be')^\top \left(\bs{A_j}\bs{A_j}^\top  + \frac{\bs{A_k}\bs{A_k}^\top  + \bP_k^\top \bs{A_k}\bs{A_k}^\top \bP_k}{2}\right)(\be - \be'),
		\end{aligned}
	\end{equation}
	where for the last inequality we apply Cauchy-Schwartz inequality. As $e_i - e_i'$ is symmetric around zero, we have from Lemma~\ref{lem:symbound} that the following event $\mathcal{E}_1$ holds with probability $1 - \frac{10p}{n} \to 1$:
	\begin{align*}
		\mathcal{E}_1 := \left\{(\be-\be')^\top (\bs{A_j}\bs{A_j}^\top  + \frac{\bs{A_k}\bs{A_k}^\top  + \bP_k^\top \bs{A_k}\bs{A_k}^\top \bP_k}{2})(\be - \be')  < \frac{1}{5}(\be-\be')^\top (\be-\be')\right\}.
	\end{align*}
	
	In addition, as $\|\tilde {\bs{V}}_j \tilde {\bs{V}}_j^\top  - \tilde {\bs{V}}_k \tilde {\bs{V}}_k^\top \bP_k\|_{\mathrm{op}} \le 2$, we have from Lemma~\ref{lem:tcomb} that the following two events $\mathcal{E}_2$ and $\mathcal{E}_3$ hold with probability converging to $1$:
	\begin{align*}
		\mathcal{E}_2 &\; := \left\{ \left|\be'^\top \left(\tilde {\bs{V}}_j \tilde {\bs{V}}_j^\top  - \tilde {\bs{V}}_k \tilde {\bs{V}}_k^\top \bP_k \right)\be\right| < \frac{1}{5} \|\be\|_2^2 \right\};\\
		\mathcal{E}_3 &\; := \left\{\left|\be'^\top \left(\tilde {\bs{V}}_j \tilde {\bs{V}}_j^\top  - \tilde {\bs{V}}_k \tilde {\bs{V}}_k^\top \bP_k \right)\be\right| < \frac{1}{5} \|\be\|_2^2 \right\}.
	\end{align*} 
	
	Working on the intersection of the three events $\mathcal{E}_1 \cap \mathcal{E}_2 \cap \mathcal{E}_3$, and applying the decomposition
	\begin{align*}
		(\be-\be')^\top (\tilde {\bs{V}}_j \tilde {\bs{V}}_j^\top  - \tilde {\bs{V}}_k \tilde {\bs{V}}_k^\top \bP_k) (\be-\be') & = \be^\top (\tilde {\bs{V}}_j \tilde {\bs{V}}_j^\top  - \tilde {\bs{V}}_k \tilde {\bs{V}}_k^\top \bP_k) \be + \be'^\top (\tilde {\bs{V}}_j \tilde {\bs{V}}_j^\top  - \tilde {\bs{V}}_k \tilde {\bs{V}}_k^\top \bP_k) \be' \\
		& - \be^\top (\tilde {\bs{V}}_j \tilde {\bs{V}}_j^\top  - \tilde {\bs{V}}_k \tilde {\bs{V}}_k^\top \bP_k) \be' - \be'^\top (\tilde {\bs{V}}_j \tilde {\bs{V}}_j^\top  - \tilde {\bs{V}}_k \tilde {\bs{V}}_k^\top \bP_k) \be,
	\end{align*}
	we have from~\eqref{eq:upp_e_21} that
	\begin{equation}\label{eq:upp_e_22}
		\begin{aligned}
			&\be^\top (\tilde {\bs{V}}_j \tilde {\bs{V}}_j^\top  - \tilde {\bs{V}}_k \tilde {\bs{V}}_k^\top \bP_k )\be + \be'^\top (\tilde {\bs{V}}_j \tilde {\bs{V}}_j^\top  - \tilde {\bs{V}}_k \tilde {\bs{V}}_k^\top \bP_k)\be' \\ 
			& \ge \frac{4}{5} \|\be - \be'\|_2^2 -  (\be - \be')^\top \bP_k(\be - \be') - \frac{1}{5} (\|\be\|_2^2 + \|\be'\|_2^2) \\
			& = \frac{3}{5} (\|\be\|_2^2 + \|\be'\|_2^2) - (\be - \be')^\top \bP_k(\be - \be') - \frac{8}{5} \be^\top \be'.
		\end{aligned}
	\end{equation}
	
	Define random events
	\begin{align*}
		\mathcal{E}_4  := \left\{(\be - \be')^\top \bP_k(\be - \be') \le \frac{1}{5}\|\be\|_2^2\right\}, \quad
		\mathcal{E}_5  := \left\{\be^\top \be' \le \frac{1}{8} \|\be\|_2^2\right\}.
	\end{align*}
	
	For $\mathcal{E}_4$, we have
	\begin{align*}
		\prob\left(\mathcal{E}_4^c\right) & \le \prob\left(\mathcal{E}_4^c \;\&\; \|\be\|_2^2 \ge n / 2\right) + \prob\left(\|\be\|_2^2 < n / 2\right) \\
		&  = \prob\left(\left\{(\be - \be')^\top \bP_k(\be - \be') > \frac{1}{5}\|\be\|_2^2\right\} \;\&\; \|\be\|_2^2 \ge n / 2\right) + \prob\left(\|\be\|_2^2 < n / 2\right) \\
		& \le \prob\left(\left\{(\be - \be')^\top \bP_k(\be - \be') > \frac{n}{10}\right\} \;\&\; \|\be\|_2^2 \ge n / 2\right) + \prob\left(\|\be\|_2^2 < n / 2\right) \\
		& \le \prob\left(\left\{(\be - \be')^\top \bP_k(\be - \be') > \frac{n}{10}\right\}\right) + \prob\left(\|\be\|_2^2 < n / 2\right)
	\end{align*}
	Then using Lemma~\ref{lem:partconv} and~\eqref{eq:e1mom}, we have that the event $\mathcal{E}_4$ holds with probability converging to $1$.
	
	For $\mathcal{E}_5$, using that all the $e_i e_i'$'s are i.i.d. random variables with $\E[|e_i e_i'|] = \E[|e_i|] \E[|e_i'|] < \infty$, we have $\be^\top \be' / n \to 0$ in probability; thus using a similar argument as $\mathcal{E}_4$, we have $\mathcal{E}_5$ holds with probability converging to $1$. 
	
	Now working on the event $\mathcal{E}_1 \cap \cdots \cap \mathcal{E}_5$ (which, as shown above, occurs with probability converging to $1$), we have from~\eqref{eq:upp_e_22} and also the definitions of $\mathcal{E}_4$ and $\mathcal{E}_5$ that 
	\[
	\be^\top (\tilde {\bs{V}}_j \tilde {\bs{V}}_j^\top  - \tilde {\bs{V}}_k \tilde {\bs{V}}_k^\top \bP_k )\be + \be'^\top (\tilde {\bs{V}}_j \tilde {\bs{V}}_j^\top  - \tilde {\bs{V}}_k \tilde {\bs{V}}_k^\top \bP_k)\be' \ge \frac{1}{5}(\|\be\|_2^2 + \|\be'\|_2^2).
	\]
	In other words, with probability converging to zero,
	\[
	\underset{\mathrm{I}}{\underbrace{\be^\top \left(\tilde {\bs{V}}_j \tilde {\bs{V}}_j^\top  - \tilde {\bs{V}}_k \tilde {\bs{V}}_k^\top \bP_k - \frac{1}{5}\bs{I}\right)\be}} + \underset{\mathrm{I}'}{\underbrace{\be'^\top \left(\tilde {\bs{V}}_j \tilde {\bs{V}}_j^\top  - \tilde {\bs{V}}_k \tilde {\bs{V}}_k^\top \bP_k - \frac{1}{5}\bs{I}\right)\be'}} < 0.
	\] 
	Since $\mathrm{I}$ and $\mathrm{I}'$ are two i.i.d. random variables, we have using their independence and identically distributed property that
	\[
	\prob(\mathrm{I} < 0) = \sqrt{\prob(\mathrm{I} < 0) \prob(\mathrm{I}' < 0)} = \sqrt{\prob(\mathrm{I} < 0, \mathrm{I}' < 0)} \le \sqrt{\prob(\mathrm{I} + \mathrm{I}' < 0)} \to 0,
	\]
	which proves~\eqref{eq:upp_e_2}. In light of this and our control of~\eqref{eq:upp_epsilon_2}, we prove the desired result.
\end{proof}

\subsection{Auxiliary lemmas}

\begin{lemma}\label{lem:permdiag}
	Let $\bP$ be a uniformly random permutation matrix. Let $\bM \in \R^{n \times n}$ be a fixed $n\times n$ matrix. Then for any $i = 1, \ldots, n$, $\E[(\bP \bM \bP^\top)_{ii}] = \frac{1}{n} \sum_{j = 1}^n \bM_{jj}$.
\end{lemma}

\begin{proof}
	Let $\sigma$ be the random permutation corresponding to $\bP$, we have
	\[
	\E [(\bP^\top \bM \bP)_{i,i}] = \bE [\bM_{\sigma{(i)}, \sigma{(i)}}] = \frac{1}{n} \sum_j \bM_{j,j},
	\]
	where the second inequality is due to that $\sigma(i)$ can be viewed as a random variable that samples uniformly at random from the set $\{1, \ldots, n\}$.
\end{proof}

\begin{lemma}\label{lem:tr}
	Consider a symmetric positive semi-definite matrix $\bM \in \R^{n \times n}$ and a permutation matrix $\bP \in \R^{n \times n}$, we have
	\[
	\tr[\bM \bP] \le \tr[\bM].
	\]
\end{lemma}

\begin{proof}
	Using the positive semi-definiteness and symmetry of $\bM$, we have for any $i,j$ ($i$ and $j$ can be equal or unequal), 
	\[
	\bM_{i,j} \le \frac{\bM_{i,i} + \bM_{j,j}}{2}.
	\]
	Let $\sigma$ be the permutation associated with $\bP$, we have
	\[
	\tr[\bM \bP] = \sum_{i = 1}^n \bM_{i, \sigma(i)} \le \sum_{i = 1}^n \frac{\bM_{i,i} + \bM_{\sigma(i),\sigma(i)}}{2} = \tr[\bM],
	\]
	which proves the desired result.
\end{proof}

%
%
%

\subsection{Theoretical analysis of the algorithms}

We will first show an lemma.

\begin{lemma}
	\label{lem:permtr}
	Consider a fixed matrix $\bM \in R^{n \times n}$ with $n \ge 2$ and a fixed permutation matrix $\bP_0 \in R^{n \times n}$ satisfying $\tr[\bP_0] = 0$. Let $\tilde{\bP} \in R^{n \times n}$ be a uniformly randomly sampled permutation matrix and define $\bP := \tilde{\bP}^{-1} \bP_0 \tilde{\bP}$. Then for any $\delta > 0$, we have that 
	\[
	\prob\left(|\tr[\bM \bs P]| \ge \frac{\sqrt{2 \tr[\bM \bM^\top]}}{\sqrt \delta}\right) \le \delta.
	\]
\end{lemma}
\begin{proof}
	Let $\tilde{\sigma}$ be the random permutation corresponding to $\tilde{\bP}$. Then we have that for any $\bP_{u, v}$, $\bP_{u, v} = 1$ if and only if $(\bP_0)_{\tilde{\sigma}(u), \tilde{\sigma}(v)} = 1$. Now that since $\tilde{\sigma}$ is a uniformly random permutation, we have that $(\tilde{\sigma}(u), \tilde{\sigma}(v))$ is a pair that is uniformly at random drawn from the set $\{(i,j) \mid i \neq j \in \{1, \ldots, n\}\}$. From this, we have for any fixed $(u, v)$,
	\[
	\prob(\bP_{u, v} = 1) = \prob((\bP_0)_{\tilde{\sigma}(u), \tilde{\sigma}(v)} = 1) = \frac{n}{n^2 - n} = \frac{1}{n - 1},
	\]
	and equivalently, $\E[\bP_{u, v}^2] = \E[\bP_{u, v}] = \frac{1}{n - 1}$.
	
	Notice also that since $\bP$ is a random permutation matrix, we have that for any fixed $u$ and any fixed $v_1 \neq v_2$, almost surely $\bP_{u, v_1} \bP_{u, v_2} = 0$.
	
	Putting together, we have
	\begin{align*}
		\E [\tr[\bM \bP]^2] &= \bE \left[\left(\sum_u \sum_v \bM_{u,v} \bP_{u,v}\right)^2\right] \le n \sum_u  \E \left[\left( \sum_v \bM_{u,v} \bP_{u,v}\right)^2\right] \\
		& = n\sum_u \sum_v \E [\bM_{u,v}^2 \bP_{u,v}^2] = \frac{n}{n - 1}\sum_u \sum_v \bM_{u,v}^2 = \frac{n}{n - 1} \tr[\bM\bM^\top] \le 2 \tr[\bM\bM^\top].
	\end{align*}
	
	From above, the desired result follows from Chebyshev's inequality.
\end{proof}

\begin{proof}[Proof of Proposition~\ref{prop:algorithm}]
	Throughout the proof we only consider the case with number of iterations $T = 1$, and the case of $T \ge 2$ can be proven via analogous argument. Let $\bP_\pi$ be the random permutation matrix associated with permutation $\pi$, and let $\tilde{\bP}_k$ be the permutation matrix associated with $\tilde{\sigma}_k$, then we have that
	\[
	\bP_k = \bP_\pi^{-1} \tilde{\bP}_k \bP_\pi,
	\]
	so that for any $k_1, k_2 \in \{1, \ldots, K\}$, we have that by setting $k_3$ as the remainder after dividing $k_1 + k_2$ by $K + 1$, 
	\[
	\bP_{k_1} \bP_{k_2} = \bP_\pi^{-1} \tilde{\bP}_{k_1} \bP_\pi \bP_\pi^{-1}\tilde{\bP}_{k_2}\bP_\pi = \bP_\pi^{-1} \tilde{\bP}_{k_1} \tilde{\bP}_{k_2}\bP_\pi = \bP_\pi^{-1} \tilde{\bP}_{k_3} \bP_\pi = \bP_{k_3}.
	\]
	This proves that the returned $\mathcal{P}_K$ always satisfies Assumption~\ref{as:permutation}.
	
	In addition, we have from Lemma~\ref{lem:permtr} and $\tr[\bV_0 \bV_0^\top \bV_0 \bV_0^\top] = \tr[\bV_0 \bV_0^\top] = p$ that for any $k$,
	\[
	\prob\left(|\tr[\bV_0 \bV_0^\top \bP_k]| \ge \sqrt{2 p} K \right) \le \frac{1}{K^2}.
	\]
	The desired result then follows by applying a union bound for all $k$.
	
	Note that since Algorithm~\ref{alg:perm} returns with non-zero probability, there must exist a $\mathcal{P}_K$ that satisfies both assumptions.
\end{proof}

{\rev 
	\section{Proof of additional power results}\label{sec:pfaddpower}
	
	\subsection{Proof of Theorem~\ref{thm:althete}}
	
	The following is an extension of Marcinkiewicz-Zygmund strong law of large numbers to the sum of non-i.i.d. random variables.
	
	\begin{lemma}
		\label{lem:althete1}
		Consider the $\bv$ in Assumption~\ref{as:hete}. 
		If $t < 1$, for any constant $\delta > 0$, it holds that, 
		\begin{align*}
			\lim_{n \to \infty} \pr(\|\bv\|_2^2 \ge \delta \cdot n^{\frac{2}{1+t}}) = 0. 
		\end{align*}
	\end{lemma}
	
	\begin{proof}
		Without loss of generality, we assume throughout this proof that $C_\varepsilon = 1$. For any constant $\delta, \epsilon > 0$, set $B := \epsilon^{\frac{1}{1-t}} \cdot \delta^{\frac{1}{1 -t}} / 6^{\frac{1}{1 -t}}$. 
		Define $f_i = \varepsilon_i \one(|\varepsilon_i| \le B i^{\frac{1}{t + 1}})$, by Assumption~\ref{as:hete}, we have that $ \sum_{i = 1}^{\infty} \Pr(f_i \neq \varepsilon_i ) < \infty$. Thus given $\epsilon > 0$, there exists an integer $N_{\epsilon}$ such that 
		\begin{align*}
			\sum_{i = N_{\epsilon}}^\infty \pr( f_i \neq \varepsilon_i) < \frac{\epsilon}{3}.
		\end{align*}
		Moreover, given $N_{\epsilon}$ and $\epsilon$, we have from Markov's inequality that exists a constant $M_{\epsilon} > 0$ such that
		\begin{align*}
			\sum_{i = 1}^{N_{\epsilon}} \pr(|\varepsilon_i| \ge M_{\epsilon}) \le \frac{\epsilon}{3}.
		\end{align*}
		
		Define the two random events 
		\[
		\mathcal{E}_1 := \{ \forall N_{\epsilon} \le i \le n, f_i = \varepsilon_i\} \quad \& \quad \mathcal{E}_2 := \{ \forall i \le N_{\epsilon},  |\varepsilon_i| \le M_{\epsilon}\}.
		\]
		Now for any $n \ge N_{\epsilon}$, we then have 
		\begin{align*}
			\pr(\|\bv\|_2^2 \ge \delta \cdot n^{\frac{2}{1+t}}) &\le \pr(\|\bv\|_2^2 \ge \delta \cdot n^{\frac{2}{1+t}} \;\&\; \mathcal{E}_1 \;\&\; \mathcal{E}_2 )  + \pr(\mathcal{E}_1^c) + \pr(\mathcal{E}_2^c) \\
			&\le \pr(\|\bv\|_2^2 \ge \delta \cdot n^{\frac{2}{1+t}} \;\&\; \mathcal{E}_1 \;\&\; \mathcal{E}_2 ) + \frac{2\epsilon}{3}.
		\end{align*}
		
		Under $\mathcal{E}_1 \;\&\; \mathcal{E}_2$, it holds that $\sum_{i = 1}^n \varepsilon_i^2 \le \sum_{i = 1}^n f_i^2 + N_{\epsilon} M_{\epsilon}^2$, whence
		\begin{align*}
			\pr(\|\varepsilon\|_2^2 \ge \delta \cdot n^{\frac{2}{1+t}} \;\&\;  \mathcal{E}_1 \;\&\;  \mathcal{E}_2 ) &\le  \pr\left(\sum_{i = 1}^n f_i^2 + N_{\epsilon} M_{\epsilon}^2 \ge \delta \cdot n^{\frac{2}{1+t}} \;\&\; \mathcal{E}_1 \;\&\; \mathcal{E}_2 \right) \\
			&\le \pr\left(\sum_{i = 1}^n f_i^2 + N_{\epsilon} M_{\epsilon}^2 \ge \delta \cdot n^{\frac{2}{1+t}}\right).
		\end{align*}
		
		Now it remains to understand the concentration of $\sum_{i = 1}^n f_i^2$. We have 
		\begin{equation*}
			\begin{aligned}
				\sum_{i = 1}^n \bE [f_i^2]
				& \le  \sum_{i = 1}^{n} \bE [\varepsilon_i^2 \one(|\varepsilon_i| \le Bi^{\frac{1}{1+t}})] \le   \sum_{i = 1}^{n} \bE [\varepsilon_i^2 \one(|\varepsilon_i| \le B n^{\frac{1}{1+t}})] \\
				& = \sum_{i = 1}^{n} \bE [|\varepsilon_i|^{1 + t} |\varepsilon_i|^{1 - t} \one(|\varepsilon_i| \le B n^{\frac{1}{1+t}})] \le B^{1 - t} n^{\frac{1 - t}{1 +t}} \sum_{i = 1}^{n} \bE [|\varepsilon_i|^{1 + t} \one(|\varepsilon_i| \le B n^{\frac{1}{1+t}})] \\
				&\le  B^{1 - t} n^{\frac{2}{1 +t}} = \frac{\epsilon \delta}{6} n^{\frac{2}{1 +t}},
			\end{aligned}
		\end{equation*}
		where for the last inequality we apply that $\E[|\varepsilon_i|^{1 + t}] \le 1$ for all $i$. In light of the above and by Markov's inequality, we have for $n \ge \max\{N_{\epsilon}, ( \frac{6 N_{\epsilon} M_{\epsilon}^2}{\epsilon\delta})^{\frac{1+t}{2}}\}$,
		\begin{align*}
			\pr(\sum_{i = 1}^n f_i^2 + N_{\epsilon} M_{\epsilon}^2 \ge \delta \cdot n^{\frac{2}{1+t}}) \le& \frac{\sum_{i = 1}^n \E[f_i^2] +  N_{\epsilon} M_{\epsilon}^2
			}{\delta \cdot n^{\frac{2}{1+t}}} 
			\le  \frac{\epsilon}{3}. 
		\end{align*}
		
		Concluding, we have that given any fixed $\epsilon, \delta > 0$, for all $n \ge \max\{N_{\epsilon}, ( \frac{6 N_{\epsilon} M_{\epsilon}^2}{\epsilon\delta})^{\frac{1+t}{2}}\}$,
		\begin{align*}
			\pr(\|\bv\|_2^2 \ge \delta \cdot n^{\frac{2}{1+t}}) \le \epsilon.
		\end{align*}
		The proof is then complete.
	\end{proof}
	
	\begin{lemma}\label{lem:tcombhete}
		Consider the $\bM \in \R^{n \times n}$ with $\|\bM\|_\op \le 1$ and random vectors $\bv$ and $\be$ satisfying Assumption~\ref{as:hete}. Let $t \in [0, 1]$ be given and assume that $b$ satisfies~\eqref{eq:bn}. Then for any fixed $\delta > 0$, 
		\[
		\lim_{n \to \infty} \pr\left(\frac{|\be^\top \bM \bv|}{b n} > \delta\right) = 0.
		\]
	\end{lemma}
	
	\begin{proof}
		When $t = 1$, the result is a direct consequence from Chebyshev's inequality since both $e_i$'s and $\varepsilon_i$'s are independent random variables with uniformly bounded second order moments. We now focus on the $t  < 1$ case. In this case, we have for any fixed $\delta, \delta' > 0$, by Chebyshev's inequality and union bound, 
		\begin{align*}
			\pr\left(\frac{|\be^\top \bM \bv|}{b n} > \delta \right) & \le \pr \left(\frac{|\be^\top \bM \bv|}{b n} > \delta \;\&\; \|\bv\|_2^2 \le \delta' n^{\frac{2}{1 + t}} \right) + \pr\left(\|\bv\|_2^2 > \delta' n^{\frac{2}{1 + t}} \right) \\
			& \le \frac{\E[|\be^\top \bM \bv|^2 \one\{\|\bv\|_2^2 \le \delta' n^{\frac{2}{1 + t}}\}]}{\delta^2 b^2 n^2}  + \pr\left(\|\bv\|_2^2 > \delta' n^{\frac{2}{1 + t}} \right) \\
			& \le \frac{C_e \delta' \cdot n^{\frac{2}{1 + t}}}{\delta^2 b^2 n^2}  + \pr\left(\|\bv\|_2^2 > \delta' n^{\frac{2}{1 + t}} \right),
		\end{align*}
		where $\delta' > 0$ can be an arbitrary constant that does not depend on $n$. From above, and that 
		\[
		\liminf_{n \to \infty} b / n^{-\frac{t}{1 + t}} > 0,
		\]
		we have that there exists a constant $C$ such that for any fixed $\delta, \delta' > 0$, with $n$ sufficiently large,
		\begin{equation}\label{eq:t10'}
			\pr\left(\frac{|\be^\top \bM \bv|}{b n} > \delta \right) \le C \delta' / \delta^2 + \pr(\|\bv\|_2^2 > \delta' n^{\frac{2}{1 + t}}).
		\end{equation}
		In light of the above, we have that given any $\eta > 0$, by choosing $\delta'$ as $ \delta_\eta := \delta^2 \eta / (2 C)$,
		\begin{equation*}
			\pr\left(\frac{|\be^\top \bM \bv|}{b n} > \delta \right) \le \frac{\eta}{2} + \pr(\|\bv\|_2^2 > \delta_\epsilon n^{\frac{2}{1 + t}}).
		\end{equation*}
		From above, and by applying Lemma~\ref{lem:althete1} we prove that by choosing $N_\eta > 0$ large enough, for any $n \ge N_\eta$,
		\begin{equation*}
			\pr\left(\frac{|\be^\top \bM \bv|}{b n} > \delta \right) \le \eta,
		\end{equation*}
		which proves the desired result.
	\end{proof}
	
	
	\begin{proof}[Proof of Theorem~\ref{thm:althete}]
		Applying Lemma~\ref{lem:tcombhete} and that $K$ is fixed, it follows from the proof of Theorem~\ref{thm:alternative1} that the only remaining task is to prove that for any fixed $\bs{P}_j, \bs{P}_k$, with probability converging to $1$,
		\begin{equation*}
			\begin{aligned}
				& \frac{\be^\top  \tilde {\bs{V}}_j \tilde {\bs{V}}_j^\top  \be - \be^\top  \tilde {\bs{V}}_k \tilde {\bs{V}}_k^\top  \bP_k \be}{n}  \ge \frac{c_e m}{2 (4 C_e / c_e + m)};\\
				& \frac{\be^\top  \tilde {\bs{V}}_j \tilde {\bs{V}}_j^\top  \be + \be^\top  \tilde {\bs{V}}_k \tilde {\bs{V}}_k^\top  \bP_k \be}{n}  \ge \frac{c_e m}{2 (4 C_e / c_e + m)}.
			\end{aligned}
		\end{equation*}
		
		Without loss of generality we just prove the first inequality. We can write
		\begin{align*}
			\frac{\be^\top  \tilde {\bs{V}}_j \tilde {\bs{V}}_j^\top  \be - \be^\top  \tilde {\bs{V}}_k \tilde {\bs{V}}_k^\top  \bP_k \be}{n} = & \frac{\be^\top  (\tilde {\bs{V}}_j \tilde {\bs{V}}_j^\top  - \mathrm{diag}(\tilde {\bs{V}}_j \tilde {\bs{V}}_j^\top ))\be - \be^\top  (\tilde {\bs{V}}_k \tilde {\bs{V}}_k^\top  \bP_k - \mathrm{diag}(\tilde {\bs{V}}_k \tilde {\bs{V}}_k^\top  \bP_k)) \be}{n}\\
			& + \frac{\be^\top  \mathrm{diag}(\tilde{\bs{V}}_j \tilde{\bs{V}}_j^\top)  \be - \be^\top  \mathrm{diag}(\tilde {\bs{V}}_k \tilde {\bs{V}}_k^\top  \bP_k) \be}{n} \\
			=: &\; \mathrm{I} + \mathrm{II},
		\end{align*}
		
		The term $\mathrm{I}$ can be controlled by a Chebyshev's inequality following the same proof in Lemma~\ref{lem:offdiagbnd}, since the proof of Lemma~\ref{lem:offdiagbnd} can also be generalized to the case where $e_i$'s are heterogeneous but have uniformly bounded variance. For term $\mathrm{II}$, knowing that we are under Assumption~\ref{as:hete}, it follows from the same lines of proof as the term $\mathrm{II}$ in the proof of Theorem~\ref{thm:alternative1} that~\eqref{eq:upp_e_II1} still holds. 
		
		Now the only remaining job is to understand $\E[\mathrm{II} \mid \bP_j, \bP_k]$. Recall the definition of $\bs{A}_k$; by setting $\bs{D}_e$ as a diagonal matrix with $(i,i)$-th entry equal to $\E[e_i^2]$, we have
		\begin{align*}
			\E[\mathrm{II} \mid \bP_j, \bP_k] &= \sum_{i = 1}^n \frac{(\tilde{\bs{V}}_j \tilde{\bs{V}}_j^\top - \tilde {\bs{V}}_k \tilde {\bs{V}}_k^\top  \bP_k)_{i,i}}{n} \E[e_i^2]  \\
			& = \frac{1}{n} \tr[\bs{D}_e] - \frac{1}{n} \tr[\bs{A}_j \bs{A}_j^\top \bs{D}_e] - \frac{1}{n} \tr[\bs{A}_k \bs{A}_k^\top \bs{P}_k \bs{D}_e] \\
			& \ge \frac{1}{n} \tr[\bs{D}_e] - C_e \frac{2p}{n} - \frac{1}{n} C_e \tr[(\bs{A}_0 \bs{A}_0^\top + \bs{B}_k \bs{B}_k^\top) \bs{P}_k].
		\end{align*}
		Using that $\liminf_{n \to \infty} \tr[\bs{D}_e] / n > c_e$, we have that there exists some $N$ such that for all $n \ge N$, $\tr[\bs{D}_e] / n > c_e$; and therefore following the same proof as~\eqref{eq:explower}, we have that with $n$ sufficiently large,
		\begin{align}
			\label{eq:explowerheto}
			\E[\mathrm{II} \mid \bP_j, \bP_k] &\ge c_e - \frac{3 C_e p - C_e \sqrt{2p}K}{n} \ge \frac{c_e m}{4 C_e / c_e + m},
		\end{align}
		i.e., that it is bounded by a constant that does not depend on $n$. Putting together, we prove the desired result.
	\end{proof}
	
	\subsection{Proof of Theorem~\ref{thm:altnonlinear}}

	\begin{proof}[Proof of Theorem~\ref{thm:altnonlinear}]
		We first need to prove that for any fixed $\delta > 0$, 
		\begin{equation*}
			\begin{aligned}
				\prob\left(\exists 1 \le k \le K \;\st\;, \frac{|\bs{h}^\top \tilde{\bs{V}}_k \tilde{\bs{V}}_k^\top \bv|}{b n} \ge \delta\right) & \to 0 ;\\
				\prob\left(\exists 1 \le k \le K \;\st\; \frac{|\bs{h}^\top  \tilde {\bs{V}}_k \tilde {\bs{V}}_k^\top \bP_k \bv|}{b n} \ge \delta\right) & \to 0 ;
			\end{aligned}
		\end{equation*}
		and that the same inequalities still hold but with $\bs{h}$ replaced by $\bs{e}$. Since $\limsup_{n \to \infty} \|\bs{h}\|_2^2 / n \le r$, we have that there exists an $N$ such that for all $n \ge N$, $\|\bs{h}\|_2^2 / n \le 2r$. Then as a direct consequence of Lemma~\ref{lem:tcomb} where we select $\bw = \bs{h}$, $\bM_k = \tilde{\bV}_k \tilde{\bV}_k^\top$ or $\tilde{\bV}_k \tilde{\bV}_k^\top \bs{P}_k$, the above two inequalities still hold. To prove that the above two inequalities still hold with $\bs{h}$ replaced by $\be$, since $K$ is taken as fixed, we may instead apply Lemma~\ref{lem:tcombhete} where we select $\bM = \tilde{\bV}_k \tilde{\bV}_k^\top$ or $\tilde{\bV}_k \tilde{\bV}_k^\top \bs{P}_k$.
		

		Following again the proof of Theorem~\ref{thm:alternative1}, it remains to prove that for any fixed $\bP_j, \bP_k$, with probability converging to $1$, 
		\begin{equation*}
			\begin{aligned}
				& \frac{(\be + \bs{h})^\top  \tilde {\bs{V}}_j \tilde {\bs{V}}_j^\top  (\be + \bs{h}) - (\be + \bs{h})^\top  \tilde {\bs{V}}_k \tilde {\bs{V}}_k^\top  \bP_k (\be + \bs{h})}{n}  \ge \frac{(c_e - r) m}{2 (4 C_e / (c_e - r) + m)};\\
				& \frac{(\be + \bs{h})^\top  \tilde {\bs{V}}_j \tilde {\bs{V}}_j^\top  (\be + \bs{h}) + (\be + \bs{h})^\top  \tilde {\bs{V}}_k \tilde {\bs{V}}_k^\top  \bP_k (\be + \bs{h})}{n}  \ge \frac{(c_e - r) m}{2 (4 C_e / (c_e - r) + m)}.
			\end{aligned}
		\end{equation*}
		In the rest of the proof we focus on the first inequality, and the second can be proven via an analogous argument. Notice that
		\begin{align*}
			& \frac{(\be + \bs{h})^\top  \tilde {\bs{V}}_j \tilde {\bs{V}}_j^\top (\be + \bs{h}) - (\be + \bs{h})^\top  \tilde {\bs{V}}_k \tilde {\bs{V}}_k^\top  \bP_k (\be + \bs{h})}{n} \\
			& = \frac{\be^\top  \tilde {\bs{V}}_j \tilde {\bs{V}}_j^\top \be - \be^\top  \tilde {\bs{V}}_k \tilde {\bs{V}}_k^\top  \bP_k \be}{n} + \frac{\bs{h}^\top  \tilde {\bs{V}}_j \tilde {\bs{V}}_j^\top \bs{h} - \bs{h}^\top  \tilde {\bs{V}}_k \tilde {\bs{V}}_k^\top  \bP_k \bs{h}}{n} \\
			& \quad + \frac{\be^\top  \tilde {\bs{V}}_j \tilde {\bs{V}}_j^\top \bs{h} - \be^\top  \tilde {\bs{V}}_k \tilde {\bs{V}}_k^\top  \bP_k \bs{h}}{n} + \frac{\bs{h}^\top  \tilde {\bs{V}}_j \tilde {\bs{V}}_j^\top \be - \bs{h}^\top  \tilde {\bs{V}}_k \tilde {\bs{V}}_k^\top  \bP_k \be}{n}.
		\end{align*}
		As a direct consequence of Chebyshev's inequality and that $\limsup_{ n \to \infty} \|\bs{h}\|_2^2 / n \le r$, we can easily prove that the last two terms on the right hand side of the above inequality converge in probability to zero; for the second inequality, we have that for any fixed constant $\delta > 0$, with sufficiently large $n$, 
		\[
		\frac{\bs{h}^\top  \tilde {\bs{V}}_j \tilde {\bs{V}}_j^\top \bs{h} - \bs{h}^\top  \tilde {\bs{V}}_k \tilde {\bs{V}}_k^\top  \bP_k \bs{h}}{n} \ge - \|\bs{h}\|_2^2 / n \ge - r - \delta.
		\]
		
		In light of the above and using exactly the same lines of proof as in Theorem~\ref{thm:althete} to deal with the first term, we have that for any constant $\delta > 0$, as $n$ goes to infinity, then with probability converging to $1$,
		\begin{align*}
			& \frac{(\be + \bs{h})^\top  \tilde {\bs{V}}_j \tilde {\bs{V}}_j^\top (\be + \bs{h}) - (\be + \bs{h})^\top  \tilde {\bs{V}}_k \tilde {\bs{V}}_k^\top  \bP_k (\be + \bs{h})}{n}  \\
			& \ge c_e - \frac{3 C_e p + C_e \sqrt{2 p} K}{n} - r - \delta \ge \frac{(c_e - r) m}{4 C_e / (c_e - r) + m} - \delta;
		\end{align*} 
		which proves the desired result since $\delta > 0$ can be chosen arbitrarily small.
	\end{proof}
	
	\subsection{Proof of Corollary~\ref{cor:zdet}}
	
	\begin{proof}
		Without loss of generality we just consider the case with $t \in [0, 1)$, and the case with $t = 1$ follows from an analogous argument. Let 
		\[
		\tilde{z}_n := \min_{1 \le j, k \le K} \min_{z \in \{0, 1\}} \bZ^\top \tilde{\bV}_j \tilde{\bV}_j^\top \bZ + (-1)^z \bZ^\top \tilde{\bV}_k \tilde{\bV}_k^\top \bP_k \bZ.
		\]
		It boils down to proving that for any $1 \le k \le K$ and for any fixed $\delta > 0$,
		\begin{equation}\label{eq:detzv}
			\lim_{n \to \infty}\pr\left(\frac{\bZ^\top \tilde{\bV}_k \tilde{\bV}_k^\top \bv}{ b\tilde{z}_n} > \delta \right) = 0,
		\end{equation}
		and that the same conclusion holds with $\tilde{\bV}_k \tilde{\bV}_k^\top$ replaced by $\tilde{\bV}_k \tilde{\bV}_k^\top \bP_k$. In this proof we just prove the former, and the later follows from an analogous argument. To achieve this goal, write $\bw_k := \sqrt{n} \tilde{\bV}_k \tilde{\bV}_k^\top \bZ / \|\bV_0 \bZ\|_2$. Since $\tilde{\bV}_k$ spans a subspace of $\bV_0$, it is straightforward that $\|\bw_k\|_2 \le \sqrt{n}$; and to prove~\eqref{eq:detzv}, it suffices to prove that for $|b'| = \Omega(n^{-\frac{t}{1 + t}})$, $\lim_{n \to \infty}\pr\left(\frac{\bw_k^\top \bv}{ b' n} > \delta \right) = 0$,
		which is a direct consequence of Lemma~\ref{lem:tcomb} where we select $K = 1$, $\bw = \bw_k$, $\bM_1 = \bs{I}$, $b = b'$ and $\gamma = 1$.
	\end{proof}
	
	\subsection{Proof of Corollary~\ref{cor:convergence}}
	
	\begin{proof}
		From Lemma~\ref{lem:t10moment}, we have that for any fixed $\delta > 0$, there exists a constant $B_\delta > 0$ such that uniformly for all $n$,
		\begin{equation}\label{eq:convbf}
			\sqrt{\frac{1}{n^{\frac{1 - t}{1 + t}}} \sum_{i = 1}^n \left(\E[\varepsilon_i \one(|\varepsilon_i| \le B_\delta i^{\frac{1}{1 + t}})]\right)^2} \le \frac{\delta}{3}.
		\end{equation}
		
		Now writing $f_i := \varepsilon_i \one(|\varepsilon_i| \le B_\delta i^{\frac{1}{1 + t}})$ and $\bs{f}$ as in the proof of Lemma~\ref{lem:tcomb}, we have for any $\bw \in \mathcal{S}^{n - 1}$,
		\[
		|\bw^\top\bv| \le |\bw^\top(\bs{f} - \E[\bs{f}])| + |\bw^\top\E[\bs{f}]| + |\bw^\top(\bv - \bs{f})| \le |\bw^\top(\bs{f} - \E[\bs{f}])| + \|\E[\bs{f}]\|_2 + \|\bv - \bs{f}\|_2.
		\]
		In light of the above and~\eqref{eq:convbf}, we have that to prove the desired statement, we only need to prove that
		\[
		\lim_{n \to \infty}\sup_{\bw \in \mathcal{S}^{n - 1}} \pr\left(\frac{|\bw^\top (\bs{f} - \E[\bs{f}])|}{n^{\frac{1 - t}{2 (1 + t)}}} > \frac{\delta}{3}\right) = 0 \quad\&\quad 
		\lim_{n \to \infty} \pr\left(\frac{\|\bv - \bs{f}\|_2}{n^{\frac{1 - t}{2 (1 + t)}}} > \frac{\delta}{3}\right) = 0.
		\]
		The second inequality follows from exactly the same lines of proof as in the proof of~\eqref{eq:t10hete2}. For the first inequality, we apply Chebyshev's inequality, Lemma~\ref{lem:t10moment} and the basic inequality $\E[(f_i - \E[f_i])^2] \le \E[f_i^2]$ to get that
		\[
		\sup_{\bw \in \mathcal{S}^{n - 1}} \pr\left(\frac{|\bw^\top (\bs{f} - \E[\bs{f}])|}{n^{\frac{1 - t}{2 (1 + t)}}} > \frac{\delta}{3}\right) \le \sup_{\bw \in \mathcal{S}^{n - 1}} \frac{9\E[|\bw^\top (\bs{f} - \E[\bs{f}])|^2]}{\delta^2 n^{\frac{1 - t}{1 + t}}} \le \frac{9 c_n}{\delta^2},
		\]
		which converges to zero as $n \to \infty$. Notice that in the last inequality we use the fact that $c_n$ does not depend on $\bw$.
		Putting together we prove the desired result.
	\end{proof}
	
	\subsection{Proof of Theorem~\ref{thm:kalternative1}}

	Since Lemma~\ref{lem:tcomb} can also work on the case where $K$ diverges with $n$, with the help of the proof of Theorem~\ref{thm:alternative1}, it remains to prove that~\eqref{eq:upp_e} still holds with a diverging $K$. Write 
	\[
	\mathrm{I}_{1k} := \frac{\be^\top  (\tilde {\bs{V}}_k \tilde {\bs{V}}_k^\top  - \mathrm{diag}(\tilde {\bs{V}}_k \tilde {\bs{V}}_k^\top ))\be}{n}, \quad \mathrm{I}_{2k} := \frac{\be^\top  (\tilde {\bs{V}}_k \tilde {\bs{V}}_k^\top  \bP_k - \mathrm{diag}(\tilde {\bs{V}}_k \tilde {\bs{V}}_k^\top  \bP_k)) \be}{n},
	\]
	\[
	\mathrm{II}_{1k} := \frac{\be^\top  \mathrm{diag}(\tilde{\bs{V}}_k \tilde{\bs{V}}_k^\top)  \be - \E[e_1^2] \tr[\tilde{\bs{V}}_k \tilde{\bs{V}}_k^\top]}{n}, \quad \mathrm{II}_{2k} := \frac{\be^\top  \mathrm{diag}(\tilde {\bs{V}}_k \tilde {\bs{V}}_k^\top  \bP_k) \be - \E[e_1^2]\tr[\tilde {\bs{V}}_k \tilde {\bs{V}}_k^\top  \bP_k]}{n},
	\]
	and
	\[
	\mathrm{III}_{j,k} := \E[e_1^2] \cdot \frac{\tr[\tilde{\bs{V}}_j \tilde{\bs{V}}_j^\top]}{n} - \E[e_1^2] \cdot \frac{\tr[\tilde{\bs{V}}_k \tilde{\bs{V}}_k^\top \bP_k]}{n}.
	\]
	
	We first consider $\mathrm{I}_{1k}$. From Lemma~\ref{lem:offdiagbnd}, we have that by Chebyshev's inequality and a union bound, as $n \to \infty$, for any fixed $\delta > 0$,
	\[
	\pr(\forall 1 \le k \le K, |\mathrm{I}_{1k}| \ge \delta) = O(K / n) = o(1).
	\]
	Using an analogous argument we can prove that the above result still holds with $\mathrm{I}_{2k}$.
	
	Second, we consider $\mathrm{II}_{1k}$. In the following, we control this by applying Lemma~\ref{lem:tcomb} where we select $e_i^2 - \E[e_i^2]$ as $\varepsilon_i$, $\mathrm{diag}(\tilde{\bs{V}}_k \tilde{\bs{V}}_k^\top)$ as $\bM_k$, ${\bf 1}$ as $\bw$, $\kappa / 2$ as $t$ and set $\gamma = 1$. We now discuss by cases based on $\kappa$.
	
	\begin{enumerate}
		\item When $0 \le \kappa < 2$, we can set the $b$ in Lemma~\ref{lem:tcomb} as $b := \sqrt{K} n^{\frac{-\kappa / 2}{1 + \kappa / 2}}$, then it follows from Lemma~\ref{lem:tcomb} that
		\[
		\sup_{1 \le k \le K} \frac{1}{n} \sum_{i = 1}^n a_{k,i} (e_i^2 - \E[e_i^2]) = o_\pr(b \cdot \|{\bf 1}\|_2^2 / n) = o_\pr(\sqrt{K} n^{- \frac{\kappa}{2 + \kappa}}) =  o_\pr(1).
		\]
		
		
		\item When $\kappa \ge 2$, we can choose the $b$ in Lemma~\ref{lem:tcomb} as $b = 1 = \omega(\sqrt{K/n})$ (note that $K = o(n)$ here). Then it follows from Lemma~\ref{lem:tcomb} that
		\begin{align*}
			\sup_{1 \le k \le K} \frac{1}{n} \sum_{i = 1}^n a_{k,i} (e_i^2 - \E[e_i^2]) = o_\pr(b \| {\bf 1}\|_2^2 / n) = o_\pr(1).
		\end{align*} 
		
	\end{enumerate}
	In light of both cases we have $\sup_{1 \le k \le K} \mathrm{II}_{1k} = o_\pr(1)$. Analogously we can derive the same bound for $\mathrm{II}_{2k}$. 
	
	We thirdly consider $\mathrm{III}_{j,k}$. Using exactly the same analysis as the $\E[\mathrm{II} \mid \bP_j, \bP_k]$ in the proof of Theorem~\ref{thm:alternative1}, we have for any $1 \le j,k \le K$,
	\[
	\mathrm{III}_{j,k} \ge \frac{1}{n} \left(n - 3p - \min\{\sqrt{2p} K, p\}\right) \ge \frac{m}{4 + m}.
	\]
	In light of our control of $\mathrm{I}_{1k}, \mathrm{I}_{2k}, \mathrm{II}_{1k}, \mathrm{II}_{2k}$ and $\mathrm{III}_{j,k}$, we prove the desired result.
}

\section{Proof of minimax rate optimality results}\label{sec:pfopt}

\subsection{Proof of Theorem~\ref{thm:lower}}

Without loss of generality we consider the scenario where $\beta = \beta^Z = 0$. Let $H_1(\tau)$ be the class of alternatives such that $|b| \ge \tau$, with $\tau$ to be specified later. Then using Neyman-Pearson lemma, we have that for any $(\bZ, \bY)$ in $H_0$ and any $(\bZ', \bY')$ in $H_1(\tau)$, 
\[
\mathcal{R}_{t, \bX}(\tau) \geq 1 - \mathbb{TV}(\prob_{\bY, \bZ}, \prob_{\bY', \bZ'}).
\]
Hence, the problem becomes constructing a $(\bZ, \bY)$ and $(\bZ', \bY')$ belonging to $H_0$ and $H_1(\tau)$ such that their total variation distance is smaller than $\eta$.

We can do the following construction. First, we construct $Z_i$ as i.i.d. binary random variables such that $\prob(Z_i = n/\gamma) = \gamma/n$ and $\prob(Z_i = - (1 - \gamma/n)^{-1}) = 1 - \gamma/n$, where $\gamma = -\log(1-\eta)/2$, and without loss of generality, $n$ is sufficiently larger such that $\gamma/n < 1$. Moreover, we construct $Z_i'$ such that for each $i$, $Z_i = Z_i'$ almost surely. 

We then construct $\varepsilon_i, \varepsilon_i'$ as i.i.d. Rademacher random variables that are independent from $Z_i, Z_i'$; and construct $\tilde Z_i$ as i.i.d replicates of $Z_i$ which are independent from other randomness in the problem. 
Finally let $Y_i = b \tilde{Z}_i + \varepsilon_i$ and $Y_i' = b Z_i' + \varepsilon_i$ where $b = c_\eta n^{-t/(1+t)}$ for some constant $c_\eta > 0$ depending only on $\eta$ such that $E[|Y_i|^{1+t}] = E[|Y_i'|^{1+t}] = 2$. Then it is straightforward that the distribution of $Y_i$ is in $\cD_{1 + t}$, so that $(\bY, \bZ)$ and $(\bY', \bZ')$ are feasible choices in $H_0$ and $H_1(\tau)$ respectively with $\tau := c_\eta n^{- t / (1 + t)}$.

Using the above construction, we control their total variation distance as
\begin{align*}
	\mathbb{TV}(\prob_{\bY, \bZ}, \prob_{\bY', \bZ'}) 
	&= \sup_B\{\prob \left(\left(\bY, \bZ\right) \in B\right) - \prob\left(\left(\bY', \bZ'\right) \in B\right) \} \\
	&\le \sup_B\{\prob \left(\left(\bY, \bZ\right) \in B\right) - \prob\left(\left(\bY', \bZ'\right) \in B, \left(\bY, \bZ\right) \in B\right) \} \\
	&\le \sup_B{\prob \left(\left(\bY, \bZ\right) \in B,\left(\bY', \bZ'\right) \not \in B \right)} \\
	&\le \prob\left( \bZ \neq \tilde \bZ \right) \le 1 - (1-\gamma/n)^{2n} \le 1 - e^{-2\gamma} = \eta.
\end{align*}

\subsection{Preliminary lemmas for Theorem~\ref{thm:supupp}}

We first invoke the following lemma, which is a uniform convergence extension of Lemma~\ref{lem:t10moment}. 

{\rev \begin{lemma}\label{lem:supt10moment}
		Let $t \in (0, 1), \gamma > 0$ be given. For any fixed $B > 0$, 
		\[
		\sup_{\bw \in \R^n_\circ} \sup_{\prob_\varepsilon \in \cD_{1 + t}}\sum_{i = 1}^n \frac{\E[w_i^2 \varepsilon_i^2 \one(|\varepsilon_i| \le B i^{\frac{1}{1 + t_1}})]}{\|\bw\|_2^2 n^{\frac{1 - t}{1 + t} + \gamma}} = o(1).
		\]
		Moreover,
		\[
		\sum_{i=1}^n \left(\E[\varepsilon_i \one(|\varepsilon_i| \le B i^{\frac{1}{1 + t_1}})] \right)^2 / n^{\frac{1 - t}{1 + t} + \gamma} = o(1),
		\]
		where $\R^n_\circ$ denotes the $n$-dimensional Euclidean space excluding the original point and $t_1$ denote a constant in $(-1, t)$ such that
		\[
		\frac{1 - t_1}{1 + t_1} - \frac{1 - t}{1 + t} = \gamma.
		\]
	\end{lemma}
	
	\begin{proof}
		As a direct consequence of Jensen's inequality it is straightforward that 
		\begin{equation}\label{eq:jensen}
			\sup_{\prob_\varepsilon \in \cD_{1 + t}} \E[|\varepsilon_1|^{1 + t_1}] \le \sup_{\prob_\varepsilon \in \cD_{1 + t}} (\E[|\varepsilon_1|^{1 +t}])^{\frac{1 + t_1}{1 + t}} \le 2^{\frac{1 + t_1}{1 + t}}.
		\end{equation}
		
		Let $f_i := \varepsilon_i \one(|\varepsilon_i| \le B i^{\frac{1}{1 + t_1}})$. To prove the first statement, following the proof of Lemma~\ref{lem:t10moment}, we only need to prove that
		\begin{equation}\label{eq:t10mom1}
			\sup_{\bw \in \R^n_\circ} \sup_{\prob_\varepsilon \in \cD_{1 + t}}\sum_{i = 1}^n \frac{\E[w_i^2 \varepsilon_i^2 \one(|\varepsilon_i| \le B a_n^{\frac{1}{1 + t_1}})]}{\|\bw\|_2^2 n^{\frac{1 - t}{1 + t} + \gamma}} = o(1)
		\end{equation}
		and that  
		\begin{equation}\label{eq:t10mom2}
			\sup_{\bw \in \R^n_\circ} \sup_{\prob_\varepsilon \in \cD_{1 + t}}\sum_{i = a_n + 1}^n \frac{\E[w_i^2 \varepsilon_i^2 \one( B a_n^{\frac{1}{1 + t_1}} < |\varepsilon_i| \le B i^{\frac{1}{1 + t_1}})]}{\|\bw\|_2^2 n^{\frac{1 - t}{1 + t} + \gamma}} = o(1).
		\end{equation}
		For~\eqref{eq:t10mom1}, following exactly the same lines of proof as in Lemma~\ref{lem:t10moment} but with $t$ replaced by $t_1$, we can easily have that its left hand side is of order $o(\sup_{\prob_\varepsilon \in \cD_{1 + t}} \E[|\varepsilon_1|^{1 + t_1}])$, which is of order $o(1)$ knowing~\eqref{eq:jensen}. For~\eqref{eq:t10mom2}, following again the proof in Lemma~\ref{lem:t10moment} but with $t$ replaced by $t_1$, we only need to prove that
		\[
		\sup_{\prob_\varepsilon \in \cD_{1 + t}} \E[|\varepsilon_1|^{1 + t_1} \one(|\varepsilon_1| > B a_n^{\frac{1}{1 + t_1}})] = o(1).
		\]
		This follows from the fact that
		\begin{align*}
			& \E[|\varepsilon_1|^{1 + t_1} \one(|\varepsilon_1| > B a_n^{\frac{1}{1 + t_1}})] \le \E[|\varepsilon_1|^{1 + t} |\varepsilon_1|^{t_1 - t} \one(|\varepsilon_1| > B a_n^{\frac{1}{1 + t_1}})] \\
			& \le \frac{\E[|\varepsilon_1|^{1 + t} \one(|\varepsilon_1| > B a_n^{\frac{1}{1 + t_1}})]}{B^{t - t_1} a_n^{\frac{t - t_1}{1 + t_1}}} \le \E[|\varepsilon_1|^{1 + t}] / (B^{t - t_1} a_n^{\frac{t - t_1}{1 + t_1}}).
		\end{align*}

		Now for the second claim, recall again the proof in Lemma~\ref{lem:t10moment}, we have
		\[
		\sum_{i = 1}^n (\E[f_i])^2 \le 2^{\frac{2}{1 + t}} n^{\frac{1 - t}{1 + t}} \left(\sum_{i = 1}^\infty \prob(|\varepsilon_i|^{1 + t_1} > B^{1 + t_1} i)\right)^{\frac{2t}{1 + t}},
		\]
		and the desired result follows from that 
		\[
		\sum_{i = 1}^\infty \prob(|\varepsilon_i|^{1 + t_1} > B^{1 + t_1} i) \le \int_0^\infty\pr\left(\frac{|\varepsilon_i|^{1 + t_1}}{B^{1 + t_1}} > x\right) dx = \frac{\E[|\varepsilon_i|^{1 + t_1}]}{B^{1 + t_1}}
		\]
		and~\eqref{eq:jensen}.
	\end{proof}
	

	
		
		
	
	Armed with the above lemma, we now introduce Lemma~\ref{lem:natcomb}, which is an extension of Lemma~\ref{lem:tcomb}. 
	
	\begin{lemma}\label{lem:natcomb}
		Consider the $\bM \in \R^{n \times n}$ with $\|\bM\|_\op \le 1$ and let $t \in (0, 1)$ be given. Then if $b \ge \Delta n^{-\frac{t}{1 + t} + \gamma}$ for some constants $\gamma, \Delta > 0$, we have that for any fixed $\delta, \gamma' > 0$,
		\[
		\lim_{n \to \infty} \sup_{\bw \in \R^n} \sup_{\prob_\varepsilon \in \cD_{1 + t}}  \prob\left(\frac{|\bw^\top \bM \bv|}{b\max\{\|\bw\|_2^2, \gamma' n\}} > \delta\right) = 0.
		\]
	\end{lemma}
	
	\begin{proof}
		When $t = 1$, the result follows from an argument analogous to the proof of Lemma~\ref{lem:t1}. Therefore we just need to prove that the result holds for $t \in (0, 1)$.
		
		Let $f_i$ be defined as in Lemma~\ref{lem:supt10moment} for some constant $B > 0$ and write $\bs{f} := (f_1, \ldots, f_n)^\top$. Then we only need to prove that
		\[
		\lim_{n \to \infty} \sup_{\bw \in \R^n} \sup_{\prob_\varepsilon \in \cD_{1 + t}}  \prob\left(\frac{|\bw^\top \bM (\bs{f} - \E[\bs{f}])|}{b\max\{\|\bw\|_2^2, \gamma' n\}} > \delta\right) = 0,
		\]
		\[
		\lim_{n \to \infty} \sup_{\bw \in \R^n} \sup_{\prob_\varepsilon \in \cD_{1 + t}} \frac{|\bw^\top \bM \E[\bs{f}]|}{b\max\{\|\bw\|_2^2, \gamma' n\}}  = 0,
		\]
		and that
		\[
		\lim_{n \to \infty} \sup_{\bw \in \R^n} \sup_{\prob_\varepsilon \in \cD_{1 + t}}  \prob\left(\frac{|\bw^\top \bM (\bv - \bs{f})|}{b\max\{\|\bw\|_2^2, \gamma' n\}} > \delta\right) = 0.
		\]

		The first inequality follows from exactly the same lines of proof as in Lemma~\ref{lem:tcomb}, except that we replace Lemma~\ref{lem:t10moment} with Lemma~\ref{lem:supt10moment}; the second inequality follows from a Cauchy-Schwartz inequality and Lemma~\ref{lem:supt10moment}. Now we are still left with the task of dealing with the third inequality. 
		
		Following the proof of Lemma~\ref{lem:tcomb}, to prove that the third inequality holds, it remains to prove that given any constant $\eta > 0$, there exists constants $N_\eta, C_\eta$ such that the following result hold:
		\begin{equation}\label{eq:supbc}
			\sup_{\prob_\varepsilon \in \cD_{1 + t}} \pr\left(\mathcal{E}_\eta^c\right) \le \eta, \quad\text{where}\quad \mathcal{E}_\eta := \{\forall i > N_\eta, f_i = \varepsilon_i, \forall \ell \le N_\eta, |\varepsilon_\ell| \le C_\eta\}.
		\end{equation}
		To achieve this goal, observe that for any integer $N > 0$,
		\begin{align*}
			& \sum_{i = N + 1}^\infty \prob(f_i \neq \varepsilon_i) = \sum_{i = N + 1}^\infty \prob ( |\varepsilon_i|^{1+t_1} > B^{1+t_1}i) \le \int_N^\infty \prob ( |\varepsilon_1|^{1+t_1} > B^{1+t_1} x) dx \\
			& = \int_0^\infty \prob \left\{ \left(\frac{|\varepsilon_1|^{1+t_1}}{B^{1+t_1}} - N\right) \one(|\varepsilon_1|^{1+t_1} > B^{1+t_1}N) >  x\right\} dx \\
			& = \E\left[\left(\frac{|\varepsilon_1|^{1+t_1}}{B^{1+t_1}} - N\right) \one(|\varepsilon_1|^{1+t_1} > B^{1+t_1}N)\right] \le \E\left[\frac{|\varepsilon_1|^{1+t_1}}{B^{1+t_1}}  \one(|\varepsilon_1|^{1+t_1} > B^{1+t_1}N)\right],
		\end{align*}
		whence
		\begin{align*}
			& \E\left[\frac{|\varepsilon_1|^{1+t_1}}{B^{1+t_1}} \one(|\varepsilon_1|^{1+t_1} > B^{1+t_1}N)\right] = \E\left[\frac{|\varepsilon_1|^{1+t} |\varepsilon_1|^{t_1 - t}}{B^{1+t_1}} \one(|\varepsilon_1|^{1+t_1} > B^{1+t_1}N)\right]\\
			& \le \E\left[\frac{|\varepsilon_1|^{1+t}}{B^{1+t_1}} \one(|\varepsilon_1|^{1+t_1} > B^{1+t_1}N)\right] \cdot B^{t_1 - t} N^{\frac{t_1 - t}{1 + t_1}} \le 2 \frac{ N^{- \frac{t - t_1}{1+t_1}}}{B^{1+t}}.
		\end{align*}
		From above, by choosing $N_\eta := \left( \frac{4}{\eta B^{1+t}} \right)^{\frac{1+t_1}{t - t_1}}$ , we have for any $\prob_\varepsilon \in \cD_{1 + t}$,
		\[
		\prob(\exists i > N_\eta \;\st\; f_i \neq \varepsilon_i) \leq \sum_{i = N_\eta + 1}^\infty \prob(f_i \neq \xi_i) \leq \frac{\eta}{2}.
		\]
		Further letting $C_\eta := (4 N_\eta/\eta)^{\frac{1}{1+t}}$, we have for any $\prob_\varepsilon \in \cD_{1 + t}$,
		\[
		\prob( \exists \ell \le N_\eta , \;\st\; |\varepsilon_\ell| > C_\eta) \le N_\eta \prob(|\varepsilon_1| > C_\eta) \le \frac{2N_\eta}{C_\eta^{1+t}} = \frac{\eta}{2}.
		\]
		In light of the above two inequalities we prove~\eqref{eq:supbc}, thereby proving the desired result.
\end{proof}}



\begin{lemma}\label{lem:navan06}
	Let $a_{n,i}$ ($i, n = 1, 2, \ldots$) be a deterministic array with $\sum_{i=1}^n |a_{n,i}|^2 \le \frac{4}{n}$.
	Let $V_i$ ($i = 1, \ldots, \infty$) be a sequence of independent random variables obeying the law $\prob_{V_i}$.
	Then for any $\gamma > 0$,
	\[
	\lim_{n \to \infty} \sup_{\prob_{V_1}, \ldots, \prob_{V_n} \in D_{1+\gamma}}\E\left[\left|\sum_{i=1}^n a_{n,i} (V_i - \E[V_i])\right|\right] = 0.
	\]
\end{lemma}

\begin{proof}
	
	Let $a = n^{\frac{1}{2(\gamma + 1)}}$; define 
	\[
	V_i' = V_i \one(|V_i| > a), \quad 
	V_i'' = V_i \one(|V_i| \le a).
	\]
	
	We first have 
	\begin{align*}
		\sum_{i=1}^n |a_{n,i}| \E[|V_i'|]&= \sum_{i=1}^n |a_{n,i}| \E[|V_i| \one(|V_i| > a)] = \sum_{i=1}^n |a_{n,i}| \E[|V_i|^{1+\gamma} |V_i|^{-\gamma} \one(|V_i| > a)] \\ 
		&\le a^{-\gamma}\sum_{i=1}^n |a_{n,i}|\E[|V_i|^{1+\gamma} \one(|V_i| > a)]  \le 2a^{-\gamma} \sum_{i=1}^n |a_{n,i}| \\&\overset{(i)}{\le}2a^{-\gamma} n^{1/2}\left( \sum_{i=1}^n |a_{n,i}|^2 \right)^{1/2} \le 4a^{-\gamma} .    
	\end{align*}
	where $(i)$ uses Cauchy-Schwartz inequality. From above, we have
	\begin{align*}
		\E\left[\left|\sum_{i=1}^n a_{n,i} (V_i - \E[V_i])\right|\right] &= 
		\E\left[\left|\sum_{i=1}^n a_{n,i} (V_i' - \E[V_i'] + V_i'' - \E[V_i''] )\right|\right] \\
		&\le \E\left[\left|\sum_{i=1}^n a_{n,i} (V_i'' - \E[V_i''])\right|\right] + 2 \sum_{i=1}^n |a_{n,i}| \E[|V_i'|] \\
		&\le \E\left[\left|\sum_{i=1}^n a_{n,i} (V_i'' - \E[V_i''])\right|\right] + 8a^{-\gamma}
	\end{align*}
	
	To deal with the first summand on the right hand side of the above inequality, we apply H\"{o}lder's inequality to get that
	\begin{align*}
		& \E\left[\left|\sum_{i=1}^n a_{n,i} (V_i'' - \E[V_i''])\right|\right]  \le \left[ \E\left[\left|\sum_{i=1}^n a_{n,i} (V_i'' - \E[V_i''])\right|^2\right] \right]^{1/2} \\
		&= \left[ \E\left[\sum_{i=1}^n a_{n,i}^2 (V_i'' - \E[V_i''])^2\right] \right]^{1/2} \le \left[ \E\left[\sum_{i=1}^n a_{n,i}^2 4a^2 \right] \right]^{1/2} = 2a \left[\sum_{i=1}^n a_{n,i}^2 \right]^{1/2} \le \frac{4a}{\sqrt{n}}
	\end{align*}
	
	Putting together, we have
	\[
	\sup_{\prob_{V_1}, \ldots, \prob_{V_n} \in D_{1+\gamma}} \E\left[\left|\sum_{i=1}^n a_{n,i} (V_i - \E[V_i])\right|\right] \le 8a^{-\gamma} + \frac{4a}{\sqrt{n}} \le 12 n^{-\frac{\gamma}{2(\gamma + 1)}},
	\]
	which gives us the desired result.
\end{proof}

\subsection{Theoretical analysis of \eqref{eq:upp_sup_high}}

\begin{proof}
	
	Following the proof of Theorem~\ref{thm:alternative1}, we only need to show that for any fixed $j, k$, for all $\delta > 0$,
	\begin{equation}\label{eq:naupp_epsilon}
		\begin{aligned}
			\sup_{\prob_e \in \cD_{2+\nu}} \sup_{\prob_\varepsilon \in \cD_{1+t}} &  \prob\left(\frac{|\be^\top \tilde{\bs{V}}_j \tilde{\bs{V}}_j^\top \bv|}{b n} > \delta\right) \to 0 ;\\
			\sup_{\prob_e \in \cD_{2+\nu}} \sup_{\prob_\varepsilon \in \cD_{1+t}} &  \prob\left(\frac{|\be^\top  \tilde {\bs{V}}_k \tilde {\bs{V}}_k^\top \bP_k \bv|}{b n} > \delta\right) \to 0 ;
		\end{aligned}
	\end{equation}
	and that, 
	\begin{equation}\label{eq:naupp_e}
		\begin{aligned}
			& \sup_{\prob_e \in \cD_{2+\nu}} \sup_{\prob_\varepsilon \in \cD_{1+t}}  \prob \left(\frac{\be^\top  \tilde {\bs{V}}_j \tilde {\bs{V}}_j^\top  \be - \be^\top  \tilde {\bs{V}}_k \tilde {\bs{V}}_k^\top  \bP_k \be}{n} < \frac{m}{2 (4 + m)}\right) \to 0 ;\\ 
			& \sup_{\prob_e \in \cD_{2+\nu}} \sup_{\prob_\varepsilon \in \cD_{1+t}}  \prob \left( \frac{\be^\top  \tilde {\bs{V}}_j \tilde {\bs{V}}_j^\top  \be + \be^\top  \tilde {\bs{V}}_k \tilde {\bs{V}}_k^\top  \bP_k \be}{n}  < \frac{m}{2 (4 + m)}\right) \to 0.
		\end{aligned}
	\end{equation}
	
	To prove the first claim of~\eqref{eq:naupp_epsilon}, since $\prob_e \in \cD_{2+\nu}$, using Lemma~\ref{lem:navan06} yields
	\[
	\sup_{\prob_e \in \cD_{2+\nu}}\bE\left[ \left|\frac{1}{n} \|\be\|_2^2 - \E[\be_1^2] \right| \right] \to 0,
	\]
	whence by Markov's inequality,
	\[
	\sup_{\prob_e \in \cD_{2+\nu}} \prob\left( \|\be\|_2^2 > 2\E[\be_1^2] n \right)  \to 0.
	\]
	
	For $\prob_e \in \cD_{2+\nu}$, using  H\"{o}lder's inequality, we have
	\[
	\E[\be_1^2] \le (\E[|\be_1|^{2 + \nu}])^{2/(2+\nu)} \le 2^{2/(2+\nu)}.
	\]
	From the above two inequalities, we have the random event $\mathcal{E} := \{\|\be\|_2^2 \le 2^{(3+\nu)/(2+\nu)} n\}$ satisfies that $\sup_{\prob_e \in \cD_{2+\nu}} \prob(\mathcal{E}^c) \to 0$. Therefore, by choosing $\gamma' = 2^{(3 + \nu) / (2 + \nu)}$ we can control the first inequality of~\eqref{eq:naupp_epsilon} via that
	\begin{align*}
		\sup_{\prob_e \in \cD_{2 + \nu}} \sup_{\prob_\varepsilon \in \cD_{1 + t}} \;& \prob\left(\frac{|\be^\top \tilde{\bs{V}}_j \tilde{\bs{V}}_j^\top \bv|}{b n} > \delta\right) \le \sup_{\prob_e \in \cD_{2 + \nu}}\sup_{\prob_\varepsilon \in \cD_{1 + t}} \prob\left(\frac{|\be^\top \tilde{\bs{V}}_j \tilde{\bs{V}}_j^\top \bv|}{b n} > \delta \mid \mathcal{E}\right) + \sup_{\prob_e \in \cD_{2 + \nu}}\prob(\mathcal{E}^c) \\
		& \overset{(i)}{=} \sup_{\prob_e \in \cD_{2 + \nu}} \sup_{\prob_\varepsilon \in \cD_{1 + t}} \prob\left(\frac{|\be^\top \tilde{\bs{V}}_j \tilde{\bs{V}}_j^\top \bv|}{b \max\{\|\be\|_2^2, \gamma' n\}} > \frac{\delta}{\gamma'} \mid \mathcal{E}\right) + \sup_{\prob_e \in \cD_{2 + \nu}} \prob(\mathcal{E}^c) \\
		& \le \sup_{\bw \in \R^n} \sup_{\prob_\varepsilon \in \cD_{1 + t}} \prob\left(\frac{|\bw^\top \tilde{\bs{V}}_j \tilde{\bs{V}}_j^\top \bv|}{b \max\{\|\bw\|_2^2, \gamma' n\}} > \frac{\delta}{\gamma'}\right) + \sup_{\prob_e \in \cD_{2 + \nu}} 
		\prob(\mathcal{E}^c),
	\end{align*}
	where for the equality $(i)$ we apply that we are under $\mathcal{E}$. Then as a direct consequence of Lemma~\ref{lem:natcomb}, we prove the first claim of~\eqref{eq:naupp_epsilon}. The second claim of~\eqref{eq:naupp_epsilon} follows from an analogous argument.
	
	In the rest of the proof we focus on proving the first statement of~\eqref{eq:naupp_e}, and the second statement can be prove via a similar argument. To prove this statement, we apply again the decomposition
	\begin{align*}
		\frac{\be^\top  \tilde {\bs{V}}_j \tilde {\bs{V}}_j^\top  \be - \be^\top  \tilde {\bs{V}}_k \tilde {\bs{V}}_k^\top  \bP_k \be}{n} = & \frac{\be^\top  (\tilde {\bs{V}}_j \tilde {\bs{V}}_j^\top  - \mathrm{diag}(\tilde {\bs{V}}_j \tilde {\bs{V}}_j^\top ))\be - \be^\top  (\tilde {\bs{V}}_k \tilde {\bs{V}}_k^\top  \bP_k - \mathrm{diag}(\tilde {\bs{V}}_k \tilde {\bs{V}}_k^\top  \bP_k)) \be}{n}\\
		& + \frac{\be^\top  \mathrm{diag}(\tilde{\bs{V}}_j \tilde{\bs{V}}_j^\top)  \be - \be^\top  \mathrm{diag}(\tilde {\bs{V}}_k \tilde {\bs{V}}_k^\top  \bP_k) \be}{n} \\
		=: &\; \mathrm{I} + \mathrm{II},
	\end{align*}
	where recall that for any matrix $\bs{A} \in \R^{n \times n}$, $\mathrm{diag}(\bs{A})$ corresponds to the diagonal matrix such that all the diagonal elements are equal to the diagonal elements of $\bs{A}$.
	
	For $\mathrm{I}$, using the same lines of proof as the term $\mathrm{I}$ in Section~\ref{sec:upp_asymp_high}, we have that for any constant $\delta > 0$,
	\begin{equation}\label{eq:naupp_e_I}
		\sup_{\prob_e \in \cD_{2+\nu}} \prob(|\mathrm{I}| < \delta) \le \frac{2^{(6 + \nu)/(2+\nu)}}{n \delta^2}.
	\end{equation}
	
	For $\mathrm{II}$, we apply the same lines of proof as the control of term $\mathrm{II}$ in Section~\ref{sec:upp_asymp_high}, except that we replace Lemma~\ref{lem:van06} with Lemma~\ref{lem:navan06}. Putting together, we obtain the desired result.
\end{proof}

\subsection{Theoretical analysis of \eqref{eq:upp_sup_low}}

\begin{lemma}\label{lem:napartconv}
	Let $\bP \in \R^{n \times n}$ be a completely random permutation matrix. We have that for any fixed $\delta > 0$,
	\[
	\lim_{n \to \infty} \sup_{\prob_e \in \cD_{1+\nu}} \prob(|\be^\top \bP \be| / n > \delta) = 0.
	\].
\end{lemma}
\begin{proof}
	Let $e_{1, 1}, e_{1, 2}, \ldots, e_{1, n}, \ldots$ and $e_{2, 1}, e_{2, 2}, \ldots, e_{2, n}, \ldots$ be two sequences of i.i.d. random variables from a distribution $\prob_e$. Then apparently if $\prob_e \in \cD_{1 + t}$, $1 \le \E[|e_{1, i} e_{2, i}|^{1 + t}] \le 4$.
	Then using Lemma~\ref{lem:navan06}, we have from Markov's inequality that for any $\delta > 0$,
	\[
	\lim_{n \to \infty} \sup_{\prob_e \cD_{1 + \nu}} \prob \left(\left|\frac{1}{n} \sum_{i=1}^n e_{1, i} e_{2, i}\right| > \delta\right) = 0.
	\]
	
	The desired result then follows from the same lines of proof as in Lemma~\ref{lem:partconv}.
	
\end{proof}

\begin{lemma}\label{lem:nasumsq}
	We have
	\[
	\lim_{n \to \infty} \sup_{\prob_e \in \tilde{\cD}}\prob\left(\|\be\|_2^2 < \frac{n}{16}\right) = 0.
	\]
\end{lemma}

\begin{proof}
	
	Let $\tilde{e}_i := \frac{1}{2} \one\left(|e_i| \geq \frac{1}{2}\right)$, then $E[\tilde{e}_i^2] \ge \frac{1}{8}$.
	By Hoeffding's inequality,
	\[
	\prob\left(\left| \sum_{i=1}^n (\tilde{e}_i^2 - E[\tilde{e}_i^2]) \right| \ge \frac{n}{16} \right) \le \exp\left(-\frac{n}{128}\right).
	\]
	
	In light of the above inequality and that almost surely, $|e_i| \ge |\tilde{e}_i|$, we obtain the desired result.
\end{proof}

\begin{proof}[Proof of~\eqref{eq:upp_sup_low}]
	Following analogous argument as in the proof of Theorem~\ref{thm:alternative2}, we tackle this problem via proving that for any $j, k \in \{1, \ldots, K\}$ and for all $\delta > 0$,
	\begin{equation}\label{eq:naupp_epsilon_2}
		\begin{aligned}
			\sup_{\prob_e \in \cD_{1+\nu} \cap \tilde{\cD}} \sup_{\prob_\varepsilon \in \cD_{1+t}}  & \prob\left(\frac{|\be^\top \tilde{\bs{V}}_j \tilde{\bs{V}}_j^\top \bv|}{b \| \be \|^2_2} \ge \delta\right) \to 0;\\
			\sup_{\prob_e \in \cD_{1+\nu} \cap \tilde{\cD}} \sup_{\prob_\varepsilon \in \cD_{1+t}}  & \prob\left(\frac{|\be^\top  \tilde {\bs{V}}_k \tilde {\bs{V}}_k^\top \bP_k \bv|}{b \| \be \|^2_2} \ge \delta\right) \to 0;
		\end{aligned}
	\end{equation}
	and that
	\begin{equation}\label{eq:naupp_e_2}
		\begin{aligned}
			& \sup_{\prob_e \in \cD_{1+\nu} \cap \tilde{\cD}} \prob
			\left( \frac{\be^\top  \tilde {\bs{V}}_j \tilde {\bs{V}}_j^\top  \be - \be^\top  \tilde {\bs{V}}_k \tilde {\bs{V}}_k^\top  \bP_k \be}{\|\be \|^2}  < \frac{1}{5}\right) \to 0;\\
			& \sup_{\prob_e \in \cD_{1+\nu} \cap \tilde{\cD}} \prob \left( \frac{\be^\top  \tilde {\bs{V}}_j \tilde {\bs{V}}_j^\top  \be + \be^\top  \tilde {\bs{V}}_k \tilde {\bs{V}}_k^\top  \bP_k \be}{\|\be \|^2}  < \frac{1}{5} \right) \to 0.
		\end{aligned}
	\end{equation}
	
	For the first claim of~\eqref{eq:naupp_epsilon_2}, writing $\mathcal{E} := \{\|\be\|_2^2 \ge n / 16\}$, we have with $\gamma' := 1 / 16$,
	\begin{align*}
		\sup_{\prob_e \in \cD_{1+\nu} \cap \tilde{\cD}} & \sup_{\prob_\varepsilon \in \cD_{1+t}}  \prob\left(\frac{|\be^\top \tilde{\bs{V}}_j \tilde{\bs{V}}_j^\top \bv|}{b \| \be \|^2_2} \ge \delta\right) \le \sup_{\prob_e \in \cD_{1+\nu} \cap \tilde{\cD}} \sup_{\prob_\varepsilon \in \cD_{1+t}}  \prob\left(\frac{|\be^\top \tilde{\bs{V}}_j \tilde{\bs{V}}_j^\top \bv|}{b \| \be \|^2_2} \ge \delta \mid \mathcal{E}\right) + \sup_{\prob_e \in \tilde{\cD}} \prob(\mathcal{E}^c) \\
		& = \sup_{\prob_e \in \cD_{1+\nu} \cap \tilde{\cD}} \sup_{\prob_\varepsilon \in \cD_{1+t}}  \prob\left(\frac{|\be^\top \tilde{\bs{V}}_j \tilde{\bs{V}}_j^\top \bv|}{b \max\{\| \be \|^2_2, \gamma' n\}} \ge \delta \mid \mathcal{E}\right) + \sup_{\prob_e \in \tilde{\cD}} \prob(\mathcal{E}^c) \\
		& \le \sup_{\bw \in \R^n} \sup_{\prob_\varepsilon \in \cD_{1+t}}  \prob\left(\frac{|\bw^\top \tilde{\bs{V}}_j \tilde{\bs{V}}_j^\top \bv|}{b \max\{\| \bw \|^2_2, \gamma' n\}} \ge \delta \mid \mathcal{E}\right) + \sup_{\prob_e \in \tilde{\cD}} \prob(\mathcal{E}^c),
	\end{align*}
	which converges to zero knowing that we have Lemmas~\ref{lem:natcomb} and~\ref{lem:nasumsq}. The second claim of~\eqref{eq:naupp_epsilon_2} can be proven via a similar argument.
	
	We now focus on~\eqref{eq:naupp_e_2}, recall the definitions of $\mathcal{E}_1$ -- $\mathcal{E}_5$, it remains to prove that
	\[
	\lim_{n \to \infty} \sup_{\prob_e \in \cD_{1+\nu} \cap \tilde{\cD}} \prob(\mathcal{E}_1^c \cup \cdots \cup \mathcal{E}_5^c) = 0.
	\]
	
	We can control $\mathcal{E}_1$ -- $\mathcal{E}_5$ following the same lines of proof as in the proof of those events in Section~\ref{sec:upp_asymp_low}, except that for $\mathcal{E}_2$ and $\mathcal{E}_3$ we replace Lemma~\ref{lem:tcomb} by Lemma~\ref{lem:natcomb}; for $\mathcal{E}_4$, we replace Lemma~\ref{lem:partconv} and~\eqref{eq:e1mom} by Lemmas~\ref{lem:napartconv} and~\ref{lem:nasumsq} respectively; and for $\mathcal{E}_5$, we additionally control the uniform convergence of $|\be^\top \be'| / n$ with Lemma~\ref{lem:navan06}. 
	
	In light of our control of all the random events, the desired result follows.
\end{proof}

\section{Additional numerical comparisons}\label{sec:addnum}
{\rev We report here additional simulations for sizes of the residual bootstrap (RB) procedure \citep{freedman1981bootstrapping}, residual randomization (RR) procedure \citep{toulis2019invariant} and the desparsified Lasso coefficient test as implemented in the \texttt{hdi} R package (HDI) \citep{dezeure2015high} at nominal levels of $1\%$ and {\revone $5\%$}. As can be seen from Table~\ref{Tab:AdditionalSizeSimulation}, all the methods are above the nominal size level in the majority of the simulation settings considered here, especially when the design or the noise is heavy-tailed.}

\begin{table}[t!]
	\centering
	\begin{tabular}{cccc|cc|cc|cc}
		\hline\hline
		& & & & \multicolumn{2}{c}{RB} & \multicolumn{2}{c}{RR} & \multicolumn{2}{c}{HDI} \\
		\hline
		$n$ & $p$ & $\bX$ & noise & 1\% & 5\% & 1\% & 5\% & 1\% & 5\% \\
		\hline
		$300$ & $100$ & $\mathcal{G}$ & $\mathcal{G}$ & $3.36$ & $10.99$ & $3.72$ & $11.1$ & $7.33$ & $17.04$\\
		$300$ & $100$ & $\mathcal{G}$ & $t_1$ & $2.06$ & $6.33$ & $1.97$ & $6.64$ & $1.86$ & $3.27$\\
		$300$ & $100$ & $\mathcal{G}$ & $t_2$ & $3.07$ & $9.68$ & $3.14$ & $9.89$ & $4.32$ & $10.89$\\
		$300$ & $100$ & $t_1$ & $\mathcal{G}$ & $3.64$ & $11.44$ & $3.74$ & $11.16$ & $70.33$ & $76.08$\\
		$300$ & $100$ & $t_1$ & $t_1$ & $2.05$ & $6.91$ & $2.04$ & $6.88$ & $47.06$ & $55.56$\\
		$300$ & $100$ & $t_1$ & $t_2$ & $2.74$ & $9.91$ & $2.83$ & $9.81$ & $64.15$ & $71.65$\\
		$600$ & $100$ & $\mathcal{G}$ & $\mathcal{G}$ & $1.93$ & $7.46$ & $1.89$ & $7.41$ & $5.36$ & $13.93$\\
		$600$ & $100$ & $\mathcal{G}$ & $t_1$ & $0.99$ & $4.09$ & $1.3$ & $4.59$ & $1.8$ & $2.6$\\
		$600$ & $100$ & $\mathcal{G}$ & $t_2$ & $1.73$ & $6.72$ & $1.65$ & $6.62$ & $4.38$ & $10.81$\\
		$600$ & $100$ & $t_1$ & $\mathcal{G}$ & $1.93$ & $7.46$ & $1.91$ & $7.38$ & $77.05$ & $82.15$\\
		$600$ & $100$ & $t_1$ & $t_1$ & $1.37$ & $5.33$ & $1.29$ & $5.04$ & $51.2$ & $60.86$\\
		$600$ & $100$ & $t_1$ & $t_2$ & $1.66$ & $6.96$ & $1.58$ & $6.82$ & $71.18$ & $77.76$\\
		$600$ & $200$ & $\mathcal{G}$ & $\mathcal{G}$ & $3.55$ & $10.9$ & $3.51$ & $10.96$ & $6.7$ & $17.7$\\
		$600$ & $200$ & $\mathcal{G}$ & $t_1$ & $1.46$ & $5.79$ & $1.75$ & $6.08$ & $1.95$ & $2.6$\\
		$600$ & $200$ & $\mathcal{G}$ & $t_2$ & $2.53$ & $9.05$ & $2.96$ & $9.77$ & $4.33$ & $12.25$\\
		$600$ & $200$ & $t_1$ & $\mathcal{G}$ & $3.92$ & $11.44$ & $3.74$ & $11.31$ & $78.18$ & $82.82$\\
		$600$ & $200$ & $t_1$ & $t_1$ & $1.68$ & $6.19$ & $1.72$ & $6.19$ & $56.9$ & $63.4$\\
		$600$ & $200$ & $t_1$ & $t_2$ & $2.55$ & $9.56$ & $2.46$ & $9.42$ & $74.05$ & $80.56$\\
		\hline\hline
	\end{tabular}
	\caption{\label{Tab:AdditionalSizeSimulation} \rev Percentage of rejections of various tests under the null, estimated over 100000 Monte Carlo repetitions, for various noise distributions at nominal levels of $\alpha=1\%$ and {\revone $\alpha=5\%$}. This table supplements Table~\ref{Tab:SizeControl} in the main text and the same data generation mechanism is used. Percentage signs are omitted.}
\end{table}

{\rev We have focused primarily on size controls at $\alpha=1\%$ and {\revone $\alpha=5\%$} in the main text. In many applications, coefficient tests are conducted multiple times, which necessitates consideration of test size validity at small nominal levels due to multiple testing corrections. In Table~\ref{Tab:Size5pc}, we investigate the estimated sizes of all tests considered in our numerical simulations at the {\revone 0.5\%} nominal level. {\revone We observe that at this smaller nominal level, size validity issues of ANOVA and naive RPT becomes more pronounced, while $\text{RPT}_{\text{em}}$ and RPT still has valid size control.}
	
	\begin{table}[t!]
		\centering
		\begin{tabular}{ccccccccccccccc}
			\hline\hline
			$n$ & $p$ & $\boldsymbol{X}$ & noise & $\text{RPT}_{\text{em}}$ & RPT & ANOVA & Naive & DR & FL & CRT & RB & RR & HDI \\
			\hline
			$300$ & $100$ & $\mathcal{G}$ & $\mathcal{G}$ & $0$ & $0$ & $0.5$ & $0.5$ & $0.5$ & $0.5$ & $0$ & $2.1$ & $2.4$ & $5.5$\\
			$300$ & $100$ & $\mathcal{G}$ & $t_1$ & $0.1$ & $0$ & $1.6$ & $1.2$ & $0.4$ & $0.8$ & $1.7$ & $1.4$ & $1.2$ & $1.6$\\
			$300$ & $100$ & $\mathcal{G}$ & $t_2$ & $0$ & $0$ & $1.1$ & $0.9$ & $0.3$ & $0.6$ & $0.4$ & $1.8$ & $1.9$ & $2.8$\\
			$300$ & $100$ & $t_1$ & $\mathcal{G}$ & $0$ & $0$ & $0.5$ & $0.5$ & $2.2$ & $0.5$ & $0$ & $2.1$ & $2.3$ & $69.3$\\
			$300$ & $100$ & $t_1$ & $t_1$ & $0$ & $0$ & $2.1$ & $1.1$ & $0.7$ & $0.7$ & $0.3$ & $1.4$ & $1.3$ & $44.6$\\
			$300$ & $100$ & $t_1$ & $t_2$ & $0$ & $0$ & $1.3$ & $0.9$ & $1.5$ & $0.5$ & $0$ & $1.5$ & $1.7$ & $61.9$\\
			$600$ & $100$ & $\mathcal{G}$ & $\mathcal{G}$ & $0.1$ & $0$ & $0.5$ & $0.5$ & $0.5$ & $0.5$ & $0$ & $1.1$ & $1.1$ & $4.2$\\
			$600$ & $100$ & $\mathcal{G}$ & $t_1$ & $0.4$ & $0.3$ & $1.4$ & $0.8$ & $0.5$ & $0.6$ & $1.5$ & $0.6$ & $0.7$ & $1.6$\\
			$600$ & $100$ & $\mathcal{G}$ & $t_2$ & $0.3$ & $0.1$ & $1.2$ & $0.8$ & $0.3$ & $0.6$ & $0.4$ & $0.9$ & $0.9$ & $3.1$\\
			$600$ & $100$ & $t_1$ & $\mathcal{G}$ & $0.1$ & $0$ & $0.5$ & $0.5$ & $2.7$ & $0.5$ & $0$ & $1.1$ & $1$ & $76$\\
			$600$ & $100$ & $t_1$ & $t_1$ & $0$ & $0$ & $1.7$ & $0.6$ & $0.7$ & $0.5$ & $0.2$ & $0.8$ & $0.7$ & $48.5$\\
			$600$ & $100$ & $t_1$ & $t_2$ & $0$ & $0$ & $1.3$ & $0.6$ & $2$ & $0.5$ & $0$ & $0.9$ & $0.8$ & $69.5$\\
			$600$ & $200$ & $\mathcal{G}$ & $\mathcal{G}$ & $0$ & $0$ & $0.5$ & $0.5$ & $0.5$ & $0.5$ & $0$ & $2.1$ & $2.3$ & $5.3$\\
			$600$ & $200$ & $\mathcal{G}$ & $t_1$ & $0.3$ & $0.2$ & $1.2$ & $0.9$ & $0.4$ & $0.7$ & $1.3$ & $0.9$ & $1$ & $1.8$\\
			$600$ & $200$ & $\mathcal{G}$ & $t_2$ & $0.1$ & $0$ & $1$ & $0.9$ & $0.3$ & $0.7$ & $0.3$ & $1.3$ & $1.8$ & $2.6$\\
			$600$ & $200$ & $t_1$ & $\mathcal{G}$ & $0$ & $0$ & $0.5$ & $0.5$ & $2.3$ & $0.5$ & $0$ & $2.6$ & $2.3$ & $76.4$\\
			$600$ & $200$ & $t_1$ & $t_1$ & $0$ & $0$ & $1.7$ & $0.9$ & $0.6$ & $0.6$ & $0.2$ & $1.1$ & $1$ & $54.8$\\
			$600$ & $200$ & $t_1$ & $t_2$ & $0$ & $0$ & $1.2$ & $0.8$ & $1.6$ & $0.5$ & $0$ & $1.5$ & $1.4$ & $72$\\
			\hline\hline
		\end{tabular}
		\caption{\label{Tab:Size5pc} \rev Percentage of rejections of various tests under the null, estimated over 100000 Monte Carlo repetitions, for various noise distributions at nominal levels of {\revone $\alpha=0.5\%$}. The data generation mechanism is the same as in Table~\ref{Tab:SizeControl}. Percentage signs are omitted.}
	\end{table}

	{\rev We also report in Figure~\ref{Fig:AdditionalPowerSimulation} a power comparison of RPT with the additional methods mentioned in Table~\ref{Tab:AdditionalSizeSimulation}. 
		RB, RR and HDI procedures exhibit better power than RPT and $\text{RPT}_{\text{EM}}$ in most of the simulation setups. However, this should be viewed in the context of their above nominal size under the null as reported in Table~\ref{Tab:AdditionalSizeSimulation}. 
	}
	
	{
		\revone
		Finally, we investigate the robustness of the power of RPT to model misspecification. Specifically, in our theoretical analysis, we have worked under Assumption~\ref{as:iid} and its variants. However, as can be seen in Figure~\ref{fig:misspec}, the power curves of RPT is almost unaffected even if the independence and linearity assumptions in Assumption~\ref{as:iid} are violated. Specifically, in this numerical experiment, we generate $\bv=(\varepsilon_1,\ldots,\varepsilon_n)^\top$ either independent of $\be=(e_1,\ldots,e_n)^\top$ or allow them to be correlated through the relationship $\varepsilon_i \mid e_i \sim t_1\mathbbm{1}_{\{e_i\geq 0\}} + 2t_1\mathbbm{1}_{\{e_i < 0\}}$. We also allow $\bZ=(Z_1,\ldots,Z_n)$ to have nonlinear dependence on $\bX=(X_1,\ldots,X_n)$ in the form of $Z_i = f(X_i \beta^Z) + e_i$, where $f: t\mapsto 1/(1+e^{-t})$ is the sigmoid function.
	}
	\begin{figure}[htbp]
		\centering
		\subfigure[Gaussian design, Gaussian noise]{\includegraphics[width=0.45\textwidth]{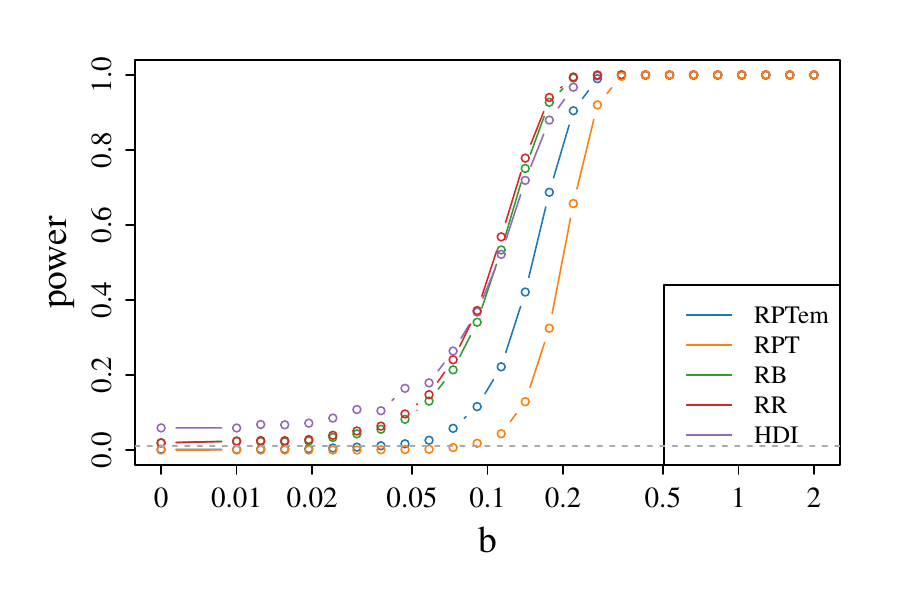}} 
		\subfigure[Gaussian design, $t_1$ noise]{\includegraphics[width=0.45\textwidth]{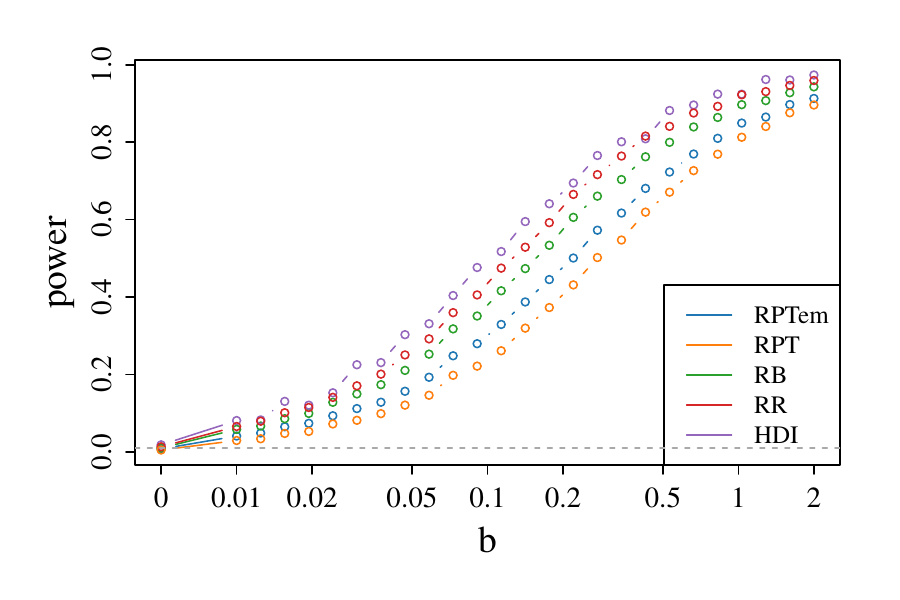}} \\ \vspace*{-0.2cm}
		\subfigure[Gaussian design, $t_2$ noise]{\includegraphics[width=0.45\textwidth]{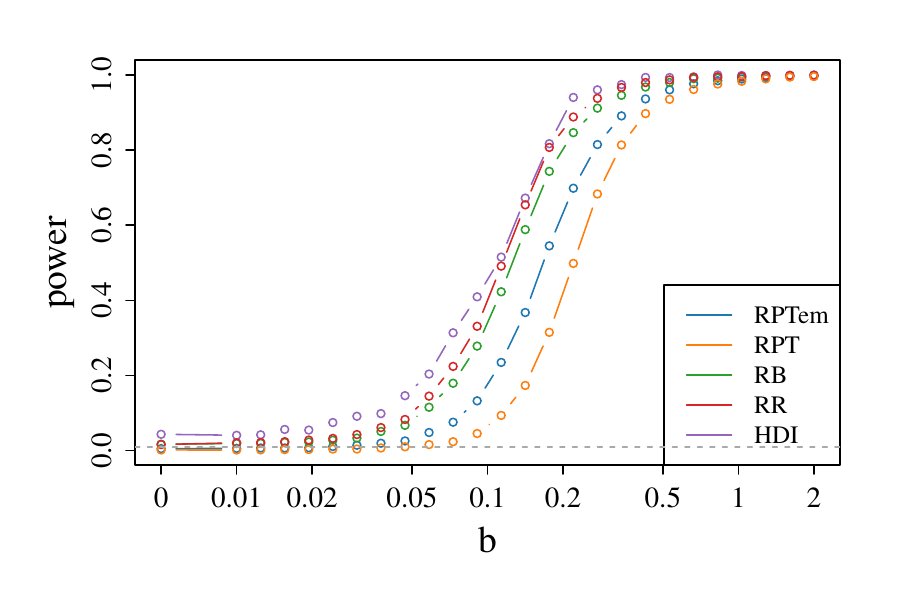}} 
		\subfigure[$t_1$ design, Gaussian noise]{\includegraphics[width=0.45\textwidth]{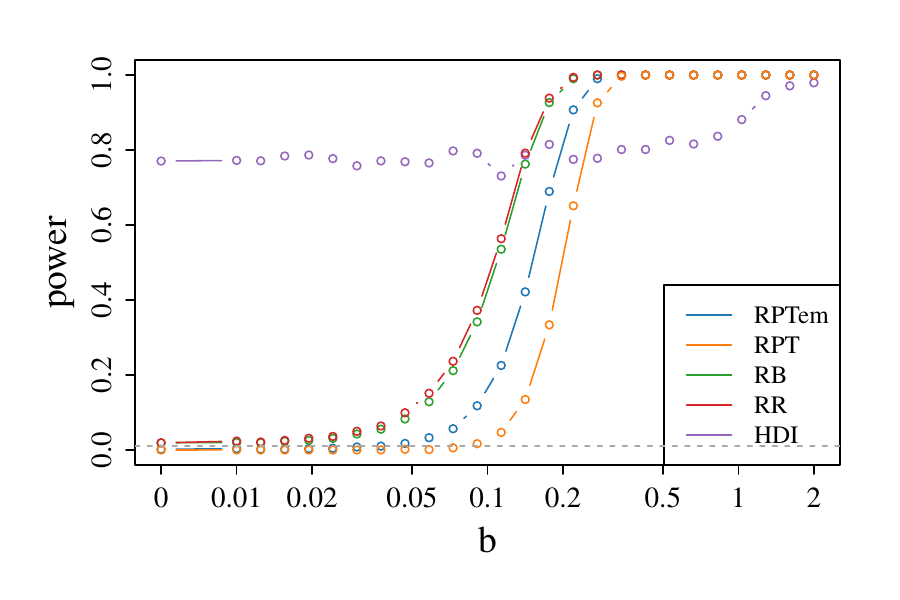}} \\ \vspace*{-0.2cm}
		\subfigure[$t_1$ design, $t_1$ noise]{\includegraphics[width=0.45\textwidth]{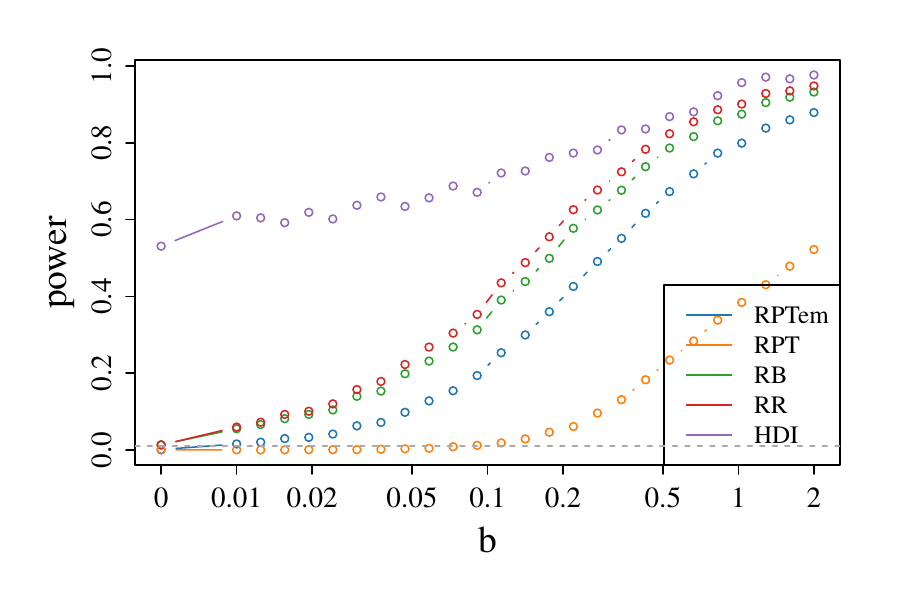}}
		\subfigure[$t_1$ design, $t_2$ noise]{\includegraphics[width=0.45\textwidth]{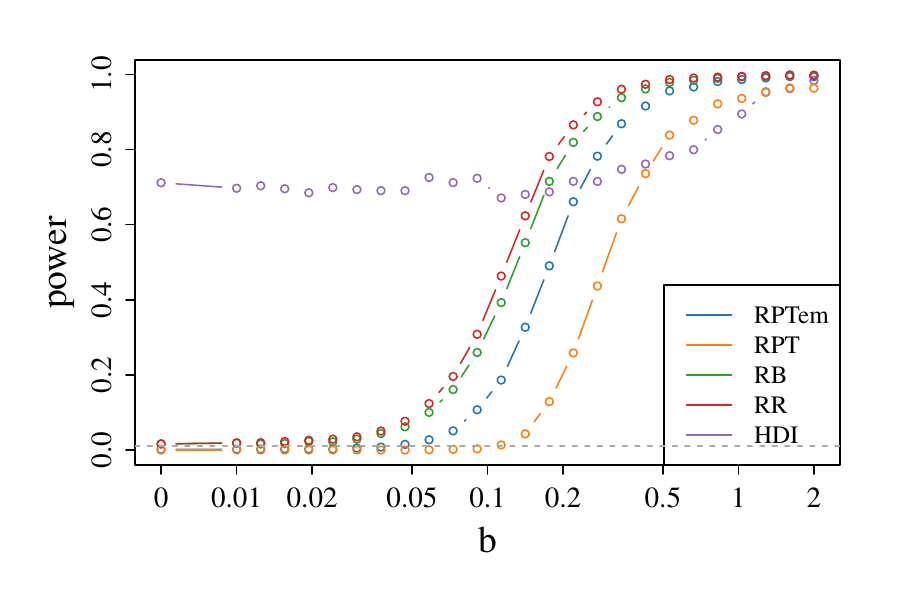}}
		\caption{\rev Power (proportion of rejections) with nominal level $\alpha = 0.01$ (represented by the horizontal dashed line) over 10000 replicates for $b = 0$ or on a logarithmic grid between 0.01 and 2. Here $\bX, \bv$ and $\be$ are generated according to various distribution types prescribed in the caption of each figure.}
		\label{Fig:AdditionalPowerSimulation}
	\end{figure}
	
	\begin{figure}[htbp]
		\centering
		\subfigure[independent noise, linear relation]{\includegraphics[width=0.45\textwidth]{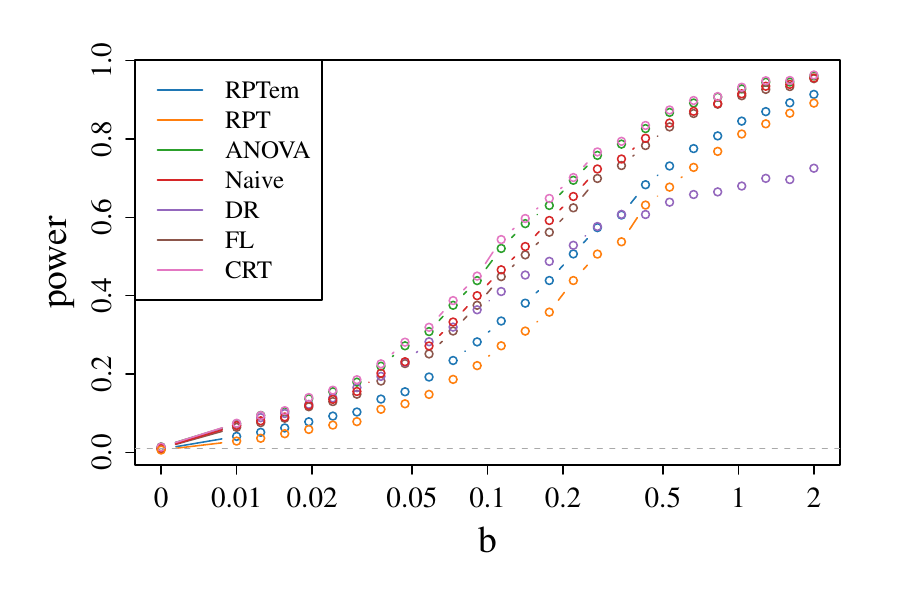}} 
		\subfigure[dependent noise, linear relation]{\includegraphics[width=0.45\textwidth]{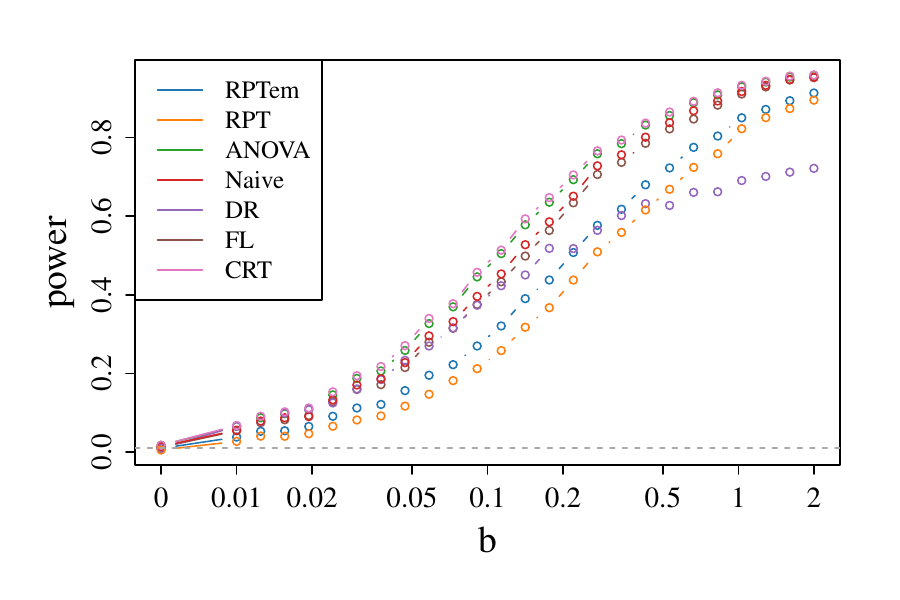}} \\ \vspace*{-0.2cm}
		\subfigure[independent noise, nonlinear relation]{\includegraphics[width=0.45\textwidth]{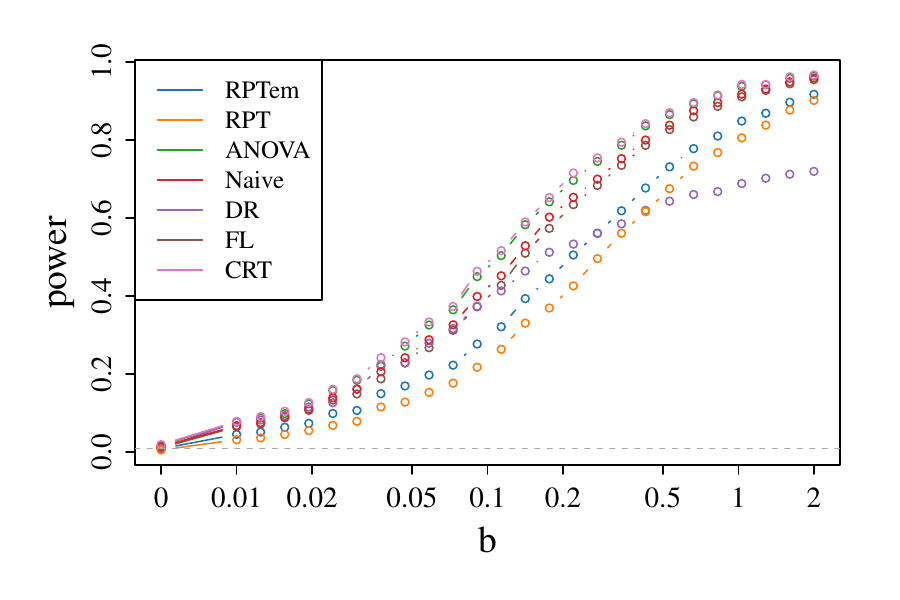}} 
		\subfigure[dependent noise, nonlinear relation]{\includegraphics[width=0.45\textwidth]{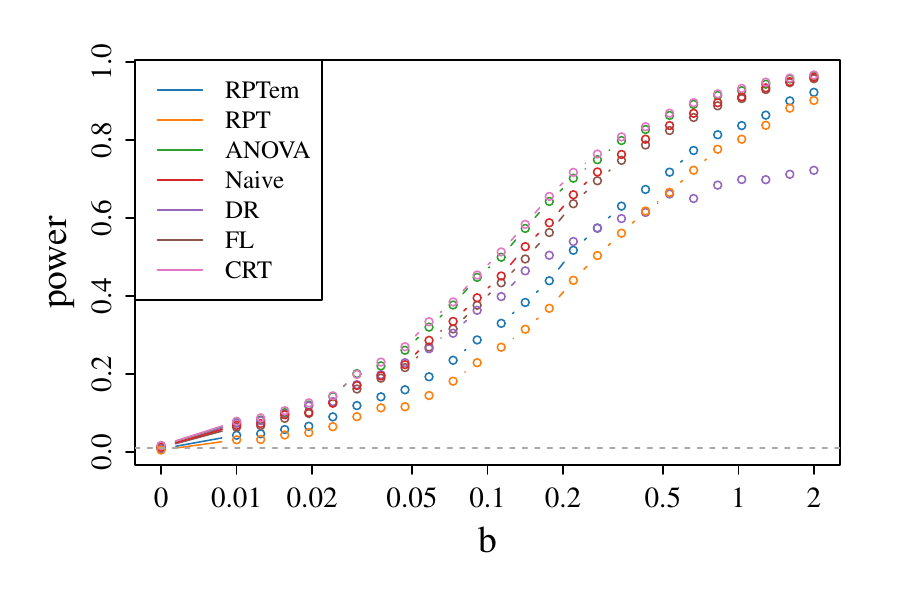}} 
		\caption{\rev Power (proportion of rejections) with nominal level $\alpha = 0.01$ (represented by the horizontal dashed line) over 10000 replicates for $b = 0$ or on a logarithmic grid between 0.01 and 2. Here $\bX$ has independent $\mathcal{N}(0,1)$ entries, $\be$ has independent $t_1$ entries, $\bv$ either has independent $t_1$ entries, independent of $\be$ (independent noise), or a mixture of $t_1$ and $2t_1$ entries dependent on the sign of $\be$ (dependent noise). The covariate $\bZ$ is modelled either using~\eqref{Eq:Model2} (linear relation) or generated through the nonlinear model described at the end of Section~\ref{sec:addnum}.}
		\label{fig:misspec}
	\end{figure}

\end{document}